\documentclass[final,leqno]{siamltex}

\usepackage{amssymb,amsbsy,amsmath,amsfonts,amssymb,amscd,bm}
\usepackage{mathrsfs}
\usepackage{epsfig, graphicx}
\usepackage{graphics,epsfig}

\newtheorem{remark}[theorem]{Remark}

\newcommand{\dx}{\partial_x}

\newcommand{\dxn}{\nabla_x}

\newcommand{\dkn}{\nabla_k}

\newcommand{\colQ}{\mathcal{Q}}
\newcommand{\energy}{\mathcal{E}}
\newcommand{\half}{\frac{1}{2}}
\newcommand{\Mel}{M_\text{el}}
\newcommand{\Qel}{\mathcal{Q}_\text{el}}

\newcommand{\eps}{\epsilon}

\newcommand{\Dt}{\Delta t}
\newcommand{\Dk}{\Delta k}
\newcommand{\Dx}{\Delta x}

\newcommand{\vel}{\sqrt{\theta}}

\title{ 
An Asymptotic-Preserving Scheme for the Semiconductor Boltzmann Equation toward the Energy-Transport Limit
}
\author{Jingwei Hu\thanks{Institute for Computational Engineering and Sciences (ICES), The University of Texas at Austin, 201 East 24th St, Stop C0200, Austin, TX 78712 (hu@ices.utexas.edu).}
 \and Li Wang\thanks{Department of Mathematics, University of California, Los Angeles, 520 Portola Plaza, Los Angeles, CA 90095 (liwang@math.ucla.edu).}}

\begin{document}
\maketitle

\begin{abstract}
We design an asymptotic-preserving scheme for the semiconductor Boltzmann equation which leads to an energy-transport system for electron mass and internal energy as mean free path goes to zero. To overcome the stiffness induced by the convection terms, we adopt an even-odd decomposition to formulate the equation into a diffusive relaxation system. New difficulty arises in the two-scale stiff collision terms, whereas the simple BGK penalization does not work well to drive the solution to the correct limit. We propose a clever variant of it by introducing a threshold on the stiffer collision term such that the evolution of the solution resembles a Hilbert expansion at the continuous level. Formal asymptotic analysis and numerical results are presented to illustrate the efficiency and accuracy of the new scheme.
\end{abstract}



\pagestyle{myheadings}
\thispagestyle{plain}
\markboth{J. HU and L. WANG}{AP SCHEME FOR SEMICONDUCTOR BOLTZMANN EQUATION}


\section{Introduction}

The semiconductor Boltzmann equation describes the transport of charge carriers in semiconductor devices. It is derived following a statistical approach by incorporating the quantum mechanical effects semiclassically, thus provides accurate description of the physics at the kinetic level \cite{MRS, Degond, Jungel}. A dimensionless form of this equation usually contains small parameters such as the mean free path or time. Besides the high dimensionality of the probability distribution function, the presence of these small parameters poses tremendous computational challenge since one has to numerically resolve the small scales.

To save the computational cost, in the past decades, various macroscopic models were derived from the Boltzmann equation based on assumptions of different dominating effects. One of the well accepted model, for example, is the drift-diffusion equation which consists of mass continuity equation for electrons (or holes) \cite{Poupaud91, GP}. Ideally, if the parameters in the kinetic equation are uniformly small in the entire domain of interest, then the macroscopic models suffice to describe the physical phenomena, and it is more efficient to just solve them \cite{JP97, DJP00, Ringhofer01}. In practical applications, however, the validity of these models may break down in part of the domain (parameters are not small anymore), and one is forced to solve the kinetic equation which contains mesoscopic information \cite{CGMS06, CGMS09}. A natural solution to this situation is a domain decomposition approach \cite{BTPQ94, KNS00}, but finding the interface condition connecting the kinetic and macroscopic equations is a highly non-trivial task. Another line of research is to find a unified scheme for the kinetic equation such that when the small parameter goes to zero, it automatically becomes a macroscopic solver. This designing concept leads to the asymptotic-preserving (AP) scheme \cite{Jin_Rev}, which was first introduced by S. Jin for transport equations in diffusive regimes \cite{Jin_AP}. In the semiconductor framework, an initial effort toward the AP schemes was proposed in \cite{JP00} for the linear Boltzmann equation with an anisotropic collision term, whose computation was further improved in \cite{Deng}. Recently a higher order scheme with a less strict stability condition was constructed in the sense that the parabolic CFL constraint is relaxed to a hyperbolic one \cite{DPR13}. All these works consider a linear collision operator with smooth kernel which uniquely defines an equilibrium state. As a result, the corresponding macroscopic equation is in the form of a drift-diffusion equation. Although this equation gives satisfactory simulation results for semiconductor devices on the micrometer scale, it is not able to capture the hot-electron effects in submicron devices \cite{Jungel10}. High field scaling deals with this problem to some extent, but it only works for the situation where the field effect is strong enough to balance the collision \cite{JW13, CL}.

In this work, we are interested in a more realistic semiconductor Boltzmann equation \cite{AD96, DLS04}. By considering the elastic collision as dominant, and electron-electron correlation as sub-dominant effects, one can pass on the asymptotic limit to obtain an energy-transport (ET) model. It consists of a system of conservation laws for mass and internal energy of charge carriers with fluxes computed through a constitutive relation \cite{JM73}. To design an AP scheme for such kinetic equation, we face two-fold challenge: 1. the convection terms are stiff; 2. two stiff collision terms live on different scales. The convection terms can be treated by an even-odd decomposition as in \cite{JPT00, JP00}. For the collision terms, due to their complicated forms, we choose to penalize them with a suitable BGK operator \cite{FJ10}. However, unlike the usual collision operator with smoothed kernel, the leading elastic operator has non unique null space (the kernel is a Delta function). Only when the electron-electron operator in next level takes into effect, the solution can be eventually driven to a fixed Fermi-Dirac distribution. A closer examination of the asymptotic behavior of the solution reveals that the penalization should be performed wisely, otherwise it won't capture the correct limit. To this end, we propose a {\it thresholded penalization} scheme. Simply speaking, when the threshold is satisfied, we turn off the leading order mechanism and move to the next order, which in some sense resembles the Hilbert expansion in the continuous case. We will show that this new scheme, under certain assumptions, satisfies the following four properties ($\alpha$ denotes the small parameter; $\Dt$, $\Dx$, and $\Dk$ are the time step and mesh size in spatial and wave vector (momentum) domain):
\begin{enumerate}
\item For fixed $\alpha$, it is consistent to the Boltzmann equation when $\Delta t, \Delta x, \Delta k\rightarrow 0$.
\item For fixed $\Delta t, \Delta x, \Delta k$, it becomes a discretization to the limiting ET system when $\alpha\rightarrow 0$.
\item It is uniformly stable for a wide range of $\alpha$, from $\alpha=O(1)$ to $\alpha\ll 1$.
\item Implicit terms can be implemented explicitly (free of Newton-type solvers).
\end{enumerate}
An important ingredient in this AP scheme is the accurate numerical solvers for the collision operators. Since the electron-electron operator falls into a special case of the quantum Boltzmann operator, we adopt the fast spectral method developed in \cite{HY12}. For the elastic collision, it is desirable to evaluate it in the same spectral framework but the direct computation would be very expensive. We propose a new fast method by exploring the low-rank structure in the coefficient matrix.


The rest of the paper is organized as follows. In the next section we give a brief review of the scalings of the semiconductor Boltzmann equation and the derivation of the ET model through a systematic approach. Section 3 is devoted to a detailed description of our schemes. We will present it in a pedagogical way that the spatially homogeneous case is considered first with an emphasis on the two-scale stiff collision terms, and then embrace the spatial dependence to treat the full problem. In either case, the asymptotic property of the numerical solution is carefully analyzed. The spectral methods for computing the collision operators are gathered at the end. In Section 4 we give several numerical examples including the simulation of a 1-D $n^+-n-n^+$ silicon diode to illustrate the efficiency, accuracy, and asymptotic properties of the scheme. Finally, the paper is concluded in Section 5.


\section{The semiconductor Boltzmann equation and its energy-transport limit}

The Boltzmann transport equation that describes the evolution of electrons in the conduction band of a semiconductor reads \cite{MRS, Degond, Jungel}
\begin{equation} \label{eqn: semiB1}
\partial_t f + \frac{1}{\hbar} \dkn \varepsilon(k) \cdot \dxn f + \frac{q}{\hbar} \dxn V(x,t)\cdot \dkn f = \colQ(f), \quad x\in \Omega \subset \mathbb{R}^d, \ \  k\in B \subset \mathbb{R}^d, \ \ d=2,3,
\end{equation}
where $f(x,k,t)$ is the electron distribution function of position $x$, wave vector $k$, and time $t$. $\hbar$ is the reduced Planck constant, and $q$ is the positive elementary charge. The first Brillouin zone $B$ is the primitive cell in the reciprocal lattice of the crystal. For simplicity, we will restrict to the parabolic band approximation, where $B$ can be extended to the whole space $B=\mathbb{R}^d$, and the energy-band diagram $\varepsilon(k)$ is given by
\begin{equation*}
\varepsilon(k) =  \frac{\hbar^2}{2m^*}|k|^2,
\end{equation*}
where $m^*$ is the effective mass of electrons. 

In principle, the electrostatic potential $V(x,t)$ is produced self-consistently by the electron density with a fixed ion background of doping profile $h(x)$ through the Poisson equation:
\begin{equation} \label{eqn: Poisson1}
\epsilon_0\dxn (\epsilon_r(x) \dxn V(x,t)) = q(\rho(x,t) - h(x)), 
\end{equation}
where 
\begin{equation*}
\rho(x,t):=\int_{\mathbb{R}^d}f(x,k,t)\, \frac{g}{(2\pi)^d}\,dk 
\end{equation*}
is the electron density (the spin degeneracy $g=2s+1$, with $s=1/2$ being the spin of electrons). $\epsilon_0$ and $\epsilon_r(x)$ are the vacuum and the relative material permittivities. The doping profile
$h(x)$ takes into account the impurities due to acceptor and donor ions in the semiconductor device. 

The collision operator $\colQ$ explains three different effects:
\begin{equation*} 
\colQ = \colQ_{\text{imp}} + \colQ_{\text{ph}} + \colQ_{\text{ee}},
\end{equation*}
where $\colQ_{\text{imp}}$ and $\colQ_{\text{ph}}$ account for the interactions between electrons and the lattice defects caused by ionized impurities and crystal vibrations (also called phonons); $\colQ_{\text{ee}}$ describes the correlations between electrons themselves. Specifically,
\begin{eqnarray*}
&&\colQ_{\text{imp}} (f)(k) =\int_{\mathbb{R}^d} \phi_{\text{imp}}(k,k') \delta (\varepsilon' - \varepsilon)(f'-f)\,dk',
\\ && \colQ_{\text{ph}}(f)(k) = \int_{\mathbb{R}^d} \phi_{\text{ph}}(k,k')\left\{ \left[ (N_{\text{ph}}+1) \delta \left(\varepsilon - \varepsilon' +\varepsilon_{\text{ph}} \right)  + N_{\text{ph}}  \delta \left(\varepsilon - \varepsilon' - \varepsilon_{\text{ph}}\right)\right]f'(1- f) \right. \nonumber
\\ &&  \hspace{0.8in} \left.- \left[ (N_{\text{ph}}+1) \delta \left(\varepsilon' - \varepsilon + \varepsilon_{\text{ph}} \right)  + N_{\text{ph}}  \delta \left(\varepsilon' - \varepsilon - \varepsilon_{\text{ph}} \right)\right] f(1-  f')  \right\}\, dk',
\\ && \colQ_{\text{ee}}(f)(k) = \int_{\mathbb{R}^{3d}} \phi_{\text{ee}}(k,k_1,k',k_1') \delta(\varepsilon' + \varepsilon_1' -\varepsilon - \varepsilon_1 )\delta(k' + k_1' -k - k_1 ) \nonumber
\\  &&  \hspace{0.8in} \times\Big[   f'f_1'(1-f)(1-f_1) - ff_1(1-f')(1-f_1')  \Big] \, dk_1 dk' dk_1',
\end{eqnarray*}
where $\delta$ is the Dirac measure, $\varepsilon$, $\varepsilon'$, $f$, $f'$, $f_1$, $f_1'$ are short notations for $\varepsilon(k)$, $\varepsilon(k')$, $f(x,k,t)$, $f(x,k',t)$, $f(x,k_1,t)$, and $f(x,k_1',t)$ respectively. $\varepsilon_{\text{ph}}$ is the phonon energy, and $N_{\text{ph}}$ is the phonon occupation number:
\begin{equation*}
N_{\text{ph}}= \frac{1}{e^{\frac{\varepsilon_{\text{ph}}}{k_B T_L}}-1},
\end{equation*}
where $k_B$ is the Boltzmann constant and $T_L$ is the lattice temperature. The scattering matrices $\phi_{\text{imp}}(k, k')$ and $\phi_{\text{ph}}(k,k')$ are symmetric in $k$ and $k'$:
\begin{equation*}
\phi_{\text{imp}}(k, k')=\phi_{\text{imp}}(k', k), \quad \phi_{\text{ph}}(k,k')=\phi_{\text{ph}}(k',k);
\end{equation*}
$\phi_{\text{ee}}(k,k_1, k', k_1')$ is symmetric pair-wisely for four variables:
\begin{equation*}
\phi_{\text{ee}}(k,k_1, k', k_1')=\phi_{\text{ee}}(k_1,k, k', k_1')=\phi_{\text{ee}}(k',k_1', k, k_1).
\end{equation*}
They are all determined by the underlying interaction laws.


\subsection{Nondimensionalization of the Boltzmann equation}

In order to nondimensionalize the Boltzmann equation (\ref{eqn: semiB1}), we introduce the following typical values:
\begin{eqnarray*}
\rho_0 : && \textrm{ typical density;}
\\ \varepsilon_0 = q V_0: &&  \textrm{ typical kinetic energy of electrons, $V_0$ is the applied bias;} 
\\  k_0:  && \textrm{ typical norm of wave vector $k$ such that $\varepsilon(k_0) =\varepsilon_0/2$;} 
\\  f_0 = \frac{(2\pi)^d\rho_0}{g k_0^d}=\eta: && \textrm{ typical distribution function scale;}
\\t_0 : && \textrm{ typical time scale;}
\\v_0 = \frac{\varepsilon_0}{\hbar k_0}:  && \textrm{ typical velocity scale;}
\\x_0 = t_0 v_0:  && \textrm{ typical length scale;}
\\ \phi_{\text{imp},0}, \phi_{\text{ph},0}, \phi_{\text{ee},0}: && \textrm{ typical values of transition rates $\phi_{\text{imp}}(k,k')$, $\phi_{\text{ph}}(k,k')$, $\phi_{\text{ee}}(k,k_1,k', k_1')$,}
\end{eqnarray*}
and define also the dimensionless parameters:
\begin{equation*}
 \alpha^2 = \frac{\varepsilon_{\text{ph}}}{\varepsilon_0}, \  \gamma^2=\frac{k_BT_L}{\varepsilon_0}, \  \nu_{\text{imp}} = \frac{\phi_{\text{imp},0} k_0^d t_0}{\varepsilon_0}, \   \nu_{\text{ph}} =  \frac{\phi_{\text{ph},0} k_0^d t_0}{\varepsilon_0}, \  \nu_{\text{ee}} = \frac{\phi_{\text{ee},0} k_0^d t_0}{\varepsilon_0}\frac{(2\pi)^d\rho_0}{g}.
\end{equation*}
After performing a change of variables as 
\begin{equation*}
\tilde{f} =\frac{f}{f_0}, \quad \tilde{t}=\frac{t}{t_0}, \quad \tilde{x}= \frac{x}{x_0}, \quad \tilde{k} = \frac{k}{k_0}, \quad \tilde{\varepsilon} =\frac{\varepsilon}{\varepsilon_0}, \quad  \tilde{V}=\frac{V}{V_0}, \quad \tilde{\phi}_{\bullet} = \frac{\phi_{\bullet}}{ {\phi}_{\bullet,0}}, 
\end{equation*}
(\ref{eqn: semiB1}) becomes
\begin{equation*}
\partial_t  f + \dkn \varepsilon \cdot \dxn f  + \dxn V \cdot \dkn  f =  \colQ_{\text{imp}}(f)
 + \colQ_{\text{ph}}(f)+ \colQ_{\text{ee}}(f),
\end{equation*}
where we have dropped the tildes without ambiguity. The collision operators take the following dimensionless form
\begin{eqnarray}
&&\colQ_{\text{imp}} (f)(k) = \int_{\mathbb{R}^d} \Phi_{\text{imp}}(k,k') \delta (\varepsilon'- \varepsilon)(f'-f) \,dk',  \nonumber
\\ &&\colQ_{\text{ph}}(f)(k) =  \int_{\mathbb{R}^d}  \Phi_{\text{ph}}(k,k')\left\{ \left[ (N_{\text{ph}}+1) \delta \left(\varepsilon \!-\! \varepsilon' \!+\! \alpha^2 \right)  + N_{\text{ph}}  \delta \left(\varepsilon \!-\! \varepsilon' \!-\! \alpha^2 \right)\right]f'(1\!- \eta f) \right. \nonumber
\\ &&  \hspace{0.8in} \left.- \left[ (N_{\text{ph}}+1) \delta \left(\varepsilon' - \varepsilon + \alpha^2 \right)  + N_{\text{ph}}  \delta \left(\varepsilon' - \varepsilon - \alpha^2 \right)\right] f(1- \eta f')  \right\} \,dk',  \nonumber
\\ && \colQ_{\text{ee}}(f)(k) =\int_{\mathbb{R}^{3d}} \Phi_{\text{ee}}(k,k_1,k',k_1') \delta(\varepsilon' + \varepsilon_1' -\varepsilon - \varepsilon_1  )\delta (k' + k_1' -k - k_1  ) \nonumber
\\  &&  \hspace{0.8in}\times\Big[   f'f_1'(1-\eta f)(1- \eta f_1) - ff_1(1- \eta f')(1 \!-\! \eta f_1')  \Big]\,  dk_1 dk' dk_1', \label{Qee}
\end{eqnarray}
where $\Phi_{\text{imp}} =  \nu_{\text{imp}} \phi_{\text{imp}}$, $\Phi_{\text{ph}} =  \nu_{\text{ph}} \phi_{\text{ph}}$, $\Phi_{\text{ee}} =  \nu_{\text{ee}} \phi_{\text{ee}}$, and 
\begin{equation*}
N_{\text{ph}}=\frac{1}{e^{\alpha^2/\gamma^2}-1}.
\end{equation*}
The energy-band diagram is now simply 
\begin{equation} \label{para}
\varepsilon(k) = \frac{1}{2}|k|^2.
\end{equation}
The Poisson equation (\ref{eqn: Poisson1}) becomes
\begin{equation} \label{eqn: Poisson}
C_0 \dxn (\epsilon_r(x) \dxn V(x,t)) = \rho(x,t) - h(x), 
\end{equation}
where $C_0=\frac{\eps_0V_0}{qx_0^2\rho_0}$ is the square of the scaled Debye length, and 
\begin{equation*} 
\rho(x,t)=\int_{\mathbb{R}^d} f(x,k,t)\,dk.
\end{equation*}


\subsection{Elastic approximation of the electron-phonon interactions}

We are interested in a high energy scale \cite{ADG96, AD96}, at which the relative energy gain or loss of electron energy during a phonon collision is very small, i.e.,
\begin{equation*}
\alpha \ll 1.
\end{equation*}
In addition, we assume that 
\begin{equation*}
\frac{\alpha}{\gamma}\sim O(1),
\end{equation*}
which means that at the high energy scale, the phonon energy $\varepsilon_{\text{ph}}$ and the lattice thermal energy $k_BT_L$ are considered as the same order of magnitude, and much smaller compared with the electron energy $\varepsilon_0$.

Treating $\alpha$ as small parameter, one can expand the electron-phonon collision operator $\colQ_{\text{ph}}$ as
\begin{equation*}
\colQ_{\text{ph}}(f)(k) =\int_{\mathbb{R}^d} (2N_{\text{ph} }+1) \Phi_{\text{ph}}(k, k')\delta (\varepsilon'  - \varepsilon)(f'-f) \,dk' + \alpha^2 \colQ_{\text{ph}}^{\text{inel}}(f)(k),
\end{equation*}
where the first term is the elastic approximation and the second term is the inelastic correction. Therefore, the total collision operator $\colQ$ can be recast as
\begin{equation*}
\colQ(f) = \colQ_{\text{el}}(f) + \colQ_{\text{ee}}(f) + \alpha^2\colQ_{\text{ph}}^{\text{inel}}(f), 
\end{equation*}
with
\begin{equation} \label{Qel}
\colQ_{\text{el}}(f)(k) = \int_{\mathbb{R}^d} \Phi_{\text{el}}(k,k') \delta \left(  \varepsilon' - \varepsilon  \right) \left( f' - f\right) \,dk',
\end{equation}
and
\begin{equation*}
\Phi_{\text{el}}(k,k')=  \Phi_{\text{imp}}(k,k') +(2N_{\text{ph}}+1) \Phi_{\text{ph}} (k,k').
\end{equation*}

Since $N_{\text{ph}}\sim O(1)$, it is reasonable to assume $\colQ_{\text{el}}$ and $\colQ_{\text{ph}}^{\text{inel}}$ are both of $O(1)$ (refer to \cite{Reggiani} for physical data). However, it is much delicate to estimate the electron-electron collision frequency as it depends on the distribution function itself. Following the discussion in \cite{DLS04}, we assign this term $O(\alpha)$. That is, the electron-electron interactions are not as strong as elastic collisions, yet their density is not small enough to be safely neglected.

The final form of the scaled Boltzmann equation is thus
\begin{equation} \label{scaledB}
\partial_t  f + \dkn \varepsilon \cdot \dxn f  + \dxn V \cdot \dkn  f =  \colQ_{\text{el}}(f) + \alpha \colQ_{\text{ee}}(f) + \alpha^2\colQ_{\text{ph}}^{\text{inel}}(f),
\end{equation}
where $\colQ_{\text{el}}$, $\colQ_{\text{ee}}$ are given by (\ref{Qel}) and (\ref{Qee}) respectively (the exact form of $\colQ_{\text{ph}}^{\text{inel}}$ will not be needed in the following discussion and is thereby omitted).


\subsection{Diffusive regime and the energy-transport limit}

To derive a macroscopic model, we consider the time and length scales in a diffusive regime: $t'= \alpha^2 t$, $x' = \alpha x $, then equation (\ref{scaledB}) rewrites as
\begin{equation} \label{kinetic1}
\partial_t f + \frac{1}{\alpha } \left(  \dkn \varepsilon \cdot \dxn f + \dxn V \cdot \dkn f  \right)  = \frac{1}{\alpha^2} \colQ_{\text{el}}(f) +  \frac{1}{\alpha} \colQ_{\text{ee}}(f) +\colQ_{\text{ph}}^{\text{inel}}(f),
\end{equation}
which is the main kinetic equation we are going to study for the rest of the paper. This subsection is devoted to a formal derivation of the asymptotic limit of (\ref{kinetic1}) as $\alpha\rightarrow 0$. Our approach, following that of \cite{DLS04}, is a combination of the Hilbert expansion and the moment method. To this end, we first list the required properties of the collision operators $\colQ_{\text{el}}$ and $\colQ_{\text{ee}}$. These will also be useful in designing numerical schemes.

\begin{proposition}\cite{AD96}  \label{prop: qel}
\begin{enumerate}
\item For any ``regular" test function $g(k)$,
\begin{equation*}
\int_{\mathbb{R}^d} \colQ_{\text{el}}(f)g\,dk=-\frac{1}{2}\int_{\mathbb{R}^{2d}}  \Phi_{\text{el}}(k,k') \delta \left(  \varepsilon' - \varepsilon  \right) \left( f' - f\right) (g'-g)\,dk dk'.
\end{equation*}
In particular, for any $g(\varepsilon(k))$, $$\int_{\mathbb{R}^d} \colQ_{\text{el}}(f)g(\varepsilon)\,dk=0.$$

\item $\colQ_{\text{el}}(f)$ is a self-adjoint, non-positive operator on $L^2(\mathbb{R}^d)$. 
\item The null space of $\colQ_{\text{el}}(f)$ is given by
\begin{equation*}
\mathcal{N}(\colQ_{\text{el}})=\{f(\varepsilon(k)),  \quad \forall f\}.
\end{equation*}
\item The orthogonal complement of $\mathcal{N}(\colQ_{\text{el}})$ is
\begin{equation*}
\mathcal{N}(\colQ_{\text{el}})^\perp= \{g(k) \   | \   \int_{S_\varepsilon} g(k)\,d N_\varepsilon(k)=0,  \quad \forall \varepsilon \in R \},
\end{equation*}
where the integral is defined through the ``coarea formula" \cite{Federer}: for any ``regular" functions $f$ and $\varepsilon(k): \mathbb{R}^d\rightarrow R$, it holds
\begin{equation} \label{coarea}
\int_{\mathbb{R}^d}f(k)\,dk=\int_{R} \left(\int_{S_{\varepsilon}} f(k) \,dN_\varepsilon(k)\right)\,d\varepsilon.
\end{equation}
Here $S_\varepsilon=\{k\in \mathbb{R}^d, \ \varepsilon(k)=\varepsilon\}$ denotes the surface of constant energy $\varepsilon$, $dN_\varepsilon(k)$ is the coarea measure, and $N(\varepsilon):=\int_{S_{\varepsilon}}\,dN_{\varepsilon}(k)$ is the energy density-of-states. Under the parabolic band approximation (\ref{para}), (\ref{coarea}) is just a spherical transformation:
\begin{equation*} 
\int_{\mathbb{R}^d}f(k)\,dk=\int_0^{\infty} \left( \int_{\mathbb{S}^{d-1}} f(|k|\sigma)|k|^{d-2}\,d \sigma\right)\,d\varepsilon,
\end{equation*}
and $N(\varepsilon)= \int_{\mathbb{S}^{d-1}}|k|^{d-2}\,d \sigma=m(\mathbb{S}^{d-1})|k|^{d-2}$.

\item The range of $\colQ_{\text{el}}$ is $\mathcal{R}(\colQ_{\text{el}})= \mathcal{N}(\colQ_{\text{el}})^\perp$. The operator is invertible as an operator from $\mathcal{N}(\colQ_{\text{el}})^\perp$ onto $\mathcal{R}(\colQ_{\text{el}})= \mathcal{N}(\colQ_{\text{el}})^\perp$. Its inverse is denoted by $\colQ_{\text{el}}^{-1}$.
\item For any $\psi(\varepsilon(k))$, we have $\colQ_{\text{el}}(f\psi)=\psi \colQ_{\text{el}}(f)$ and $\colQ_{\text{el}}^{-1}(f\psi)=\psi \colQ_{\text{el}}^{-1}(f)$.
\end{enumerate}
\end{proposition}

\begin{proposition}\cite{ADG96}  \label{prop: qee}
\begin{enumerate}
\item For any ``regular" test function $g(k)$,
\begin{eqnarray*}
\int_{\mathbb{R}^d} &&\colQ_{\text{ee}}(f)g\,dk=-\frac{1}{4}\int_{\mathbb{R}^{4d}} \Phi_{\text{ee}}(k,k_1,k',k_1') \delta(\varepsilon' \!+\! \varepsilon_1' \!-\! \varepsilon \!-\! \varepsilon_1  )\delta (k' \!+\! k_1' \!-\! k \!-\! k_1  ) \nonumber\\  
&&  \!\!\!\! \times\Big[   f'f_1'(1\!-\!\eta f)(1\!-\! \eta f_1) - ff_1(1\!-\! \eta f')(1\!-\! \eta f_1')  \Big](g'\!+\! g_1' \!-\! g \!- \! g_1)\,  dkdk_1 dk' dk_1'.
\end{eqnarray*}
In particular, we have the conservation of mass and energy 
\begin{equation*}
\int_{\mathbb{R}^d} \colQ_{\text{ee}}(f)\,dk=\int_{\mathbb{R}^d} \colQ_{\text{ee}}(f)\varepsilon \,dk=0.
\end{equation*}
\item H-theorem: let $H(f)=\ln (f/(1-\eta f))$, then $\int_{\mathbb{R}^d} \colQ_{\text{ee}}(f) H(f)\,dk \leq 0$, 
and if $f=f(\varepsilon(k))$,
\begin{equation*} 
\int_{\mathbb{R}^d} \colQ_{\text{ee}}(f) H(f)\,dk=0  \Longleftrightarrow f=M(\varepsilon(k)) \Longleftrightarrow \colQ_{\text{ee}}(f)=0,
\end{equation*}
where
\begin{equation} \label{FD}
M(\varepsilon(k))=\frac{1}{\eta}\frac{1}{z^{-1}e^{\frac{\varepsilon(k)}{T}}+1}
\end{equation}
is the Fermi-Dirac distribution function \cite{Pathria}. The variables $z$ and $T$ are the fugacity and (electron) temperature. Alternatively, $M$ can be defined in terms of the chemical potential $\mu=T\ln z $ and $T$.
\end{enumerate}
\end{proposition}

We also need the properties of the operator $\overline{\colQ_{\text{ee}}}$, an energy space counterpart of $\colQ_{\text{ee}}$: for any $F(\varepsilon(k))$, $$\overline{\colQ_{\text{ee}}}(F)(\varepsilon):=\int_{S_\varepsilon}\colQ_{\text{ee}}(F)(k)\,dN_{\varepsilon}(k).$$

\begin{proposition}\cite{AD96}  \label{prop: qeeb}
\begin{enumerate}
\item Conservation of mass and energy $$\int_R \overline{\colQ_{\text{ee}}}(F)\,d\varepsilon=\int_{R} \overline{\colQ_{\text{ee}}}(F)\varepsilon \,d\varepsilon=0.$$

\item H-theorem: let $H(F)=\ln (F/(1-\eta F))$, then $\int_R \overline{\colQ_{\text{ee}}}(F) H(F)\,d\varepsilon\leq 0$, and
\begin{equation*} 
\int_R \overline{\colQ_{\text{ee}}}(F) H(F)\,d\varepsilon=0  \Longleftrightarrow F=M(\varepsilon) \Longleftrightarrow \overline{\colQ_{\text{ee}}}(F)=0.
\end{equation*}
\end{enumerate}
\end{proposition}


Now that the mathematical preliminaries are set up, we are ready to derive the macroscopic limit. The main result is summarized in the following theorem.

\begin{theorem}\cite{DLS04}
In equation (\ref{kinetic1}), when $\alpha\rightarrow 0$, the solution $f$ formally tends to a Fermi-Dirac distribution function (\ref{FD}), with the position and time dependent fugacity $z(x,t)$ and temperature $T(x,t)$ satisfying the so-called Energy-Transport (ET) model:
\begin{eqnarray} \label{ET3}
\partial_t \left( \begin{array}{c} \rho \\ \rho \energy \end{array} \right) 
+  \left( \begin{array}{c} \dxn\cdot j_{\rho} \\ \dxn \cdot j_{\energy}   \end{array} \right) 
-   \left(   \begin{array}{c} 0 \\  \dxn V \cdot j_{\rho} \end{array} \right) = \left( \begin{array}{c} 0 \\ W_{\text{ph}}^{\text{inel}} \end{array} \right),
\end{eqnarray}
where the density $\rho$ and energy $\energy$ are defined as
\begin{equation} \label{rhoe}
\rho(z,T)=\int_{\mathbb{R}^d}f\,dk=\int_{\mathbb{R}^d}M\,dk, \quad \energy(z,T)=\frac{1}{\rho}\int_{\mathbb{R}^d}f\varepsilon\,dk=\frac{1}{\rho}\int_{\mathbb{R}^d}M\varepsilon\,dk;
\end{equation}
the fluxes $j_{\rho}$ and $j_{\energy}$ are given by
\begin{eqnarray} \label{flux}  
 \left( \begin{array}{c} j_{\rho}(z,T) \\ j_{\energy}(z,T) \end{array} \right) = -  \left( \begin{array}{cc} \mathcal{D}_{11}  & \mathcal{D}_{12}\\ \mathcal{D}_{21} & \mathcal{D}_{22}  \end{array} \right) 
\left(   \begin{array}{c} \frac{\dxn z}{z}-\frac{\dxn V}{T} \\  \frac{\dxn T}{T^2} \end{array} \right)
\end{eqnarray}
with the diffusion matrices 
\begin{equation}
\mathcal{D}_{ij}=\int_{\mathbb{R}^d}\dkn \varepsilon \otimes \colQ^{-1}_{\text{el}}(-\dkn \varepsilon)M(1-\eta M)\varepsilon^{i+j-2}\,dk;
\end{equation}
and the energy relaxation operator $W_{\text{ph}}^{\text{inel}}$ is 
\begin{equation} \label{Wph}
W_{\text{ph}}^{\text{inel}}(z,T)= \int_{\mathbb{R}^d}\colQ_{\text{ph}}^{\text{inel}} (M)\varepsilon \,dk.
\end{equation}
\end{theorem}

\begin{proof}
Inserting the Hilbert expansion $f = f_0 + \alpha f_1 + \alpha ^2 f_2+ \dots$ into equation (\ref{kinetic1}) and collecting equal powers of $\alpha$ leads to
\begin{eqnarray}
&& O(\alpha^{-2}):   \quad  \colQ_{\text{el}} (f_0) = 0,   \label{order-2}\\ 
&& O(\alpha^{-1}): \quad   \colQ_{\text{el}} (f_1) =\dkn \varepsilon \cdot \dxn f_0  + \dxn V \cdot \dkn f_0- \colQ_{\text{ee}} (f_0). \label{order-1}
\end{eqnarray}
From (\ref{order-2}) and Proposition \ref{prop: qel} (3), we know there exists some function $F$ such that
\begin{equation*} 
f_0(x,k,t) = F(x, \varepsilon(k),t).
\end{equation*} 
Plugging $f_0$ into (\ref{order-1}):
\begin{equation} \label{order-11}
\colQ_{\text{el}} (f_1) = \dkn \varepsilon \cdot \left( \dxn F  + \dxn V \partial_{\varepsilon} F  \right)- \colQ_{\text{ee}} (F).
\end{equation}
The solvability condition for $f_1$ (Proposition \ref{prop: qel} (4--5)) implies
\begin{equation*}
\int_{S_\varepsilon} \dkn \varepsilon \cdot \left( \dxn F  + \dxn V \partial_{\varepsilon} F  \right)\,dN_\varepsilon (k) =\int_{S_\varepsilon}  \colQ_{\text{ee}} (F)\,dN_\varepsilon (k) .
\end{equation*}
Clearly the left hand side of above equation is equal to zero ($\dkn \varepsilon$ is odd in $k$), so 
\begin{equation*}
\int_{S_\varepsilon}  \colQ_{\text{ee}} (F)\,dN_\varepsilon (k)= \overline{\colQ_{\text{ee}}}(F)=0.
\end{equation*}
By Proposition \ref{prop: qeeb} (2), we know that $F$ is a Fermi-Dirac distribution $M$ with position and time dependent $z(x,t)$ and $T(x,t)$.
Therefore, $\colQ_{\text{ee}} (F)$ itself is equal to zero by Proposition \ref{prop: qee} (2), and (\ref{order-11}) reduces to
\begin{equation*} 
\colQ_{\text{el}} (f_1) = \dkn \varepsilon \cdot \left( \dxn M  + \dxn V \partial_{\varepsilon} M \right).
\end{equation*}
Then using Proposition \ref{prop: qel} (5--6), we have
\begin{equation}  \label{f1}
f_1=-\colQ_{\text{el}}^{-1} (-\dkn \varepsilon)  \cdot \left( \dxn M  + \dxn V \partial_{\varepsilon} M\right).
\end{equation}
Going back to the original equation (\ref{kinetic1}), multiplying both sides by $(1,\varepsilon(k))^T$ and integrating w.r.t. $k$ gives:
\begin{eqnarray} \label{ET1}
&&\int_{\mathbb{R}^d} \left[ \partial_t f+ \frac{1}{\alpha}\left( \dkn \varepsilon \cdot \dxn f  + \dxn V \cdot \dkn f\right)\right] \left(\begin{array}{c} 1 \\ \varepsilon \end{array} \right)\,dk\nonumber\\
&&=\int_{\mathbb{R}^d} \left[  \frac{1}{\alpha^2}\colQ_{\text{el}} (f) +\frac{1}{\alpha} \colQ_{\text{ee}} (f)+\colQ_{\text{ph}}^{\text{inel}}(f)\right] \left(\begin{array}{c} 1 \\ \varepsilon \end{array} \right)\,dk.
\end{eqnarray}
Terms involving $\colQ_{\text{el}} (f)$ and $\colQ_{\text{ee}} (f)$ vanish due to Propositions \ref{prop: qel} (1) and \ref{prop: qee} (1). Recall that $\colQ_{\text{ph}}^{\text{inel}}$ is the difference between $\colQ_{\text{ph}}$ and an elastic operator, and both of them can be easily seen to conserve mass, so $\int_{\mathbb{R}^d} \colQ_{\text{ph}}^{\text{inel}} (f)\,dk=0$. Thus (\ref{ET1}) simplifies to
\begin{eqnarray} \label{ET2}
\quad  \quad \partial_t \left( \begin{array}{c} \rho \\ \rho \energy \end{array}\right) +\int_{\mathbb{R}^d}  \frac{1}{\alpha}\left( \dkn \varepsilon \cdot \dxn f  + \dxn V \cdot \dkn f\right) \left(\begin{array}{c} 1 \\ \varepsilon \end{array} \right)\,dk=\left(\begin{array}{c} 0\\ W_{\text{ph}}^{\text{inel}}(f)\end{array}\right),
\end{eqnarray}
where $W_{\text{ph}}^{\text{inel}}(f)= \int_{\mathbb{R}^d}\colQ_{\text{ph}}^{\text{inel}} (f)\varepsilon\,dk$.

From the previous discussion, we know $f=f_0+\alpha f_1+\dots$ with $f_0$ being the Fermi-Dirac distribution (\ref{FD}), and $f_1$ given by (\ref{f1}). Plugging $f$ into (\ref{ET2}), to the leading order we have ($O(\alpha^{-1})$ term drops out since $f_0$ is an even function in $k$):
\begin{eqnarray} \label{ET21}
\quad \quad \partial_t \left( \begin{array}{c} \rho \\ \rho \energy \end{array}\right) +\int_{\mathbb{R}^d}  \left( \dkn \varepsilon \cdot \dxn f_1 + \dxn V \cdot \dkn f_1\right)\left(\begin{array}{c} 1 \\ \varepsilon \end{array} \right)\,dk=\left(\begin{array}{c} 0\\ W_{\text{ph}}^{\text{inel}}(M)\end{array}\right).
\end{eqnarray}
Utilizing the special form of $M$, one can rewrite $f_1$ as
\begin{equation*}
f_1=-\colQ_{\text{el}}^{-1} (-\dkn \varepsilon)  \cdot \left( \frac{\dxn z}{z}-\frac{\dxn V}{T}+\varepsilon \frac{\dxn T}{T^2}\right)M(1-\eta M).
\end{equation*}
Then a simple manipulation of (\ref{ET21}) yields the ET system (\ref{ET3}).
\end{proof}

The ET model (\ref{ET3}) is widely used in practical and industrial applications (see \cite{Jungel10} for a review and references therein). It can also be derived from the Boltzmann equation through the so-called SHE (Spherical Harmonics Expansion) model \cite{GVBO93}. In either case, the rigorous theory behind the formal limit is still an open question.

\begin{remark}
If the electron-electron interaction $\colQ_{\text{ee}}$ is assumed as one of the dominant terms in (\ref{kinetic1}), i.e., the same order as elastic collision $\colQ_{\text{el}}$ (which could be the physically relevant situation for very dense electrons), one can still derive the ET model (\ref{ET3}) via a similar procedure \cite{ADG96}, but the diffusion coefficients $\mathcal{D}_{ij}$ are different. A rigorous result is available in this case \cite{ADG00}.
\end{remark}

\begin{remark}
If we consider even longer time scale such that the energy relaxation term $W_{\text{ph}}^{\text{inel}}$ equilibrates $T$ to the lattice temperature $T_L$, and further assume that $M$ is the classical Maxwellian, then (\ref{ET3}) reduces to a single equation
\begin{equation*} 
\partial_t  \rho+\dxn \cdot \left [ -\mathcal{D}_{11}\left(\frac{\dxn \rho}{\rho}-\frac{\dxn V}{T_L}\right)\right]=0.
\end{equation*}
This is the classical drift-diffusion model \cite{Roosebroeck50, Poupaud91}.
\end{remark}

\begin{remark}
Unlike the classical statistics, given macroscopic variables $\rho$ and $\energy$ (\ref{rhoe}), finding the corresponding Fermi-Dirac distribution (\ref{FD}) is not a trivial issue. Under the parabolic approximation (\ref{para}), $\rho$ and $\energy$ are related to $z$ and $T$ via \cite{HJ}
\begin{eqnarray} \label{system}
\left\{\begin{array} {l}  \displaystyle \rho=\frac{(2\pi T)^{\frac{d}{2}}}{\eta}\mathcal{F}_{\frac{d}{2}}(z), \\  \displaystyle \mathcal{E}=\frac{d}{2}T\frac{\mathcal{F}_{\frac{d}{2}+1}(z)}{\mathcal{F}_{\frac{d}{2}}(z)},
\end{array} \right.
\end{eqnarray}
where $\mathcal{F}_{\nu}(z)$ is the Fermi-Dirac function of order $\nu$
\begin{equation} \label{fermi}
\mathcal{F}_{\nu}(z)=\frac{1}{\Gamma (\nu) }\int _0^\infty  \frac{x^{\nu-1}}{z^{-1}e^x+1}\,dx,  \quad 0<z<\infty,
\end{equation}
and $ \displaystyle \Gamma(\nu)$ is the Gamma function. For small $z$ ($0<z<1$), the integrand in (\ref{fermi}) can be expanded in powers of $z$:
\begin{equation*} 
\mathcal{F}_{\nu}(z)=\displaystyle \sum_{n=1}^{\infty}(-1)^{n+1}\frac{z^n}{n^{\nu}}=z-\frac{z^2}{2^{\nu}}+\frac{z^3}{3^{\nu}}-\dots. 
\end{equation*}
Thus, when $z\ll1$, $\mathcal{F}_{\nu}(z)$ behaves like $z$ itself and one recovers the classical limit. 
\end{remark}


\section{Asymptotic-preserving (AP) schemes for the semiconductor Boltzmann equation}

Equation (\ref{kinetic1}) contains three different scales. To overcome the stiffness induced by the $O(\alpha^{-2})$ and $O(\alpha^{-1})$ terms, a fully implicit scheme would be desirable. However, neither the collision operators nor the convection terms are easy to solve implicitly. Our goal is to design an appropriate numerical scheme that is uniformly stable in both kinetic and diffusive regimes, i.e., works for all values of $\alpha$ ranging from $\alpha\sim O(1)$ to $\alpha \ll 1$, while the implicit terms can be treated explicitly. 

We will first consider a spatially homogeneous case with an emphasis on the collision operators, and then include the spatial dependence to treat the convection terms. To facilitate the presentation, we always make the following assumptions without further notice:
\begin{enumerate}
\item The inelastic collision operator $\colQ^{\text{inel}}_{\text{ph}}$ in (\ref{kinetic1}) is assumed to be zero, since it is the weakest effect and its appearance won't bring extra difficulties to numerical schemes.
\item The scattering matrices $\Phi_{\text{el}}$ and $\Phi_{\text{ee}}$ are {\it rotationally invariant}:
\begin{equation*}
\Phi_{\text{el}}(k,k')=\Phi_{\text{el}}(|k|,|k'|), \quad \Phi_{\text{ee}}(k,k_1,k',k_1')=\Phi_{\text{ee}}(|k|,|k_1|,|k'|,|k_1'|).
\end{equation*}
\end{enumerate}
Then it is not difficult to verify that (see Proposition \ref{prop: qel} (4))
\begin{equation} \label{Qel3}
\colQ_{\text{el}}(f)(k)=\lambda_{\text{el}}(\varepsilon)([f](\varepsilon)-f(k)),
\end{equation}
where 
\begin{equation*}
\lambda_{\text{el}}(\varepsilon(k)):= \int_{\mathbb{R}^d} \Phi_{\text{el}} (k, k')\delta (\varepsilon'  - \varepsilon) \,dk'=\Phi_{\text{el}}(|k|,|k|)N(\varepsilon),
\end{equation*}
and $[f](\varepsilon(k))$ is the mean value of $f$ over sphere $S_{\varepsilon}$:
\begin{equation*} 
[f](\varepsilon(k)):=\frac{1}{N(\varepsilon)}\int_{S_{\varepsilon}}f(k)\,dN_{\varepsilon}(k).
\end{equation*}
In particular, for any odd function $f(k)$,
\begin{equation*}
\colQ_{\text{el}}(f)(k)=-\lambda_{\text{el}}(\varepsilon)f(k), \quad \colQ_{\text{el}}^{-1}(f(k))=-\frac{1}{\lambda_{\text{el}}(\varepsilon)} f(k).
\end{equation*}
This observation is crucial in designing our AP schemes.


\subsection{The spatially homogeneous case}

In the spatially homogeneous case, equation (\ref{kinetic1}) reduces to
\begin{equation} \label{homo}
\partial_t f  = \frac{1}{\alpha^2} \colQ_{\text{el}}(f) +  \frac{1}{\alpha} \colQ_{\text{ee}}(f),
\end{equation}
where $f$ only depends on $k$ and $t$. An explicit discretization of (\ref{homo}), e.g., the forward Euler scheme, suffers from severe stability constraints: $\Delta t$ has to be smaller than $O(\alpha^2)$. Implicit schemes do not have such a restriction, but require some sort of iteration solvers for $\colQ_{\text{el}}$ and $\colQ_{\text{ee}}$ which can be quite complicated. 

To tackle these two stiff terms, we adopt the penalization idea in \cite{FJ10}, i.e., penalize both $\colQ_{\text{el}}$ and $\colQ_{\text{ee}}$ by their corresponding ``BGK" operators:
\begin{equation} \label{penalization}
\partial_tf = \frac{ \colQ_{\text{el}}(f)-\beta_{\text{el}}(M_{\text{el}}-f) }{\alpha^2}+\frac{\beta_{\text{el}}(M_{\text{el}}-f) }{\alpha^2}+  \frac{\colQ_{\text{ee}}(f)-\beta_{\text{ee}}(M_{\text{ee}}-f)}{\alpha} +\frac{\beta_{\text{ee}}(M_{\text{ee}}-f) }{\alpha}.
\end{equation}
Using the properties of $\colQ_{\text{ee}}$ and $\colQ_{\text{el}}$ from the last section, $M_{\text{ee}}$ can be naturally chosen as the Fermi-Dirac distribution $M$ (in the homogeneous case $\rho$ and $\energy$ are conserved, so $M$ is an absolute Maxwellian and can be obtained from the initial condition), whereas $M_{\text{el}}$ can in principle be any function of $\varepsilon$ such that $M_{\text{el}}$ and $f$ share the same density and energy (the choice of $M_\text{el}$ is not essential as we shall see, and we will get back to this when we consider the spatially inhomogeneous case). As the goal of penalization is to make the residue $\colQ_{\bullet}(f)-\beta_{\bullet}(M_{\bullet}-f)$ as small as possible so that it is non-stiff or less stiff, and $\colQ_{\text{el}}$, $\colQ_{\text{ee}}$ can be expressed symbolically as
\begin{equation*} \label{col}
\colQ_{\text{el}}(f)(k)=\colQ_{\text{el}}^+(f)(\varepsilon)-\lambda_{\text{el}}(\varepsilon)f(k); \quad \colQ_{\text{ee}}(f)(k)=\colQ_{\text{ee}}^+(f)(k)-\colQ_{\text{ee}}^-(f)(k)f(k),
\end{equation*}
we hence choose 
\begin{equation} \label{coef}
\beta_{\text{el}}\approx\max_{\varepsilon}\lambda_{\text{el}}(\varepsilon);  \quad \beta_{\text{ee}}\approx\max_{k} \colQ_{\text{ee}}^-(f)(k).
\end{equation}
Other choices are also possible \cite{FJ10}. Generally speaking, we only need the coefficient to be a rough estimate of the Frechet derivative of the collision operator around equilibrium. 

Therefore, a first-order IMEX scheme for (\ref{penalization}) is written as 
\begin{eqnarray} \label{homoscheme}
\frac{f^{n+1}-f^n}{\Delta t} = \ &&\frac{ \colQ_{\text{el}}(f^n)-\beta_{\text{el}}(M_{\text{el}}-f^n) }{\alpha^2}+\frac{\beta_{\text{el}}(M_{\text{el}}-f^{n+1}) }{\alpha^2} \nonumber 
\\&& +  \frac{\colQ_{\text{ee}}(f^n)-\beta_{\text{ee}}(M-f^n)}{\alpha} +\frac{\beta_{\text{ee}}(M-f^{n+1}) }{\alpha}.
\end{eqnarray}


\subsubsection{Asymptotic properties of the numerical solution}

To better understand the asymptotic behavior of the numerical solution, in this subsection we assume that $\colQ_{\text{ee}}(f)=M-f$. Then scheme (\ref{homoscheme}) becomes
\begin{eqnarray*} 
\frac{f^{n+1}-f^n}{\Delta t} = \ && \frac{ \colQ_{\text{el}}(f^n)-\beta_{\text{el}}(M_{\text{el}}-f^n) }{\alpha^2}  +\frac{\beta_{\text{el}}(M_{\text{el}}-f^{n+1}) }{\alpha^2} \nonumber
\\&&  +  \frac{(1-\beta_{\text{ee}})(M-f^n)}{\alpha} +\frac{\beta_{\text{ee}}(M-f^{n+1}) }{\alpha}.
\end{eqnarray*}
This is equivalent to
\begin{eqnarray*} \label{init}
f^{n+1} =&&\ \frac{1+\frac{\Delta t}{\alpha^2}(\beta_{\text{el}}-\lambda_{\text{el}})+\frac{\Delta t}{\alpha}(\beta_{\text{ee}}-1)}{1+\frac{\Delta t}{\alpha^2}\beta_{\text{el}}+\frac{\Delta t}{\alpha}\beta_{\text{ee}}}f^n+\frac{\frac{\Delta t}{\alpha^2}}{1+\frac{\Delta t}{\alpha^2}\beta_{\text{el}}+\frac{\Delta t}{\alpha}\beta_{\text{ee}}}\colQ_{\text{el}}^+(f^n)\nonumber
\\&& +\frac{\frac{\Delta t}{\alpha}}{1+\frac{\Delta t}{\alpha^2}\beta_{\text{el}}+\frac{\Delta t}{\alpha}\beta_{\text{ee}}}M \nonumber\\
=&& \ \frac{1+\frac{\Delta t}{\alpha^2}(\beta_{\text{el}}-\lambda_{\text{el}})+\frac{\Delta t}{\alpha}(\beta_{\text{ee}}-1)}{1+\frac{\Delta t}{\alpha^2}\beta_{\text{el}}+\frac{\Delta t}{\alpha}\beta_{\text{ee}}}f^n+\text{ some function of }  \varepsilon.
\end{eqnarray*}
Iteratively, it gives
\begin{equation*}
f^{n}=\left(\frac{1+\frac{\Delta t}{\alpha^2}(\beta_{\text{el}}-\lambda_{\text{el}})+\frac{\Delta t}{\alpha}(\beta_{\text{ee}}-1)}{1+\frac{\Delta t}{\alpha^2}\beta_{\text{el}}+\frac{\Delta t}{\alpha}\beta_{\text{ee}}}\right)^nf^0+\text{ some function of } \varepsilon.
\end{equation*}
So when $\alpha$ is small, for any $m$, there exists some integer $N$, s.t.
\begin{equation} \label{radial}
f^n\leq O(\alpha^m)+\text{ some function of } \varepsilon, \quad \text{for } n>N,
\end{equation}
which means that $f^n$ can be arbitrarily close to the null space of $\colQ_{\text{el}}$: 
\begin{equation} \label{eqn: qel5}
\colQ_{\text{el}}(f^n)\leq O(\alpha^m),\quad \text{for } n>N.
\end{equation} 

At this stage, if we examine the distance between $f$ and $M$, we found that
\begin{equation} \label{oldAP}
f^{n+1} \!- \!M=\frac{1+\frac{\Delta t}{\alpha^2}\beta_{\text{el}}+\frac{\Delta t}{\alpha}(\beta_{\text{ee}}-1)}{1+\frac{\Delta t}{\alpha^2}\beta_{\text{el}}+\frac{\Delta t}{\alpha}\beta_{\text{ee}}}(f^n \!-\! M)+\frac{\frac{\Delta t}{\alpha^2}}{1+\frac{\Delta t}{\alpha^2}\beta_{\text{el}}+\frac{\Delta t}{\alpha}\beta_{\text{ee}}}\colQ_{\text{el}}(f^n)
\end{equation}
so 
\begin{equation*}
|f^{n+1}-M|\leq r|f^n-M|+O(\alpha^m),\quad \text{for } n>N,
\end{equation*}
where
\begin{equation*}
r=\left|\frac{1+\frac{\Delta t}{\alpha^2}\beta_{\text{el}}+\frac{\Delta t}{\alpha}(\beta_{\text{ee}}-1)}{1+\frac{\Delta t}{\alpha^2}\beta_{\text{el}}+\frac{\Delta t}{\alpha}\beta_{\text{ee}}}\right|.
\end{equation*}
Then for proper $\beta_{\text{ee}}$, we have $0<r<1$ and
\begin{equation} \label{oldAP1}
|f^{n+n_1}-M|\leq r^{n_1}|f^n-M|+O(\alpha^m),\quad \text{for } n>N.
\end{equation}
This implies that no matter what the initial condition is, $f$ will eventually be driven to the desired Fermi-Dirac distribution, but the convergence rate can be rather slow for small $\alpha$. What even worse is, when $\alpha \rightarrow0$, $r\rightarrow1$, we see from (\ref{oldAP}), (\ref{radial}) that $f$ will stay around some function of $\varepsilon$, but nothing guarantees it is $M$! This violates the property 2 mentioned in the Introduction. 

On the other hand, we know from (\ref{eqn: qel5}) that $\colQ_{\text{el}}(f^n)$ can be arbitrarily small after some time. What if we just set it equal to zero afterwards? Dropping this term as well as its penalization leads to
\begin{equation*} 
\frac{f^{n+1}-f^n}{\Delta t} =  \frac{(1-\beta_{\text{ee}})(M-f^n)}{\alpha} +\frac{\beta_{\text{ee}}(M-f^{n+1}) }{\alpha},
\end{equation*}
which is
\begin{equation} \label{fastconv}
f^{n+1}-M=\frac{1+\frac{\Delta t}{\alpha}(\beta_{\text{ee}}-1)}{1+\frac{\Delta t}{\alpha}\beta_{\text{ee}}}(f^n-M).
\end{equation}
Hence as long as $\beta_{\text{ee}}>1/2$, $f^n$ will converge to $M$ regardless of the initial condition. Compared with (\ref{oldAP}), this one has much faster convergence rate.


\subsubsection{The thresholded AP scheme} \label{sec: threshold}

The above simple analysis illustrates that the penalization idea \cite{FJ10} should be applied wisely, especially when there are stiff terms with different scales. Back to the original equation (\ref{homo}), we propose to solve it as follows.

At time step $t^{n+1}$, check the norm of $\colQ_{\text{el}}(f^n)$ in $k$:
\begin{itemize}
\item if $\|\colQ_{\text{el}}(f^n)\|>\delta$, apply scheme (\ref{homoscheme});
\item otherwise, apply (\ref{homoscheme}) with $\colQ_{\text{el}}(f^n)=\beta_{\text{el}}=0$.
\end{itemize}
As explained previously, the threshold can be chosen based on the property we expect in (\ref{eqn: qel5}), say, $\delta = \alpha^m$, $m>2$. However, in practice, as any numerical solver of $\colQ_{\text{el}}$ has certain accuracy, we therefore set
\begin{equation} \label{thres}
\delta = (\Dk)^l,
\end{equation}
where $l$ denotes the error order of the numerical solver for $\colQ_{\text{el}}$.
\begin{remark}
The choice of threshold (\ref{thres}) does not violate the consistency of our method, since when $\Dk \rightarrow 0$, we have $\delta \rightarrow 0$, and we are back to the original scheme (\ref{homoscheme}) whose consistency to (\ref{homo}) is not a problem (i.e. the scheme satisfies the property 1 in the Introduction).
\end{remark}


\subsection{The spatially inhomogeneous case}

We now include the spatial dependence to treat the full problem (\ref{kinetic1}). To handle the newly added stiff convection terms, we follow the idea of \cite{JP00} to form it into a set of parity equations. Denote $f^+ = f(x,k,t)$, $f^- = f(x,-k,t)$, then they solve
\begin{eqnarray*}
\partial_t f^+  +  \frac{1}{\alpha } \left(  \dkn \varepsilon  \cdot  \dxn f^+  +  \dxn V  \cdot  \dkn f^+  \right)  =\ &&  \frac{1 }{\alpha^2}\colQ_{\text{el}}(f^+)  + \frac{1}{\alpha}\colQ_{\text{ee}}(f^+) ,\\ 
 \partial_t f^-  -  \frac{1}{\alpha } \left(  \dkn \varepsilon  \cdot  \dxn f^- +  \dxn V  \cdot  \dkn  f^-  \right)  = \ &&\frac{1}{\alpha^2}\colQ_{\text{el}}(f^-)  +  \frac{1 }{\alpha}\colQ_{\text{ee}}(  f^-).
\end{eqnarray*}
Now write
\begin{equation*}
r(x,k,t) = \half (f^+ + f^-),  \quad j (x,k,t) = \frac{1}{2\alpha} (f^+ - f^-),
\end{equation*}
we have
\begin{eqnarray}
&&\partial_t r +  \dkn \varepsilon  \cdot  \dxn j  +  \dxn V  \cdot  \dkn j = \frac{\colQ_{\text{el}}(r)}{\alpha^2} + \frac{\colQ_{\text{ee}}(f^+ )+ \colQ_{\text{ee}}(f^-) }{2\alpha},  \label{parityr}\\ 
&& \partial_t j + \frac{1}{\alpha^2} \left( \dkn \varepsilon  \cdot  \dxn r  +  \dxn V  \cdot  \dkn r\right) =  -\frac{\lambda_{\text{el}}}{\alpha^2} j + \frac{ \colQ_{\text{ee}}(f^+)  -  \colQ_{\text{ee}}(f^-)}{2\alpha^2},  \label{parityj}
\end{eqnarray}
where we used the fact that $j$ is an odd function in $k$, thus $\colQ_{\text{el}}(j)=-\lambda_{\text{el}}j$.

For (\ref{parityr}--\ref{parityj}), the same penalization as in the homogenous case suggests
\begin{eqnarray*}
&&\partial_t r +  \dkn \varepsilon  \cdot  \dxn j  +  \dxn V  \cdot  \dkn j = \frac{\colQ_{\text{el}}(r) - \beta_{\text{el}}(M_{\text{el}}-r)}{\alpha^2} + \frac{\beta_{\text{el}}(M_{\text{el}}-r)}{\alpha^2}   \nonumber \\ 
&&\hspace{4.5cm} + \frac{\colQ_{\text{ee}}(f^+ )+ \colQ_{\text{ee}}(f^-) - 2 \beta_{\text{ee}}(M-r) }{2\alpha} +\frac{\beta_{\text{ee}} (M-r)}{\alpha},\\ 
 &&\partial_t j + \frac{1}{\alpha^2} \left( \dkn \varepsilon  \cdot  \dxn r  +  \dxn V  \cdot  \dkn r\right) =  -\frac{\lambda_{\text{el}}}{\alpha^2} j  \nonumber\\
 &&\hspace{4.5cm}+ \frac{ \colQ_{\text{ee}}(f^+) -  \colQ_{\text{ee}}(f^-) + 2\beta_{\text{ee}} \alpha j  }{2\alpha^2}-\frac{\beta_{\text{ee}}}{\alpha} j.
\end{eqnarray*}
Note here $M=M(x,\varepsilon,t)$ is the local Fermi-Dirac distribution, and $M_\text{el}=M_\text{el}(x,\varepsilon,t)$ is some function of $\varepsilon$ whose form will be specified later. The coefficients $\beta_{\text{el}}$ and $\beta_{\text{ee}}$ are chosen the same as in (\ref{coef}), except that $\beta_{\text{ee}}$ can also be made space dependent. 

The above equations can be formed into a diffusive relaxation system \cite{JPT00}:
\begin{eqnarray}
&&\quad \quad \quad \partial_t r +  \dkn \varepsilon  \cdot  \dxn j  +  \dxn V  \cdot  \dkn j = G_1(r,j)  +\frac{\beta_{\text{el}}(M_{\text{el}}-r)}{\alpha^2} + \frac{\beta_{\text{ee}}(M-r)}{\alpha}, \label{dr1} \\ 
&&\quad \quad \quad \partial_t j  +  \theta   \left( \dkn \varepsilon  \cdot  \dxn r  +  \dxn V  \cdot  \dkn r\right)  =  G_2(r,j)- \frac{\beta_{\text{ee}}} {\alpha} j\nonumber
\\ &&\hspace{5cm} -\frac{1}{\alpha^2} \left[ \lambda_{\text{el}} j  +  (1 -   \alpha^2 \theta) (\dkn \varepsilon  \cdot  \dxn r   +  \dxn V  \cdot  \dkn r)\right] , \label{dr2}
\end{eqnarray}
where 
\begin{eqnarray}
&& \quad G_1(r,j) =   \frac{\colQ_{\text{el}}(r) - \beta_{\text{el}}(M_{\text{el}}-r)}{\alpha^2} + \frac{\colQ_{\text{ee}}(f^+ )+ \colQ_{\text{ee}}(f^-) - 2 \beta_{\text{ee}} (M-r) }{2\alpha} ,  \label{G1}
\\ &&\quad G_2(r,j) = \frac{ \colQ_{\text{ee}}(f^+)  -  \colQ_{\text{ee}}(f^-) + 2\beta_{\text{ee}} \alpha j  }{2\alpha^2} , \label{G2}
\end{eqnarray}
and $0 \leq \theta(\alpha) \leq 1/\alpha^2$ is a control parameter simply chosen as $\theta(\alpha) = \min \left\{ 1, 1/\alpha^2\right\}$.

A first-order IMEX scheme for the system (\ref{dr1}--\ref{dr2}) thus reads
\begin{eqnarray}
&&\frac{r^{n+1}-r^n}{\Delta t}+  \dkn \varepsilon  \cdot  \dxn j^n +  \dxn V^n  \cdot  \dkn j^n =  G_1(r^n,j^n) +\frac{\beta_{\text{el}} (M_{\text{el}}^{n+1}-r^{n+1})}{\alpha^2} \nonumber
\\ && \hspace{8.3cm} +\frac{\beta_{\text{ee}} (M^{n+1}-r^{n+1})}{\alpha}, \label{rr}\\ 
&&\frac{j^{n+1}-j^n}{\Dt} + \theta  \left( \dkn \varepsilon  \cdot  \dxn r^n  +  \dxn V^n  \cdot  \dkn r^n\right) =G_2(r^n,j^n)-\frac{\beta_{\text{ee}}}{\alpha} j^{n+1}\nonumber\\
&& \hspace{2.1cm}-\frac{1}{\alpha^2}[ \lambda_{\text{el}} j^{n+1}  +  (1 -   \alpha^2 \theta) (\dkn \varepsilon  \cdot  \dxn r^{n+1} + \dxn V^{n+1}  \cdot  \dkn r^{n+1})].\label{jj}
\end{eqnarray}


\subsubsection{Asymptotic properties of the numerical solution}

Leaving aside other issues such as the spatial discretization, let us for the moment again assume that $\colQ_{\text{ee}}(f)=M-f$, then (\ref{rr}--\ref{jj}) simplify to
\begin{eqnarray*}
&&\frac{r^{n+1}-r^n}{\Delta t} \!+\!  \dkn \varepsilon  \cdot  \dxn j^n +  \dxn V^n  \cdot  \dkn j^n =
\frac{\colQ_{\text{el}}(r^n) \!-\! \beta_{\text{el}}(M_{\text{el}}^n \!-\! r^n)}{\alpha^2} + \frac{\beta_{\text{el}} (M_{\text{el}}^{n+1} \!-\!r^{n+1})}{\alpha^2}   \nonumber \\ 
 &&\hspace{6.7cm}+ \frac{(1-\beta_{\text{ee}})(M^n \!-\! r^n)}{\alpha} +\frac{\beta_{\text{ee}} (M^{n+1} \!-\! r^{n+1})}{\alpha} ,\\ 
&&  \frac{j^{n+1}-j^n}{\Dt} +  \theta  \left( \dkn \varepsilon  \cdot  \dxn r^n  +  \dxn V^n  \cdot  \dkn r^n\right) = \frac{\beta_{\text{ee}}-1}{\alpha}j^n -\frac{\beta_{\text{ee}}}{\alpha} j^{n+1} \nonumber\\ 
&& \hspace{3.8cm}-\frac{1}{\alpha^2}[ \lambda_{\text{el}} j^{n+1}  +  (1 -   \alpha^2 \theta) (\dkn \varepsilon  \cdot  \dxn r^{n+1} + \dxn V^{n+1}  \cdot  \dkn r^{n+1})]  .
\end{eqnarray*}
The scheme for $j$ is actually
\begin{eqnarray*}
&&j^{n+1} =\frac{1+\frac{\Delta t}{\alpha}(\beta_{\text{ee}}-1)}{1+\frac{\Delta t}{\alpha^2}\lambda_{\text{el}}+\frac{\Delta t}{\alpha}\beta_{\text{ee}}}j^n-\frac{ \frac{\Delta t}{\alpha^2}}{1+\frac{\Delta t}{\alpha^2}\lambda_{\text{el}}+\frac{\Delta t}{\alpha}\beta_{\text{ee}}}(\dkn \varepsilon  \cdot  \dxn r^{n+1} +  \dxn V^{n+1}  \cdot  \dkn r^{n+1})\nonumber\\
&&+\frac{\theta\Delta t}{1+\frac{\Delta t}{\alpha^2}\lambda_{\text{el}}+\frac{\Delta t}{\alpha}\beta_{\text{ee}}}\left[ \left(\dkn \varepsilon  \cdot  \dxn r^{n+1} +\dxn V^{n+1}\cdot \dkn r^{n+1} \right)- \left(\dkn \varepsilon  \cdot  \dxn r^n +\dxn V^n\cdot \dkn r^n\right) \right].
\end{eqnarray*}
When $\alpha\ll 1$ (so $\theta=1$), suppose all functions are smooth, we have
\begin{equation*}
j^{n+1} =-\frac{1}{\lambda_{\text{el}}}(\dkn \varepsilon  \cdot  \dxn r^{n+1} +  \dxn V^{n+1}  \cdot  \dkn r^{n+1})+O(\alpha).
\end{equation*}
The scheme for $r$ results in
\begin{eqnarray*}
r^{n+1}=\ && \frac{1+\frac{\Delta t}{\alpha^2}(\beta_{\text{el}} \!\!-\!\! \lambda_{\text{el}})+\frac{\Delta t}{\alpha}(\beta_{\text{ee}} \!\!-\!1)}{1+\frac{\Delta t}{\alpha^2}\beta_{\text{el}}+\frac{\Delta t}{\alpha}\beta_{\text{ee}}}r^n-\frac{\Delta t}{1\!+\! \frac{\Delta t}{\alpha^2}\beta_{\text{el}}\!+\! \frac{\Delta t}{\alpha}\beta_{\text{ee}}}(\dkn \varepsilon  \!\cdot\!  \dxn j^n \!+\!  \dxn V^n \! \cdot \! \dkn j^n)\nonumber\\
&&+\frac{\frac{\Delta t}{\alpha^2}}{1+\frac{\Delta t}{\alpha^2}\beta_{\text{el}}+\frac{\Delta t}{\alpha}\beta_{\text{ee}}}\colQ_{\text{el}}^+(r^n)
+\frac{\frac{\Delta t}{\alpha}\beta_{\text{ee}}}{1+\frac{\Delta t}{\alpha^2}\beta_{\text{el}}+\frac{\Delta t}{\alpha}\beta_{\text{ee}}}(M^{n+1}-M^n)\nonumber
\\
&&+ \frac{\frac{\Delta t}{\alpha}}{1+\frac{\Delta t}{\alpha^2}\beta_{\text{el}}+\frac{\Delta t}{\alpha}\beta_{\text{ee}}}M^n+ \frac{\frac{\Delta t}{\alpha^2}\beta_{\text{el}}}{1+\frac{\Delta t}{\alpha^2}\beta_{\text{el}}+\frac{\Delta t}{\alpha}\beta_{\text{ee}}}(M_{\text{el}}^{n+1}-M_{\text{el}}^n)
  \nonumber\\
=\ && \frac{1+\frac{\Delta t}{\alpha^2}(\beta_{\text{el}}-\lambda_{\text{el}})+\frac{\Delta t}{\alpha}(\beta_{\text{ee}}-1)}{1+\frac{\Delta t}{\alpha^2}\beta_{\text{el}}+\frac{\Delta t}{\alpha}\beta_{\text{ee}}}r^n+O(\alpha^2)+\text{ some function of }  \varepsilon.
\end{eqnarray*}
Iteratively, this gives
\begin{equation*}
r^n=\left(\frac{1+\frac{\Delta t}{\alpha^2}(\beta_{\text{el}}-\lambda_{\text{el}})+\frac{\Delta t}{\alpha}(\beta_{\text{ee}}-1)}{1+\frac{\Delta t}{\alpha^2}\beta_{\text{el}}+\frac{\Delta t}{\alpha}\beta_{\text{ee}}}\right)^nr^0+O(\alpha^2)+\text{ some function of }  \varepsilon.
\end{equation*}
Clearly the first term on the right hand side will be damped down as time goes by. After several steps, we have
\begin{equation*}
r^n=O(\alpha^2)+\text{ some function of } \varepsilon, \quad \text{for } n>N,
\end{equation*}
which implies
\begin{equation*}
\colQ_{\text{el}}(r^n)=O(\alpha^2),  \quad \text{for } n>N.
\end{equation*} 

At this stage, if we continue to penalize $\colQ_{\text{el}}(r^n)$, similarly as before,
\begin{eqnarray} \label{eqn: ap 1}
r^{n+1}-M^{n+1}=\ && \frac{1+\frac{\Delta t}{\alpha^2}\beta_{\text{el}}+\frac{\Delta t}{\alpha}(\beta_{\text{ee}}-1)}{1+\frac{\Delta t}{\alpha^2}\beta_{\text{el}}+\frac{\Delta t}{\alpha}\beta_{\text{ee}}}(r^n-M^n)+\frac{\frac{\Delta t}{\alpha^2}}{1+\frac{\Delta t}{\alpha^2}\beta_{\text{el}}+\frac{\Delta t}{\alpha}\beta_{\text{ee}}}\colQ_{\text{el}}(r^n)\nonumber \\
&&-\frac{\Dt}{1+\frac{\Delta t}{\alpha^2}\beta_{\text{el}}+\frac{\Delta t}{\alpha}\beta_{\text{ee}}}\left[(\dkn \varepsilon  \cdot  \dxn j^n +  \dxn V^n  \cdot  \dkn j^n)+\frac{M^{n+1}-M^n}{\Dt}\right]\nonumber\\
&&-\frac{\frac{\Dt^2}{\alpha^2}\beta_{\text{el}}}{1+\frac{\Delta t}{\alpha^2}\beta_{\text{el}}+\frac{\Delta t}{\alpha}\beta_{\text{ee}}}\left(\frac{M^{n+1}-M^n}{\Dt}-\frac{M_{\text{el}}^{n+1}-M_{\text{el}}^n}{\Dt}\right) \nonumber\\
=\ && \frac{1+\frac{\Delta t}{\alpha^2}\beta_{\text{el}}+\frac{\Delta t}{\alpha}(\beta_{\text{ee}}-1)}{1+\frac{\Delta t}{\alpha^2}\beta_{\text{el}}+\frac{\Delta t}{\alpha}\beta_{\text{ee}}}(r^n-M^n)+O(\alpha^2 + \Dt),  \quad \text{for } n>N.
\end{eqnarray}
From (\ref{eqn: ap 1}) we see that $r$ will converge to $M$ with a dominant $O(\Dt)$ error. Therefore, a good choice of $\Mel$ is
\begin{equation*}
\Mel = M,
\end{equation*}
which will not only simplify the scheme, but also improve the asymptotic error from $O(\Dt)$ to $O(\alpha^2)$. 

Moreover, like the spatially homogeneous case, we suffer from the same problem that the convergence rate is too slow for small but finite $\alpha$. To accelerate the convergence, we again choose some threshold $\delta$ as in (\ref{thres}) such that once $\|\Qel(r^n)\|(x)<\delta$ (check it for every $x$), we set $\Qel(r^n)=0$ and turn off its penalization. Then (\ref{eqn: ap 1}) becomes
\begin{eqnarray*}
r^{n+1}-M^{n+1}=\ && \frac{1+\frac{\Delta t}{\alpha}(\beta_{\text{ee}}-1)}{1+\frac{\Delta t}{\alpha}\beta_{\text{ee}}}(r^n-M^n)\nonumber\\
&&-\frac{\Dt}{1+\frac{\Delta t}{\alpha}\beta_{\text{ee}}}\left[(\dkn \varepsilon  \cdot  \dxn j^n +  \dxn V^n  \cdot  \dkn j^n)+\frac{M^{n+1}-M^n}{\Dt}\right]
\nonumber\\ = \ && \frac{1+\frac{\Delta t}{\alpha}(\beta_{\text{ee}}-1)}{1+\frac{\Delta t}{\alpha}\beta_{\text{ee}}}(r^n-M^n)+O(\alpha),
\end{eqnarray*}
and we gain much faster convergence rate.


\subsubsection{The thresholded semi-discrete AP scheme} \label{sec: SAP}

Based on the discussion above, we integrate the thresholding idea into (\ref{rr}--\ref{jj}) to propose the following semi-discrete AP scheme.

At time step $t^{n+1}$, given $f^n=(f^+)^n$, $(f^-)^n$, $r^n$, $j^n$, $\rho^n$, $\energy^n$, and $V^n$:
\begin{itemize}
\item {\it Step 1}: Compute $M^{n+1}$ used in (\ref{rr}) and $V^{n+1}$ in (\ref{jj}).

Although (\ref{rr}) appears implicit (recall $\Mel=M$), $M^{n+1}$ can be computed explicitly similarly as in \cite{FJ10}. Specifically, we multiply both sides of (\ref{rr}) by $(1,\varepsilon(k))^T$ and integrate w.r.t. $k$. Utilizing the conservation properties of $\colQ_{\text{el}}$, $\colQ_{\text{ee}}$, and the BGK operator, we get
\begin{eqnarray*} 
\int_{\mathbb{R}^d} \!\!r^{n+1} \!\! \left(\begin{array}{c} 1\\ \varepsilon \end{array} \right)\,\!\!dk=\int_{\mathbb{R}^d}\!\! r^n \!\!\left(\begin{array}{c} 1\\ \varepsilon \end{array} \right)\,\!\!dk \!-\! \Delta t \int_{\mathbb{R}^d}\!\!(\dkn \varepsilon  \! \cdot \! \dxn j^n \!+\!  \dxn V^n  \! \cdot \! \dkn j^n) \left(\begin{array}{c} 1 \\ \varepsilon \end{array} \right)\,\!\! dk.
\end{eqnarray*}
Note that by definition
\begin{eqnarray*} 
\int_{\mathbb{R}^d}r \left(\begin{array}{c} 1 \\ \varepsilon \end{array} \right)\,dk=\int_{\mathbb{R}^d}f \left(\begin{array}{c} 1 \\ \varepsilon \end{array} \right)\,dk=\left(\begin{array}{c}\rho \\ \rho\energy \end{array} \right),
\end{eqnarray*}
so the preceding scheme just gives an evolution of the macroscopic variables 
\begin{equation*} 
 \left(\begin{array}{c} \rho^{n+1}\\ \rho^{n+1}\energy^{n+1} \end{array} \right)= \left(\begin{array}{c} \rho^n\\ \rho^n\energy^n \end{array} \right)-\Delta t \int_{\mathbb{R}^d}(\dkn \varepsilon  \cdot  \dxn j^n +  \dxn V^n  \cdot  \dkn j^n) \left(\begin{array}{c} 1 \\ \varepsilon \end{array} \right)\,\! dk.
\end{equation*}
Once $\rho^{n+1}$ and $\energy^{n+1}$ are computed, we can invert the system (\ref{system}) to get $z^{n+1}$ and $T^{n+1}$ (details see \cite{HJ}). Plugging them into (\ref{FD}) then defines $M^{n+1}$. Given $\rho^{n+1}$, $V^{n+1}$ can be easily obtained by solving the Poisson equation (\ref{eqn: Poisson}).
\item {\it Step 2}: Compute $r^{n+1}$.

At every spatial point $x$, check the norm of $\colQ_{\text{el}}(r^n)$ in $k$: 
\begin{itemize}
\item if $\|\colQ_{\text{el}}(r^n)\|(x)>\delta$, apply scheme (\ref{rr});
\item otherwise, apply (\ref{rr}) with $\colQ_{\text{el}}(r^n)(x)=\beta_{\text{el}}=0$.
\end{itemize}
\item {\it Step 3}: Employ scheme (\ref{jj}) to get $j^{n+1}$.
\item {\it Step 4}: Reconstruct $f^{n+1}=(f^+)^{n+1}=r^{n+1}+\alpha j^{n+1}$ and $(f^-)^{n+1}=r^{n+1}-\alpha j^{n+1}$.
\end{itemize}


\subsubsection{Space discretization}

We finally include the spatial discretization to the previous semi-discrete scheme to construct a fully-discrete scheme. We will show that when $\alpha\rightarrow 0$, it automatically becomes a discretization for the limiting ET model (\ref{ET3}), thus is Asymptotic Preserving (satisfies the property 2 in the Introduction). 

For the sake of brevity, we will present the method in a splitting framework, namely, separating the explicit and implicit parts in (\ref{rr}--\ref{jj}). This is equivalent to an unsplit version since our scheme is of an IMEX type. We also assume a slab geometry: $x\in\Omega \subset \mathbb{R}^1$. Extension to higher dimensions is straightforward. 

Let $r^{n}_{l,m}$, $j_{l,m}^n$ denote the numerical approximation of $r(x_l, k_m, t^n)$ and $j(x_l, k_m, t^n)$, where $0<l \leq N_x$, $0<m=(m_1,...,m_d)\leq N_k^d$, $N_x$ and $N_k$ are the number of points in $x$ and $k$ directions respectively. We have at time step $t^{n+1}$:

\begin{itemize}
\item {\it Step 1}: Solve the explicit part of (\ref{rr}--\ref{jj}) to get $r_{l,m}^*$ and $j_{l,m}^*$. The thresholding idea is embedded in this step when computing $G_1(r^n_{l,m},j^n_{l,m})$: if $\|\colQ_{\text{el}}(r^n)\|<\delta$ at $x_l$, set $\colQ_{\text{el}}(r^n)(x_l)=0$. 

For convection terms, we use the upwind scheme with a slope limiter \cite{Leveque}. To determine the upwind flux, one first needs to transform $r$ and $j$ into Riemann invariants. Let $u = r + \frac{1}{\sqrt{\theta}} j$, $v = r - \frac{1}{\sqrt{\theta}}j$, then (note that indices $m\pm 1$ and $m\pm \half$ below refer to the shifts in the first component of $m$)
\begin{eqnarray*}
&&\frac{u^*_{l,m} - u^n_{l,m}}{\Dt} + \vel \left[   k_{m_1}  \frac{u_{l+\half,m}^n-u_{l-\half,m}^n}{\Dx}   + (\partial_x V)^n_l  \frac{u_{l,m+\half}^n - u_{l,m-\half}^n}{\Dk}   \right] \nonumber
 \\&& \hspace{5cm} = G_1(r_{l,m}^n, j_{l,m}^n) + \frac{1}{\sqrt{\theta}} G_2(r_{l,m}^n, j_{l,m}^n), 
\\ &&  \frac{v^*_{l,m}-v^n_{l,m}}{\Dt}- \vel \left[   k_{m_1} \frac{v_{l+\half,m}^n-v_{l-\half,m}^n}{\Dx}  + (\partial_x V)^n_l \frac{v_{l,m+\half}^n - v_{l,m-\half}^n}{\Dk}   \right] 
\nonumber 
\\ && \hspace{5cm} = G_1(r_{l,m}^n, j_{l,m}^n) - \frac{1}{\sqrt{\theta}} G_2(r_{l,m}^n, j_{l,m}^n),   
\end{eqnarray*}
where $\partial_x V$ is discretized by a central difference:
\begin{equation*}
(\partial_xV)^n_l:=\frac{V^n_{l+1}-V^n_{l-1}}{2\Dx}.
\end{equation*}
The fluxes are defined as (superscript $n$ are neglected)
\begin{eqnarray*}
&&u_{l+\half,m} = \left\{ \begin{array}{ll}  u_{l,m} + \half \phi\left(\frac{u_{l,m}-u_{l-1,m}}{u_{l+1,m}-u_{l,m}}\right) (u_{l+1,m} - u_{l,m}), &   k_{m_1}>0,  
                                                            \\ u_{l+1,m} - \half \phi \left(\frac{u_{l+2,m}-u_{l+1,m}}{u_{l+1,m}-u_{l,m}}\right) (u_{l+1,m} - u_{l,m}), &   k_{m_1}<0,
                               \end{array} \right.      \\
&& u_{l,m+\half} = \left\{ \begin{array}{ll}  u_{l,m} +\half \phi\left(\frac{u_{l,m}-u_{l,m-1}}{u_{l,m+1}-u_{l,m}}\right) (u_{l,m+1} - u_{l,m}), &  (\partial_x V)_l>0,  
                                                            \\ u_{l,m+1} -  \half \phi \left(\frac{u_{l,m+2}-u_{l,m+1}}{u_{l,m+1}-u_{l,m}}\right) (u_{l,m+1} - u_{l,m}), &   (\partial_x V)_l<0,
   \end{array} \right.    
\end{eqnarray*}
here $\phi(\sigma)$ is a slope limiter function, e.g., the minmod limiter is given by $\phi(\sigma)=\max(0,\min(1,\sigma))$. Fluxes $v_{l+\half,m}$, $v_{l,m+\half}$ are defined similarly. Upon obtaining $u_{l,m}^*$ and $v_{l,m}^*$, $r_{l,m}^*$ and $j_{l,m}^*$ are recovered by
\begin{equation*}
r_{l,m}^* =\half(u_{l,m}^* + v_{l,m}^*), \quad j_{l,m}^* =\frac{\sqrt{\theta}}{2}(u_{l,m}^* - v_{l,m}^*).
\end{equation*}

\item {\it Step 2}: Solve for the macroscopic quantities $\rho^*_l$ and $\energy^*_l$ and thus define $M^*_{l,m}$ and $V^*_{l,m}$.
\begin{equation*}
\rho_l^*=\sum_m r^*_{l,m}\Delta k^d, \quad \energy_l^*=\frac{1}{2}\sum_m r^*_{l,m}|k_m|^2 \Delta k^d/\rho_l^*.
\end{equation*}
Finding $M^*_{l,m}$ is then exactly the same as in the semi-discrete scheme. $V^*_{l,m}$ is solved from a simple finite-difference discretization of the Poisson equation.

\item {\it Step 3}: Solve the implicit part of (\ref{rr}--\ref{jj}) to get $r^{n+1}_{l,m}$ and $j^{n+1}_{l,m}$ (if the threshold is satisfied in {\it Step 1}, set $\beta_{\text{el}}=0$ in $r$'s equation as well):
\begin{eqnarray}
\quad \frac{r^{n+1}_{l,m}-r^*_{l,m}}{\Delta t}=\ &&  \frac{\beta_{\text{el}} (M^{n+1}_{l,m}-r^{n+1}_{l,m})}{\alpha^2}+\frac{\beta_{\text{ee}} (M^{n+1}_{l,m}-r^{n+1}_{l,m})}{\alpha}, \label{rrimp}\\ 
\quad \frac{j^{n+1}_{l,m}-j^*_{l,m}}{\Dt}=\ &&-\frac{1}{\alpha^2}\left[ \lambda_{\text{el}} j^{n+1}_{l,m}  +  (1 -  \alpha^2 \theta) \left(k_{m_1}   \frac{r^{n+1}_{l+1,m}-r^{n+1}_{l-1,m}}{2\Delta x} \right. \right. \nonumber\\
&& \left.\left. + (\partial_x V)^{n+1}_l \frac{r^{n+1}_{l,m+1}-r^{n+1}_{l,m-1}}{2\Delta k}\right)\right] - \frac{\beta_{\text{ee}}}{\alpha} j^{n+1}_{l,m}.\label{jjimp}
\end{eqnarray}
First, it is easy to see that macroscopic quantities $ \rho^*_l$ and $\energy^*_l$ remain unchanged during this step (the right hand side of (\ref{rrimp}) is conservative). Therefore, the previously obtained $M^*_{l,m}$ and $V^*_{l,m}$ are in fact $M^{n+1}_{l,m}$ and $V^{n+1}_{l,m}$. From (\ref{rrimp}) one can easily obtain $r^{n+1}_{l,m}$, and then (\ref{jjimp}) directly gives rise to $j^{n+1}_{l,m}$.
\end{itemize}


\subsubsection{Asymptotic properties of the fully discrete scheme}

As already shown in Section 3.2.1, sending $\alpha$ to zero in (\ref{rr}--\ref{jj}) for $\colQ_{\text{ee}}=M-f$ leads to 
\begin{eqnarray*}
&& r^{n+1} = M^{n+1},
\\&& j^{n+1}= -\frac{1}{\lambda_{\text{el}}} \left(  \dkn \varepsilon \cdot  \dxn  r^{n+1} + \dxn V^{n+1} \cdot \dkn r^{n+1} \right),
\end{eqnarray*}
which in 1-D fully discrete form read
\begin{eqnarray*}
&&r_{l,m}^{n+1} = M_{l,m}^{n+1},\\
&&j_{l,m}^{n+1} = -\frac{1}{\lambda_{\text{el}}} \left(  k_{m_1}  \frac{ r_{l+1,m}^{n+1} -r_{l-1,m}^{n+1} }{2\Dx} + (\partial_x V)_l^{n+1}  \frac{ r^{n+1}_{l,m+1} - r^{n+1}_{l,m-1} }{2\Dk} \right).
\end{eqnarray*}
Plugging these relations into the discrete scheme of $r$ resulted from the last subsection, we get (after multiplication by $(1, \varepsilon(k))^T$ and integration w.r.t. $k$):
\begin{eqnarray} \label{scheme: ET}
&&\frac{1}{\Delta t} \left( \begin{array}{c} \rho^{n+1} - \rho^n\\  \rho^{n+1}\energy^{n+1} - \rho^n \energy^{n}\end{array}\right)  -  \int_{\mathbb{R}^d} \frac{1}{\lambda_{\text{el}}(\varepsilon)}   \left\{ \frac{ k_{m_1}}{2\Dx}
\left[  k_{m_1} \frac{ M_{l+2,m}^n - 2M_{l,m}^n + M_{l-2,m}^n }{2\Dx} \right.\right.\nonumber
\\&& \hspace{0cm} \left. \left.  +(\partial_x V)_{l+1}^n  \frac{M_{l+1,m+1}^n - M^n_{l+1, m-1}}{2\Dk}  -  (\partial_x V)_{l-1}^n  \frac{M_{l-1,m+1}^n - M^n_{l-1, m-1}}{2\Dk}   \right] \right.  \nonumber
\\&& \hspace{0cm} \left. -\frac{ |k_{m_1}|}{2\Dx} \left( M_{l+1,m}^n - 2M_{l,m}^n + M_{l-1,m}^n  \right)+  \frac{(\partial_x V)^n_l}{2\Dk}  \left[  (\partial_x V)_l^n  \frac{M_{l,m+2}^n - 2M_{l,m}^n + M_{l,m-2}^n }{2\Dk}  \right.\right.\nonumber
\\&&  \hspace{0cm} \left. \left. + k_{m_1+1}  \frac{M_{l+1,m+1}^n - M_{l-1,m+1}^n}{2\Dx}-  k_{m_1-1}  \frac{M_{l+1,m-1}^n - M_{l-1,m-1}^n}{2\Dx}  \right.\right] \nonumber
\\&&   \hspace{1cm} \left. - \frac{ |(\partial_x V)_l^n  |}{2\Dk}  \left( M_{l,m+1}^n - 2M_{l,m}^n + M_{l,m-1}^n  \right)    \right\} \left(\begin{array}{c} 1 \\ \varepsilon  \end{array} \right)\,  dk = 0,
\end{eqnarray}
which is a {\it kinetic} scheme \cite{Deshpande86} for the ET model (\ref{ET3}) (compare with (\ref{ET21}) and (\ref{f1})).  Here for notation simplicity, we only consider the upwind scheme, the slope limiter can be added in the same manner. 


\subsection{Spectral methods for the collision operators $\colQ_{\text{el}}$ and $\colQ_{\text{ee}}$} \label{sec: collision}

In this subsection, we briefly outline the spectral methods for computing the collision operators $\colQ_{\text{el}}$ and $\colQ_{\text{ee}}$. For numerical purpose, we assume the scattering matrices $\Phi_{\text{el}}(k,k')= \Phi_{\text{ee}}(k,k_1,k',k_1')\equiv1$, and the wave vector $k\in \mathbb{R}^2$.

\subsubsection{Computing $\colQ_{\text{ee}}$}

Under the parabolic band assumption, (\ref{Qee}) reads
\begin{eqnarray*} 
\colQ_{\text{ee}}(f)(k) =\ && \int_{\mathbb{R}^{6}} \delta (k' + k_1' -k - k_1)\delta\left(\frac{k'^2}{2}+\frac{k_1'^2}{2}-\frac{k^2}{2}-\frac{k_1^2}{2}\right) \nonumber
\\  &&  \times \Big[  f'f_1'(1-\eta f)(1- \eta f_1) - ff_1(1- \eta f')(1- \eta f_1')  \Big]\,  dk_1 dk' dk_1'.
\end{eqnarray*} 
By a change of variables, it is not difficult to rewrite the above integral in the center of mass reference system:
\begin{equation} \label{CMR}
\quad \quad \quad \colQ_{\text{ee}}(f)(k) =\frac{1}{2}\int_{\mathbb{R}^{2}}\int_{\mathbb{S}^1} \Big[   f'f_1'(1-\eta f)(1- \eta f_1) - ff_1(1- \eta f')(1- \eta f_1')  \Big]\,  d\sigma\,dk_1,
\end{equation} 
where 
\begin{eqnarray*}
\left\{
\begin{array}{l}
\displaystyle k'=\frac{k+k_1}{2}+\frac{|k-k_1|}{2}\sigma, \\
\displaystyle k_1'=\frac{k+k_1}{2}-\frac{|k-k_1|}{2}\sigma.
\end{array}\right.
\end{eqnarray*}
This is just the usual form of the quantum Boltzmann collision operator for a Fermi gas (of 2-D Maxwellian molecules). Our way of computing $\colQ_{\text{ee}}$ follows the fast spectral method in \cite{HY12}. The starting point is to further transform (\ref{CMR}) to a Carleman form
\begin{equation} \label{Carleman}
 \colQ_{\text{ee}}(f) (k)=\int_{\mathbb{R}^{2}}\int_{\mathbb{R}^2} \delta(x\cdot y)\Big[   f'f_1'(1-\eta f)(1- \eta f_1) - ff_1(1- \eta f')(1- \eta f_1')  \Big]\,  dx\,dy,
\end{equation} 
where $k_1=k+x+y$, $k'=k+x$, and $k_1'=k+y$. If $\text{Supp}(f(k))\subset \mathcal{B}_S$ (a ball with radius $S$), we can truncate (\ref{Carleman}) as
\begin{eqnarray*} 
\colQ_{\text{ee}}^R(f)(k) &&=\int_{\mathcal{B}_R}\int_{\mathcal{B}_R} \delta(x\cdot y)\Big[   f'f_1'(1-\eta f)(1- \eta f_1) - ff_1(1- \eta f')(1- \eta f_1')  \Big]\,  dx\,dy \nonumber\\
&&=\int_{\mathcal{B}_R}\int_{\mathcal{B}_R} \delta(x\cdot y)\Big[ (f'f_1'-ff_1)-\eta(f'f_1'f_1+f'f_1'f-f'f_1f-f_1f_1'f) \Big]\,  dx\,dy
\end{eqnarray*}
with $R=2S$. We next choose a computational domain $\mathcal{D}_L=[-L,L]^2$ for $k$, and extend the function $f(k)$ periodically to the whole space $\mathbb{R}^2$. $L$ is chosen such that $L\geq \frac{3\sqrt{2}+1}{2}S$ to avoid aliasing effect. We then approximate $f(k)$ by a truncated Fourier series
\begin{equation} \label{FS}
f(k)=\sum_{j=-N/2}^{N/2-1}\hat{f}_j e^{i\frac{\pi}{L}j\cdot k},
\end{equation}
where
\begin{equation*}
\hat{f}_j=\frac{1}{(2L)^2}\int_{\mathcal{D}_L}f(k) e^{-i\frac{\pi}{L}j\cdot k}\,dk.
\end{equation*}
Inserting the Fourier expansion of $f$ into $\colQ_{\text{ee}}^R(f)$, and performing a spectral-Galerkin projection, we can get the governing equation for $\widehat{\colQ_{\text{ee}}^R}(f)_j$. The computation is sped up by discovering a convolution structure and a separated expansion of the coefficient matrix. The final cost is roughly $O(MN^3\log N)$, where $N$ is the number of points in each $k$ direction, and $M$ is the number of angular discretization. More details can be found in \cite{HY12}. 

\subsubsection{Computing $\colQ_{\text{el}}$}

Under the parabolic band assumption (\ref{para}), (\ref{Qel}) reads (see also (\ref{Qel3}))
\begin{equation} \label{Qel1}
\colQ_{\text{el}}(f)(k) = \int_{\mathbb{R}^2}\delta \left(  \varepsilon' - \varepsilon  \right) \left( f' - f\right) \,dk' = \int_{\mathbb{S}^1}f(|k|\sigma)\,d\sigma-2\pi f(k).
\end{equation}
Compared to $\colQ_{\text{ee}}$, this one is much easier to compute. For instance, one can do a direct numerical quadrature plus interpolation to approximate the integral over $\mathbb{S}^1$. To achieve better accuracy, we here present an efficient way to compute $\colQ_{\text{el}}(f)$ based on the same spectral framework of $\colQ_{\text{ee}}(f)$. Specifically, we still adopt the Fourier expansion (\ref{FS}). Inserting it into (\ref{Qel1}), we get
\begin{equation} \label{Qel2}
\colQ_{\text{el}}(f)(k)=\sum_{j=-N/2}^{N/2-1} B(k,j)\hat{f}_j-2\pi f(k),
\end{equation}
where 
\begin{equation*}
B(k,j)=2\pi J_0\left(\frac{\pi}{L}|k||j|\right).
\end{equation*}
Here $L$ has to be $L\geq \frac{\sqrt{2}+2}{2}S$ to avoid aliasing. A direct computation of the above summation requires obviously $O(N^4)$ flops, which can be quite costly. But note that the coefficient matrix $B(k,j)$ only depends on the magnitude of $k$ and $j$, which means that its rank is roughly $O(N)$. Therefore, we can find a low rank decomposition of $B(k,j)$ as (for 2-D problems this can be precomputed via a SVD)
\begin{equation*}
B(k,j)=\sum_{r=1}^{O(N)}U_r(k)V_r(j).
\end{equation*}
Then computing the summation in (\ref{Qel2}) becomes
\begin{equation*}
\sum_{j=-N/2}^{N/2-1} B(k,j)\hat{f}_j=\sum_{j=-N/2}^{N/2-1} \sum_{r=1}^{O(N)} U_r(k)V_r(j)\hat{f}_j=\sum_{r=1}^{O(N)} U_r(k)\left [\sum_{j=-N/2}^{N/2-1} V_r(j)\hat{f}_j\right].
\end{equation*}
The cost is reduced to $O(N^3)$.

To summarize, we have two fast spectral solvers for the collision operators $\colQ_{\text{el}}$ and $\colQ_{\text{ee}}$ (the cost of $\colQ_{\text{ee}}$ is dominant due its intrinsic complexity). They both provide high-order accuracy and are thus suitable for testing asymptotic properties of our scheme.

\section{Numerical examples}

In this section, we present several numerical examples using our AP schemes. The wave vector $k$ is assumed to be 2-D: $[-L_k,L_k]^2$ is the computational domain and $N_k$ is the number of points in each $k$ direction. The space is 1-D: $x\in [0,L_x]$ and $N_x$ is the number of spatial discretization.


\subsection{The spatially homogeneous case}

We first check the behavior of the solution in the spatially homogeneous case. Consider the nonequilibrium initial data
\begin{equation*}
f_0(k_1,k_2)=\frac{1}{\pi}\left[(k_1-1)^2+(k_2-0.5)^2\right]e^{-\left[(k_1-1)^2+(k_2-0.5)^2\right]}.
\end{equation*}
The parameters are chosen as $\alpha=1e-3$ (diffusive regime), $\eta=10$ (strong quantum effect; the equilibrium is very different from the classical Maxwellian), $L_k=10.5$, $N_k=64$. Under this condition, a stable explicit scheme would require $\Delta t= O(1e-6)$, while our scheme gives fairly good results with much coarser time step $\Delta t=1$. Figure \ref{fig: errorAP1} shows the AP error in $L^\infty$-norm
\begin{equation} \label{errorAP1}
\text{errorAP}_{L^\infty}^n = \max_{k_1,k_2} | f^n - M^n|
\end{equation}
with time. Here we can see that during the initial period of time, this error decreases very slowly as explained in (\ref{radial}): $f$ is only driven to some function of $\varepsilon$, not $M$, because of the  dominating mechanism $\Qel$. Once the threshold comes into play as shown in (\ref{fastconv}), $f$ will start to converge to $M$ at a reasonable speed, which appears as a sharp transition in dashed curve in Figure \ref{fig: errorAP1}. As a comparison, the solid curve is obtained by the regular AP scheme without threshold, the error decreases very slowly as in (\ref{oldAP1}). Figure \ref{fig: evolution} displays the evolution of $f$ at different times, where we clearly see that $f$ first transits from non-radially symmetric to radially symmetric and then moves toward the desired Fermi-Dirac distribution $M$.

\begin{figure}[!h]
        \centering
       \includegraphics[width=0.7\textwidth]{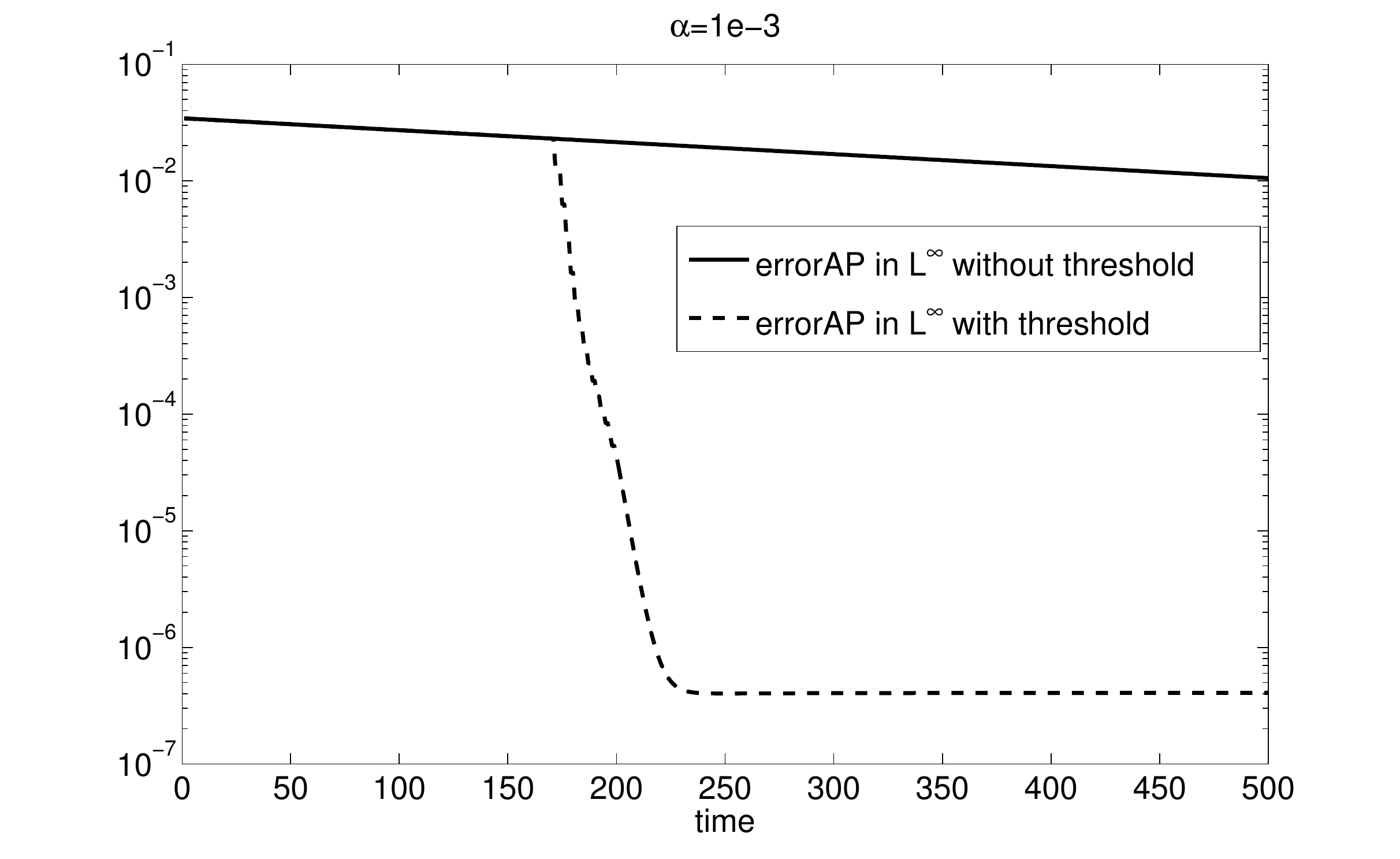}
 \caption{Plots of the asymptotic error (\ref{errorAP1}) versus time. Here $\eta=10$, $\Delta t=1$, $L_k=10.5$, and $N_k=64$.}
 \label{fig: errorAP1}
 \end{figure}
 
 \begin{figure}[!h]
  \begin{minipage}[t]{0.5\linewidth}
       \centering
       \includegraphics[width=1.1\textwidth]{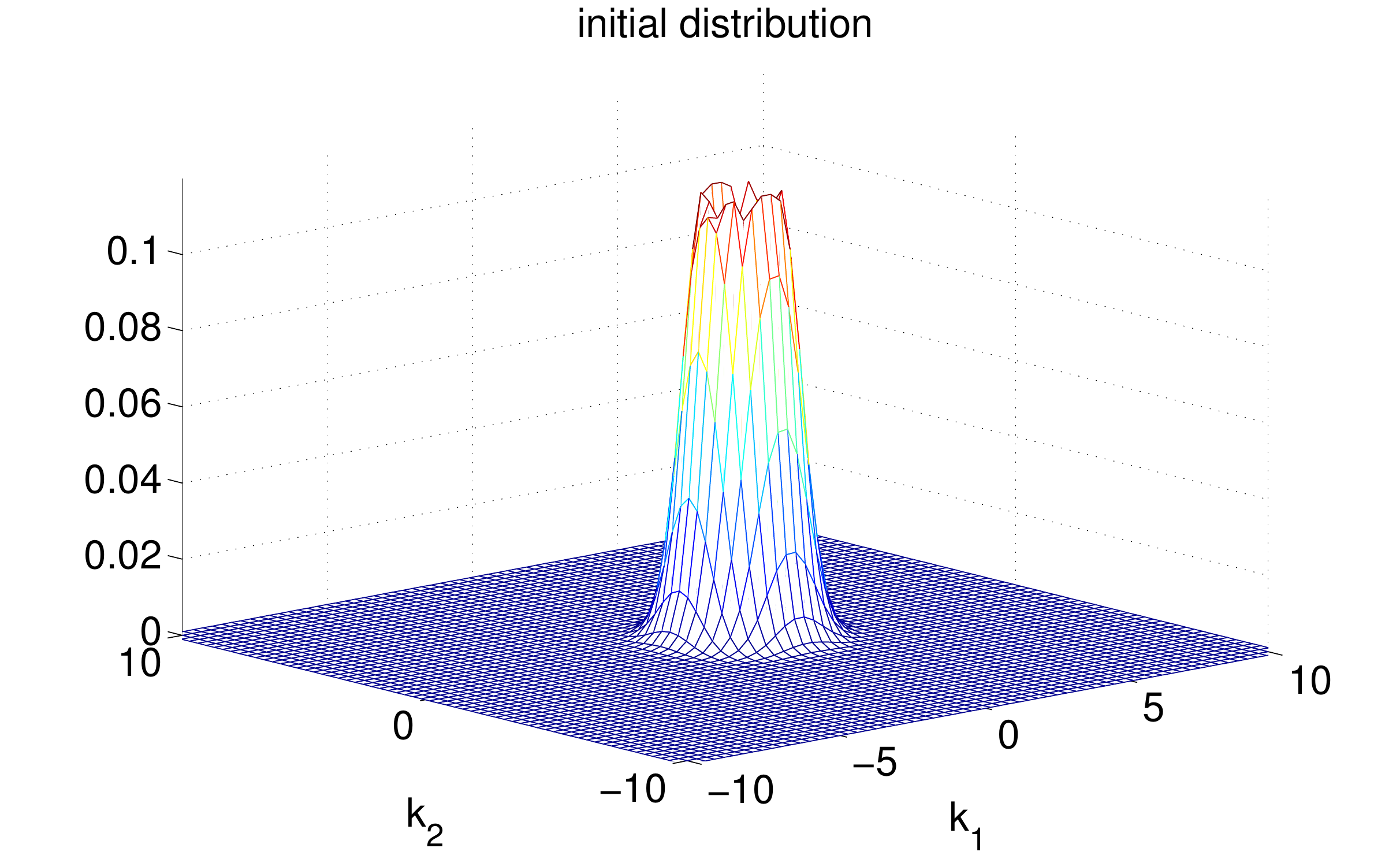}
 \end{minipage}
 \begin{minipage}[t]{0.5\linewidth}
      \centering
     \includegraphics[width=1.1\textwidth]{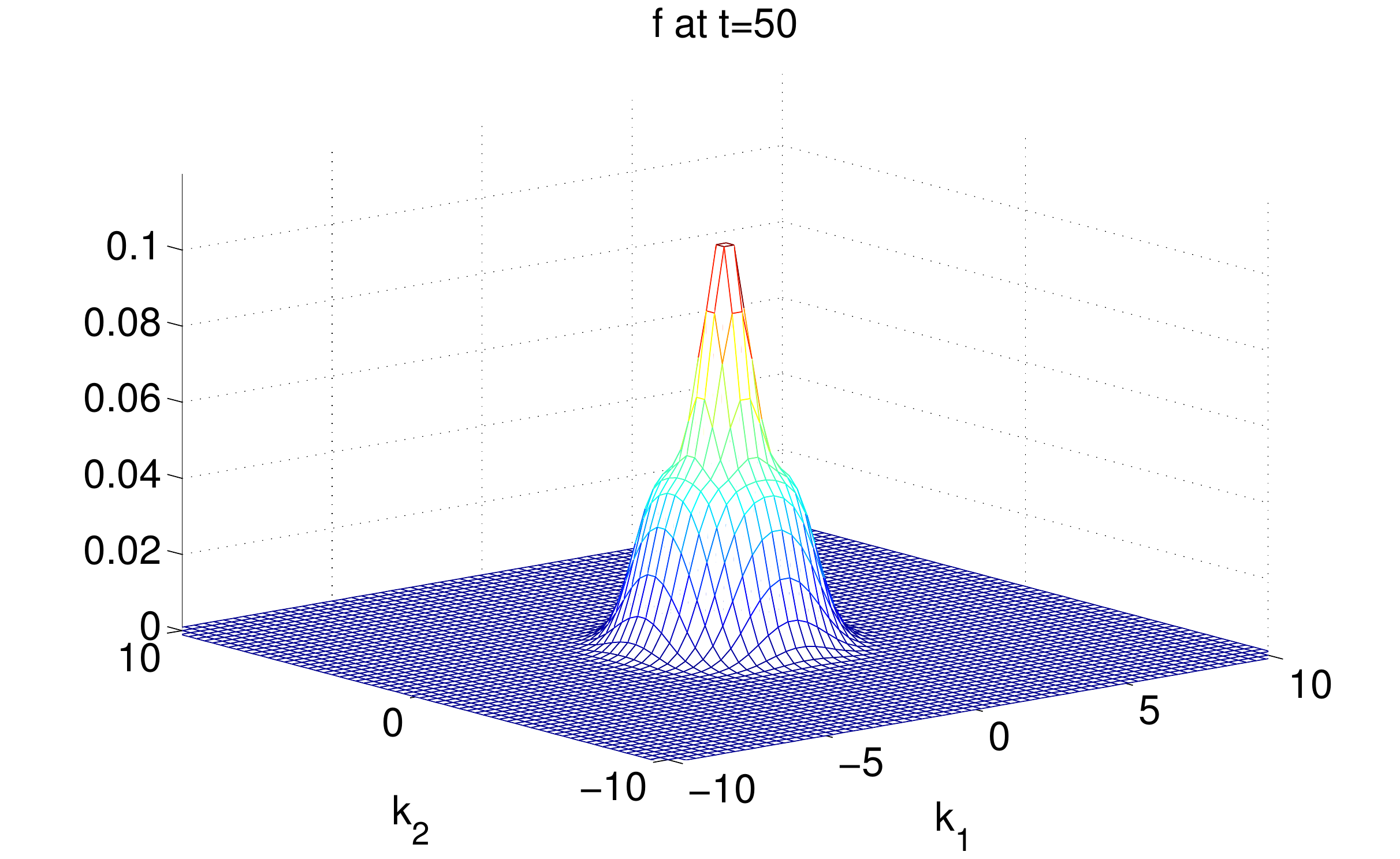}
  \end{minipage}
  \\
    \begin{minipage}[t]{0.5\linewidth}
       \centering
       \includegraphics[width=1.1\textwidth]{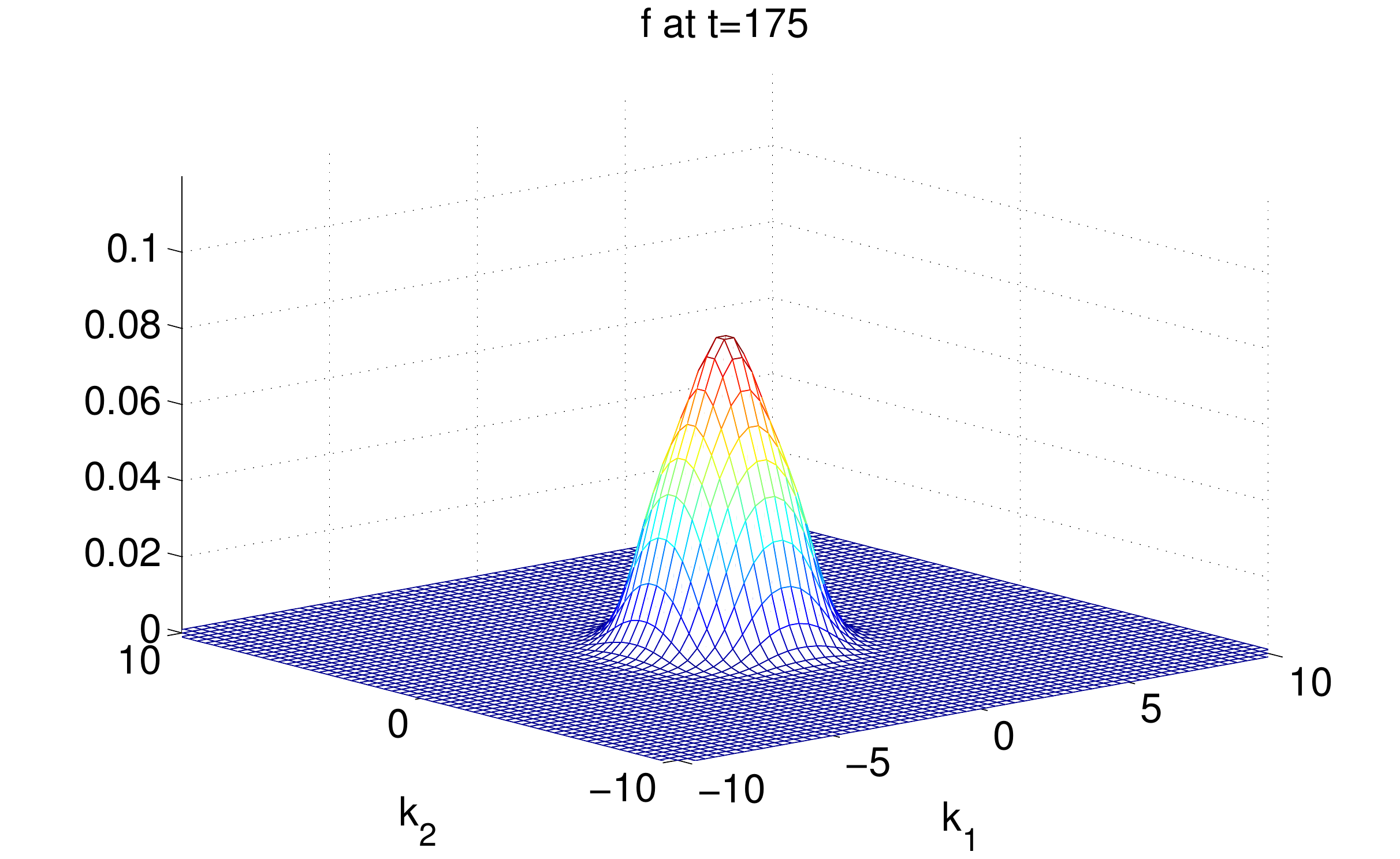}
 \end{minipage}
 \begin{minipage}[t]{0.5\linewidth}
      \centering
     \includegraphics[width=1.1\textwidth]{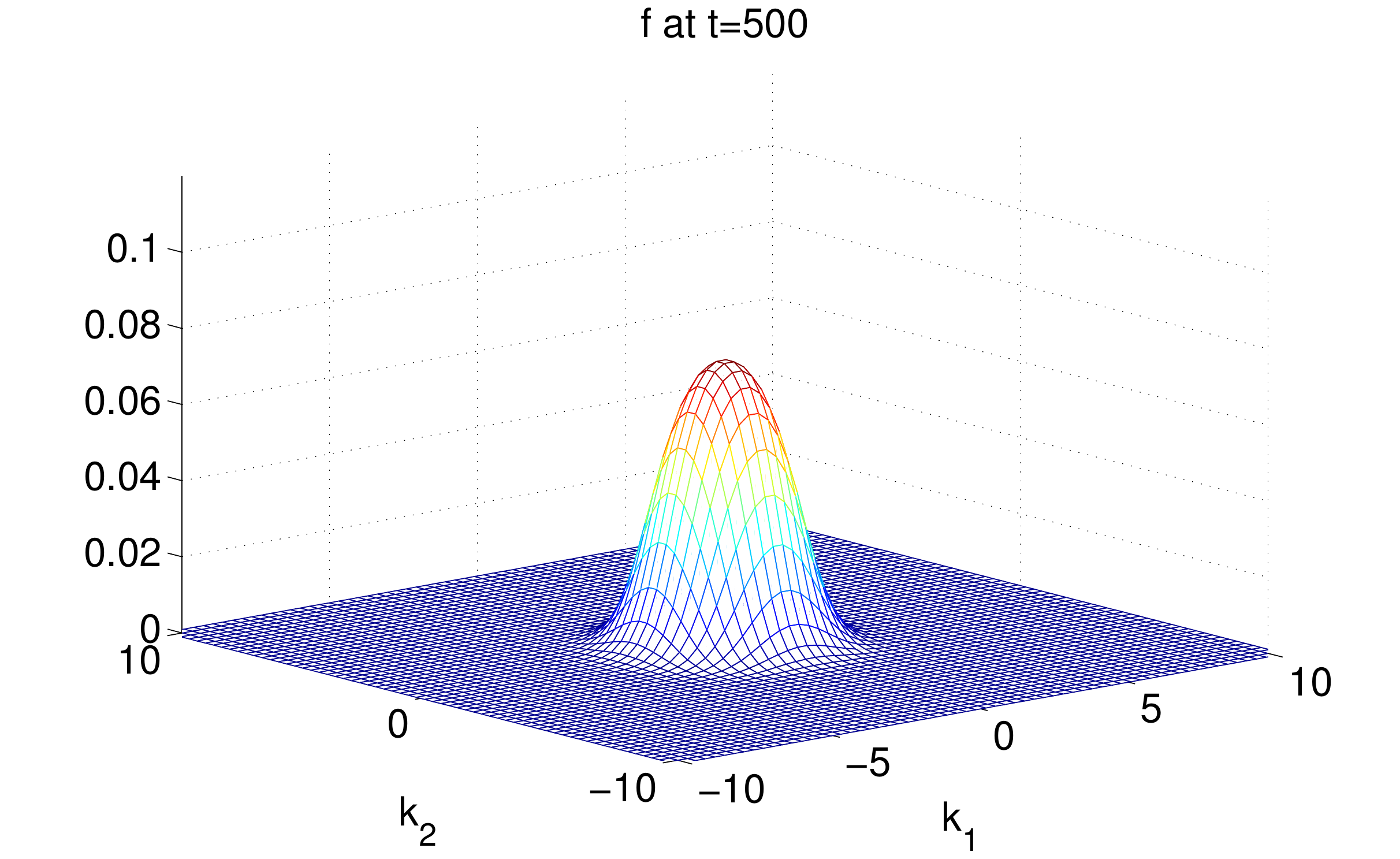}
  \end{minipage} 
 \caption{Evolution of $f$ at times $t=0, 50, 175$ and $500$. Here $\eta=10$, $\Delta t=1$, $L_k=10.5$, and $N_k=64$.}
 \label{fig: evolution}
 \end{figure}


\subsection{The spatially inhomogeneous case}

In the rest of simulation, we always take $L_x=1$, $L_k=9.2$, and assume periodic boundary condition in $x$ and zero boundary in $k$. We will consider both $\eta$ small which corresponds to the classical (nondegenerate) regime and $\eta$ not small which leads to the quantum (degenerate) regime. The time step $\Delta t$ is chosen to only satisfy the parabolic CFL condition: $\Delta t=O(\Delta x^2)$ (independent of $\alpha$).

\subsubsection{AP property}

Consider equation (\ref{kinetic1}) with nonequilibrium initial data 
\begin{equation*}
f_0(x,k_1,k_2) =\frac{1}{2\pi} \left( e^{-80(x-\frac{L_x}{2})^2}+1 \right) \left(   e^{- \left[(k_1-1)^2+k_2^2 \right]} + e^{- \left[(k_1+1)^2+k_2^2 \right] }  \right).
\end{equation*}
The electric field $\dx V$ is set to be one. 

We check the asymptotic property by looking at the distance between $r$ and $M$, i.e.,
\begin{equation} \label{eqn: errorAP}
\text{errorAP}_{L^1}^n = \sum_{x,k_1,k_2} | r^n - M^n| \Dx \Dk^2, \   \  
\text{errorAP}_{L^\infty}^n = \max_{x,k_1,k_2} | r^n - M^n|.
\end{equation}
The results are gathered in Figure \ref{fig: errorAP}, where we observe a similar trend as in the space homogeneous case that $r$ converges to $M$ in two stages: first to a radially symmetric function (some function of $\varepsilon$) and then the local Maxwellian $M$. This in some sense mimics the Hilbert expansion in the derivation of ET model in Section 2.

\begin{figure}[!h]
  \begin{minipage}[t]{0.5\linewidth}
       \centering
       \includegraphics[width=1.1\textwidth]{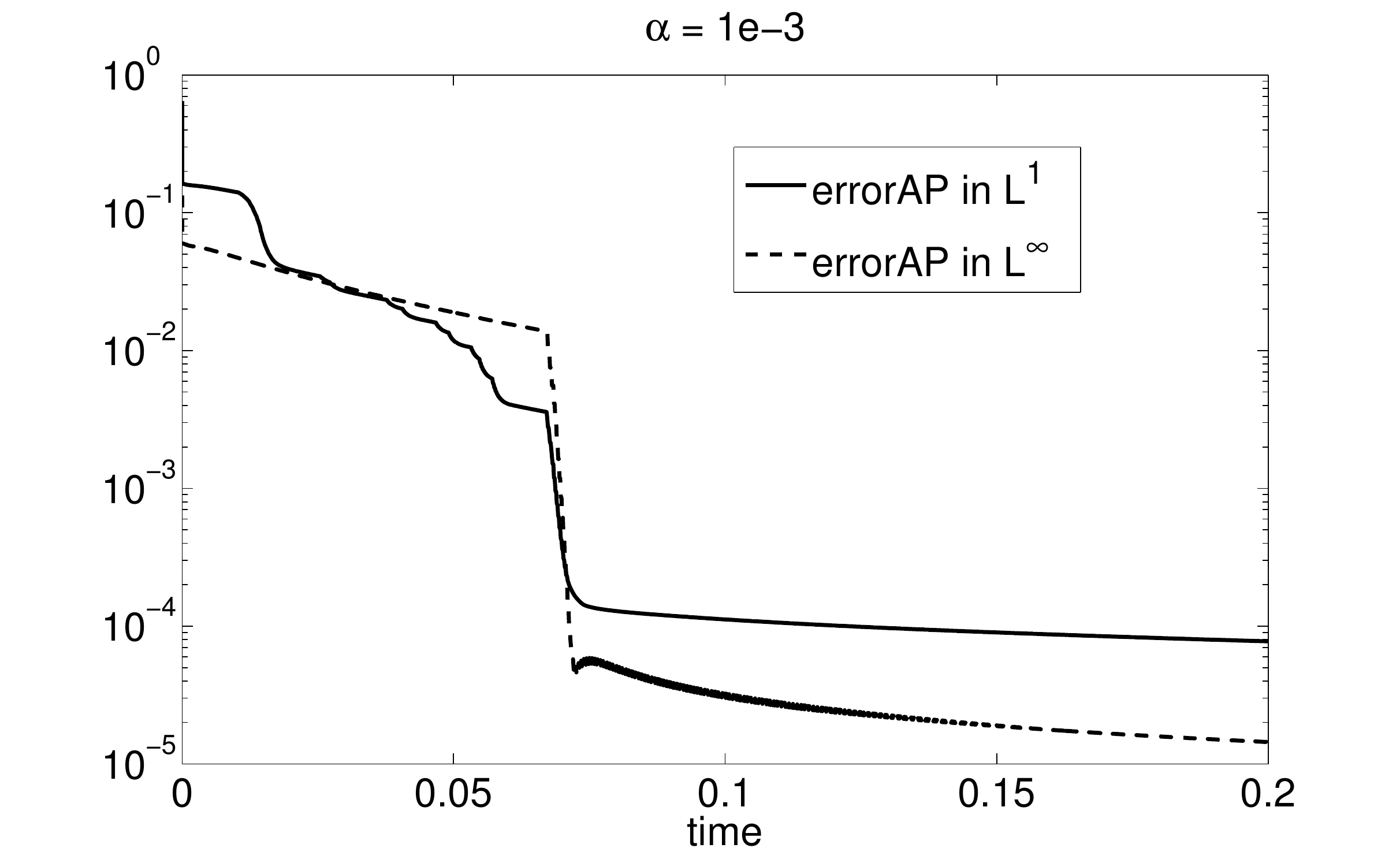}
 \end{minipage}
 \begin{minipage}[t]{0.5\linewidth}
      \centering
     \includegraphics[width=1.1\textwidth]{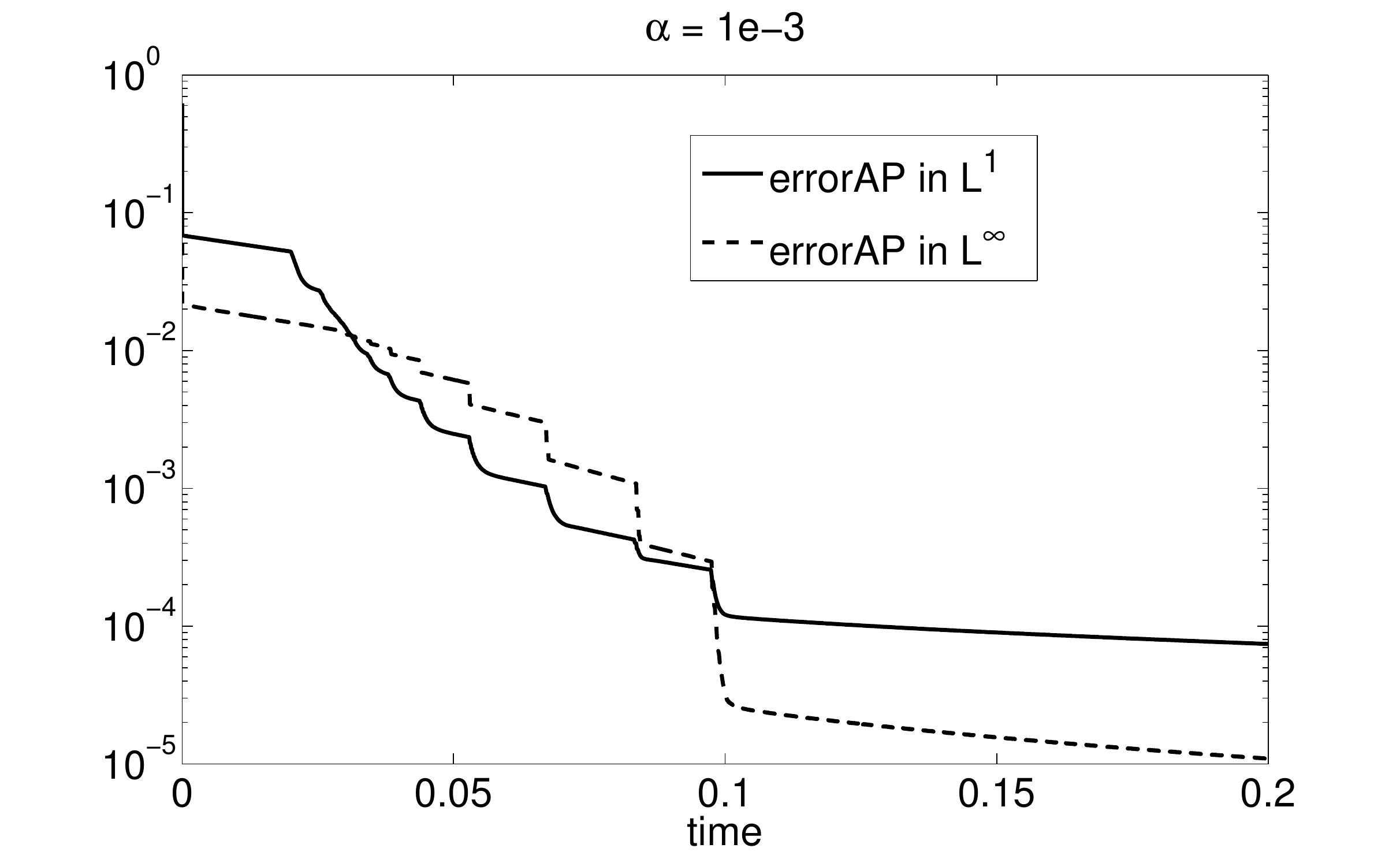}
  \end{minipage}
 \caption{Plots of the asymptotic error (\ref{eqn: errorAP}) versus time. Here $N_x=40$, $\Dt = 0.2 \Dx^2$. Left: $\eta = 0.01$ (classical regime), $N_k = 32$. Right: $\eta = 3$ (quantum regime), $N_k=64$.}
 \label{fig: errorAP}
 \end{figure}
 

\subsubsection{1-D $n^+$--$n$--$n^+$ ballistic silicon diode}

We finally simulate a 1-D $n^+$--$n$--$n^+$ ballistic silicon diode, which is a simple model of the channel of a MOS transistor. The initial data is taken to be
\begin{eqnarray*}
f_0 (x,k_1,k_2)=\ && \left( \!\!1.1+ \frac{ \tanh\left( 40(x-\frac{5L_x}{8})\right) - \tanh\left( 40(x-\frac{3L_x}{8}) \right) }{2}  \right) \nonumber
\\&& \times \left(   e^{- \left[(k_1-1)^2+k_2^2 \right]} + e^{- \left[(k_1+1)^2+k_2^2 \right] }  \right),
\end{eqnarray*}
For Poisson equation (\ref{eqn: Poisson}), we choose $h(x) =\rho_0(x)=\int f_0\,dk$, $\epsilon_r(x) \equiv1$, $C_0 = 1/1000$, with boundary condition $V(0) = 0, V(L_x) =1$. The doping profile $h(x)$ is shown in Figure \ref{fig: doping}. 

We consider two regimes:  one is the kinetic regime with $\alpha =1$, where we compare our solution with the one obtained by the explicit scheme (forward Euler); the other is the diffusive regime with $\alpha = 1e-3$, where our solution is compared with that of the ET system using the kinetic solver (\ref{scheme: ET}). Both $\eta = 1e-2$ and $\eta = 1$ are checked and good agreements are obtained in Figures \ref{fig: eta1alpha1}, \ref{fig: eta1e-2alpha1}, \ref{fig: eta1alpha1e-3}, \ref{fig: eta1e-2alpha1e-3}. Here the macroscopic quantities plotted are mass density $\rho$, internal energy $\energy$ defined in (\ref{rhoe}), electron temperature $T$ and fugacity $z$ obtained through (\ref{system}), electric field $E= - \dx V$, and mean velocity $u$ defined as $u = j_\rho/\rho$. 

\begin{figure}[!h]
\centering
\includegraphics[width = 0.6\textwidth]{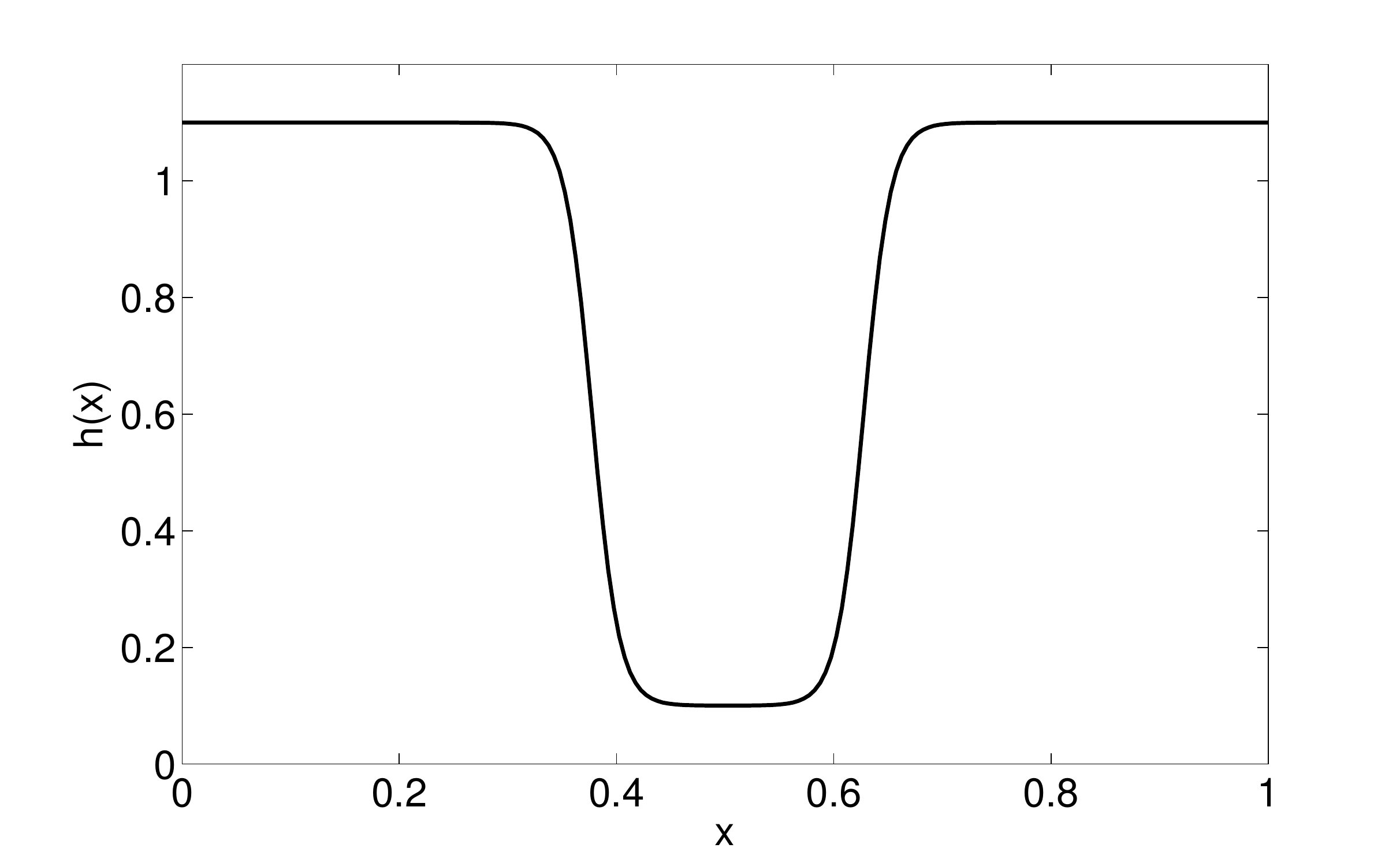}
\caption{Doping profile $h(x)$ for 1-D $n^+$--$n$--$n^+$ ballistic silicon diode.}
\label{fig: doping}
\end{figure}

 \begin{figure}[!h]
  \begin{minipage}[t]{0.5\linewidth}
       \centering
       \includegraphics[width=1.1\textwidth]{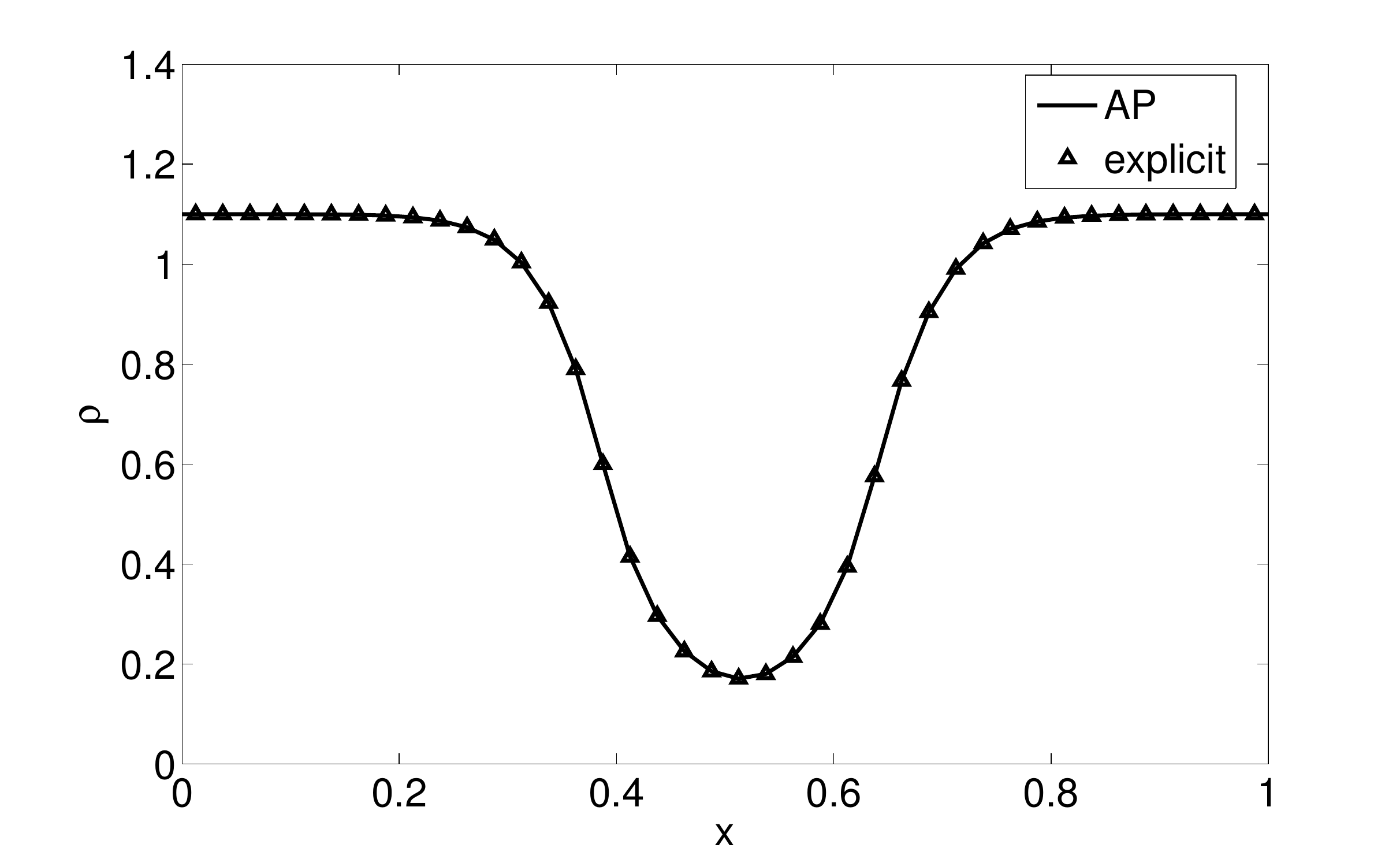}
 \end{minipage}
 \begin{minipage}[t]{0.5\linewidth}
      \centering
      \includegraphics[width=1.1\textwidth]{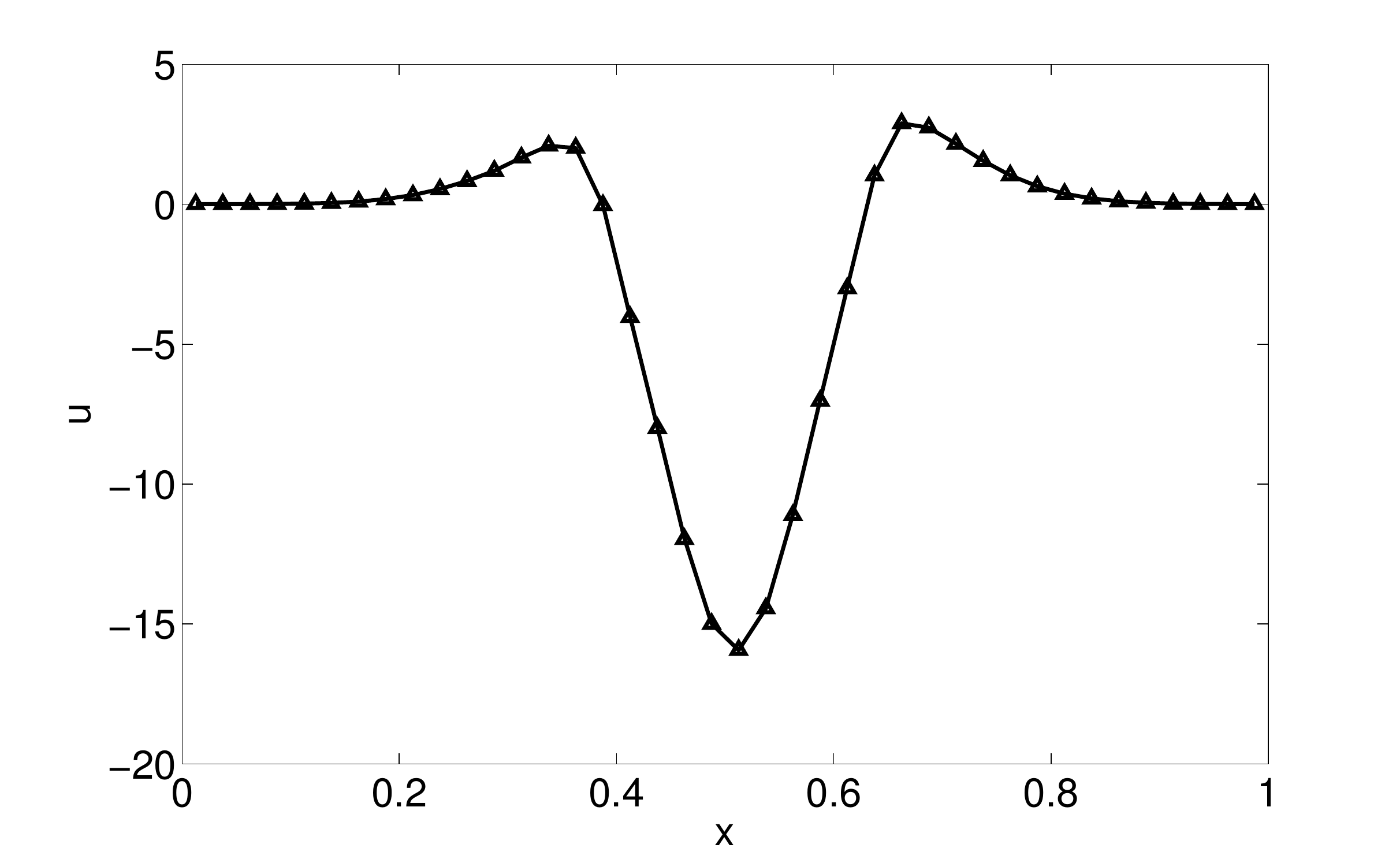}
  \end{minipage}
\\
  \begin{minipage}[t]{0.5\linewidth}
       \centering
       \includegraphics[width=1.1\textwidth]{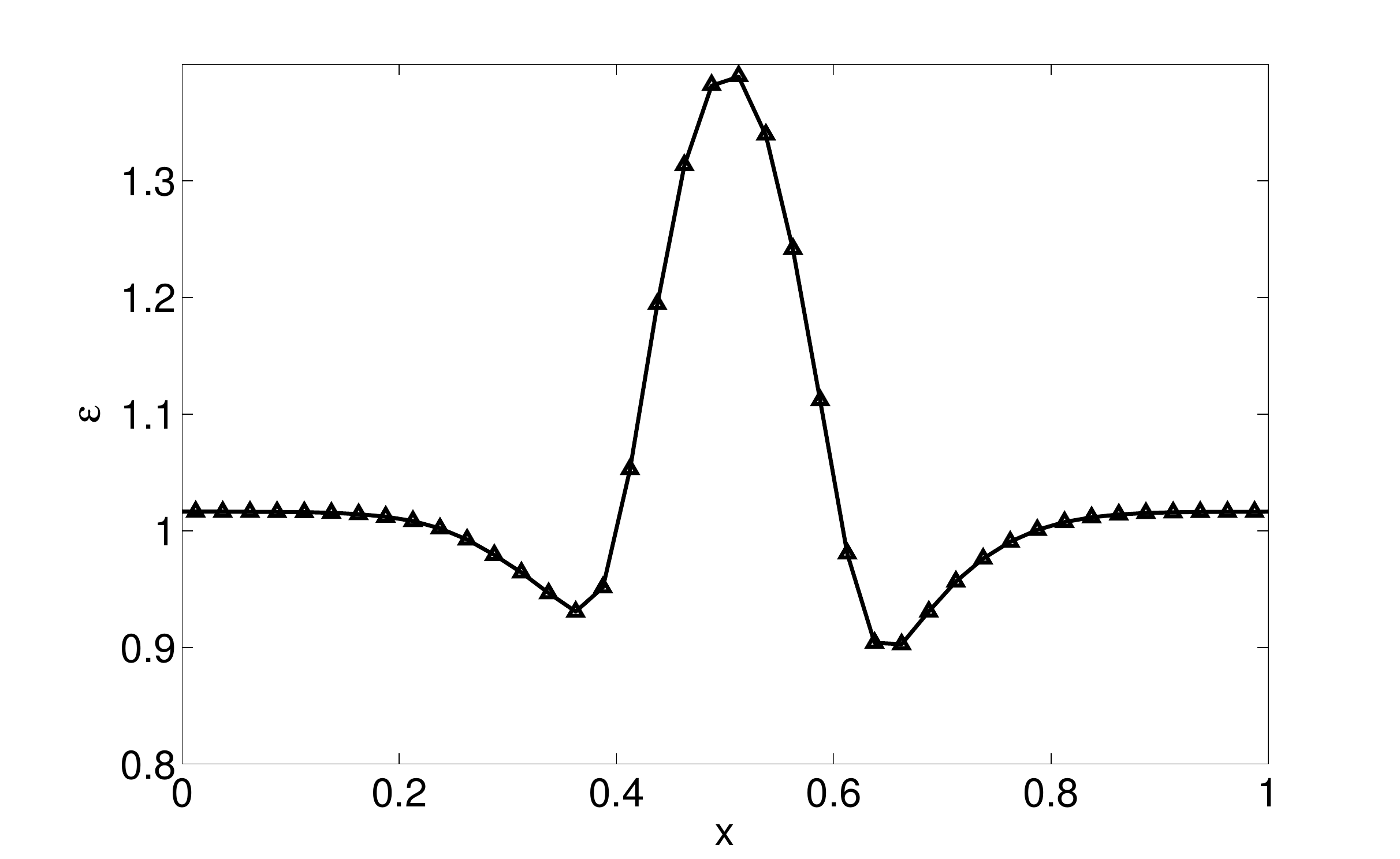}
 \end{minipage}
 \begin{minipage}[t]{0.5\linewidth}
      \centering
      \includegraphics[width=1.1\textwidth]{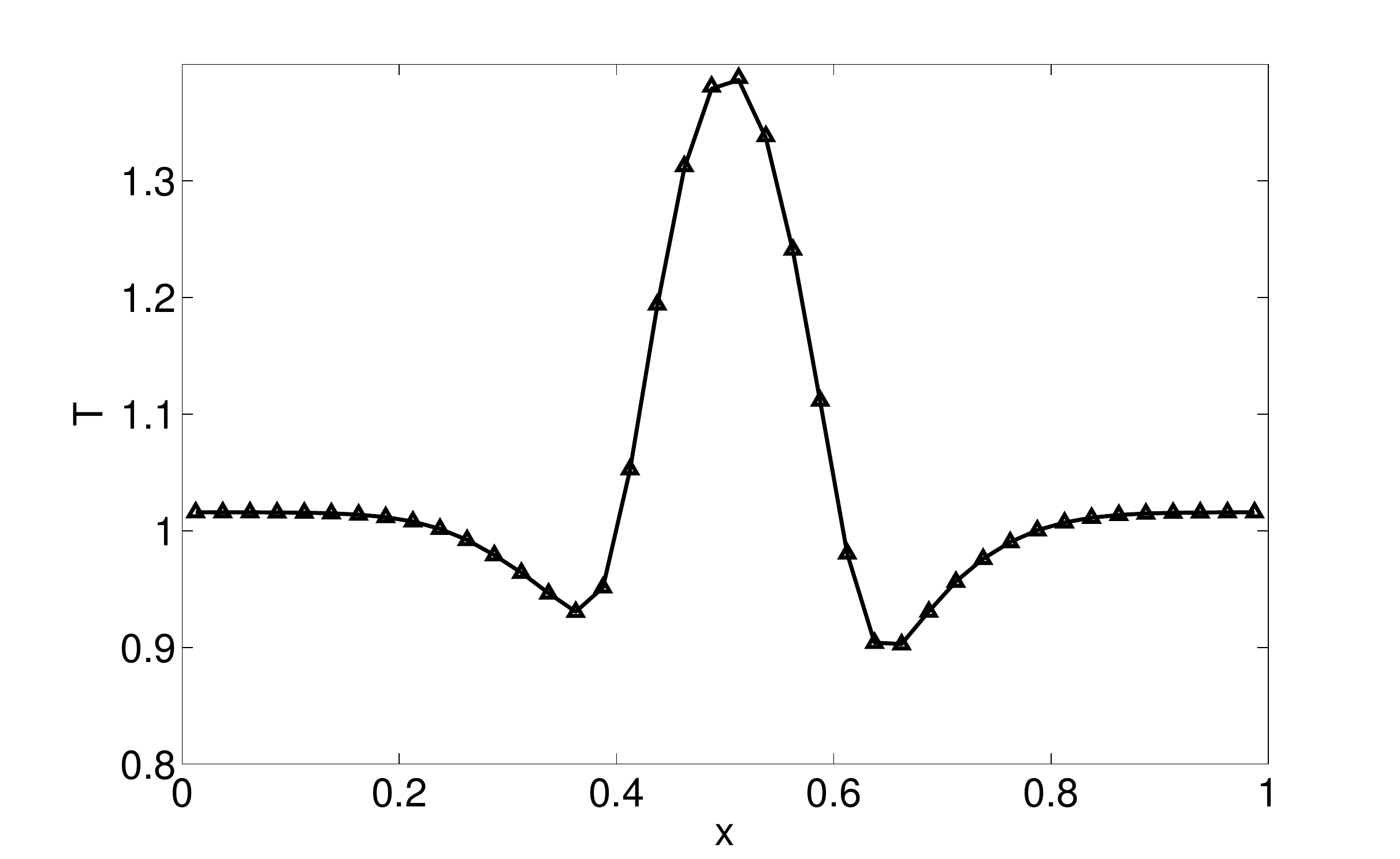}
  \end{minipage}
  \\
    \begin{minipage}[t]{0.5\linewidth}
       \centering
       \includegraphics[width=1.1\textwidth]{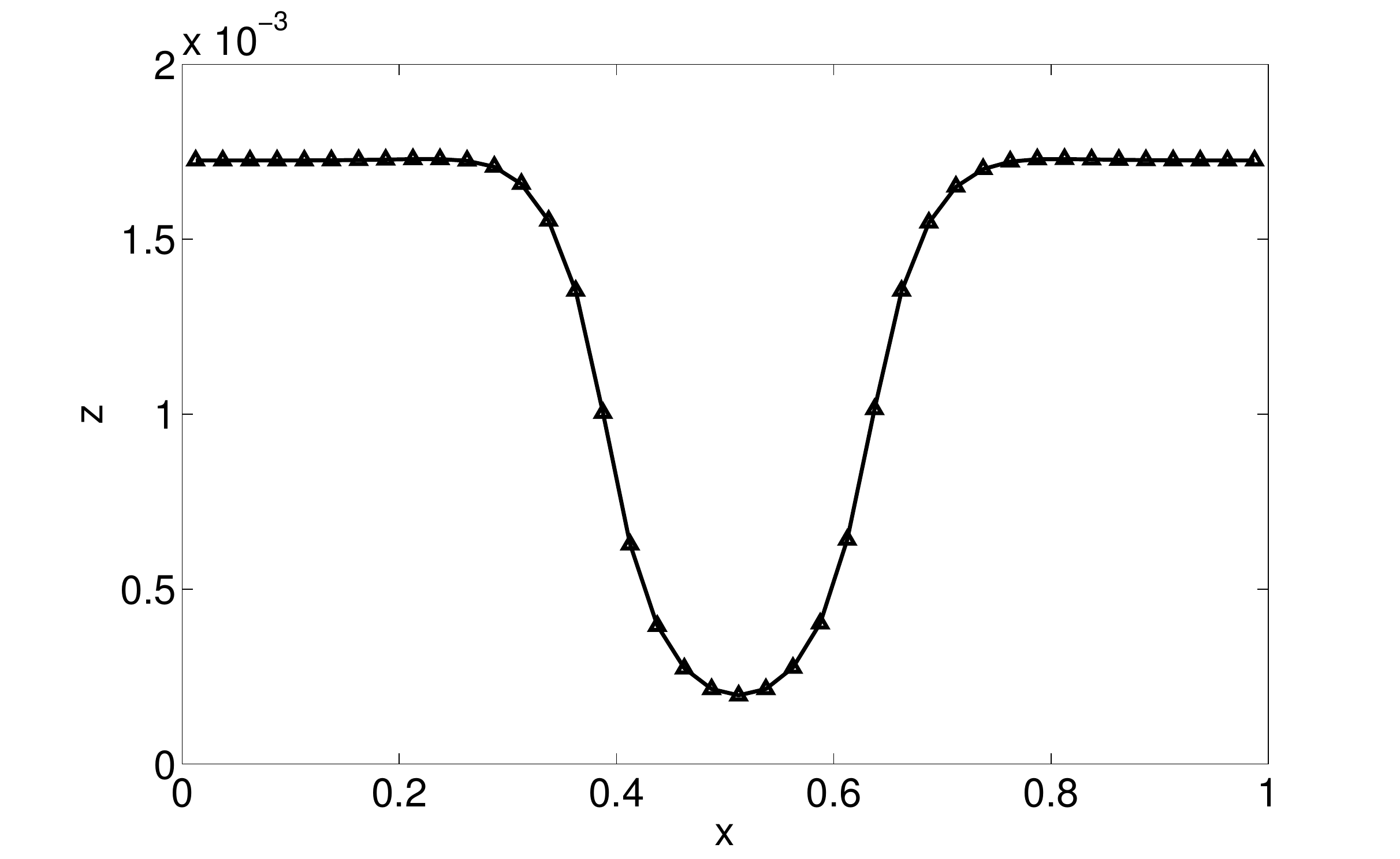}
 \end{minipage}
 \begin{minipage}[t]{0.5\linewidth}
      \centering
      \includegraphics[width=1.1\textwidth]{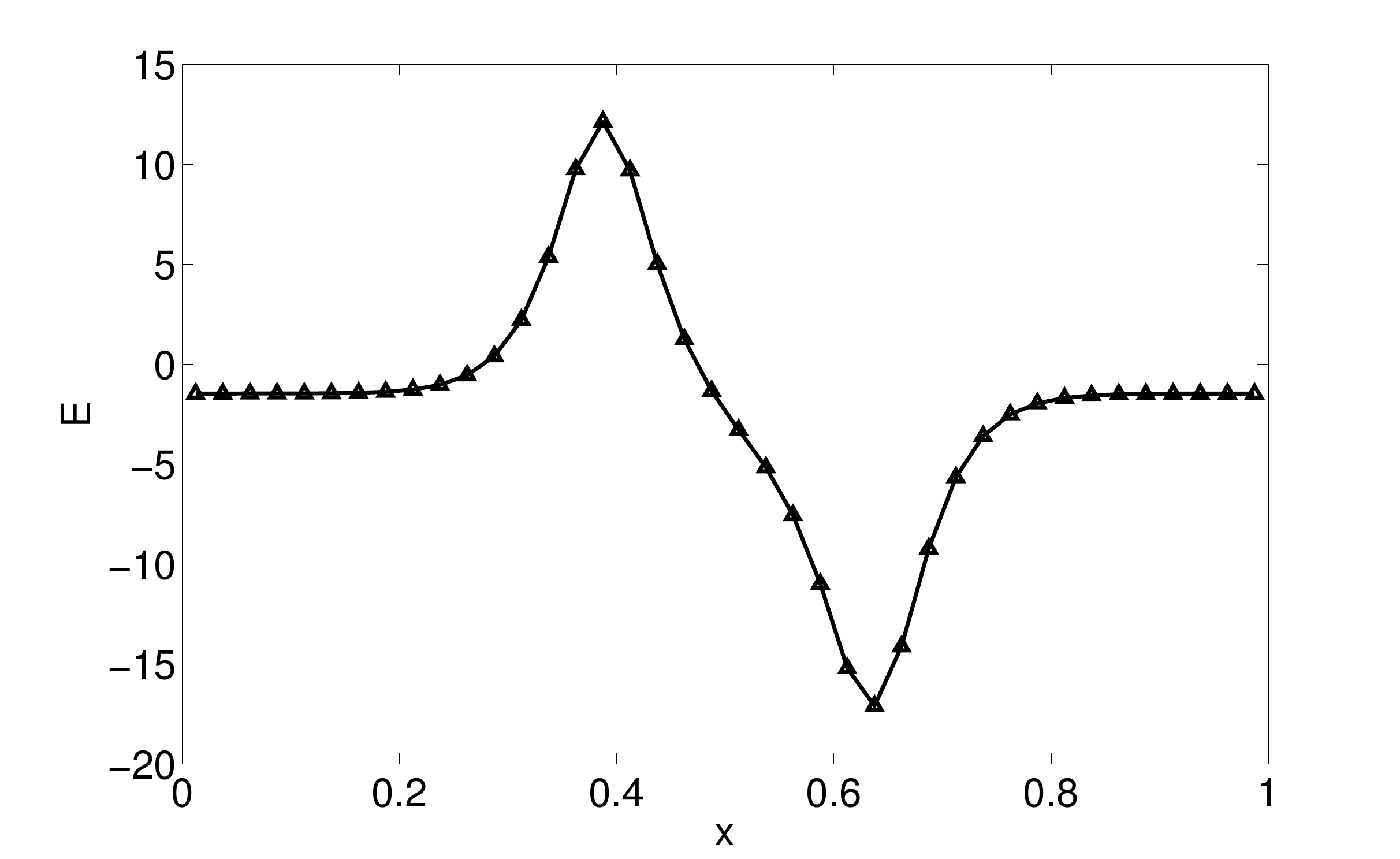}
  \end{minipage}
 \caption{Density $\rho$, velocity $u$, internal energy $\energy$, temperature $T$, fugacity $z$, and electric field $E$. `---' is the AP scheme, `$\triangle$' is the forward Euler scheme. Here $\alpha = 1$, $\eta = 0.01$, $N_x = 40$, $N_k = 32$, $\Dt = 0.2 \Dx^2$.}
 \label{fig: eta1e-2alpha1}
 \end{figure}

 \begin{figure}[!h]
  \begin{minipage}[t]{0.5\linewidth}
       \centering
       \includegraphics[width=1.1\textwidth]{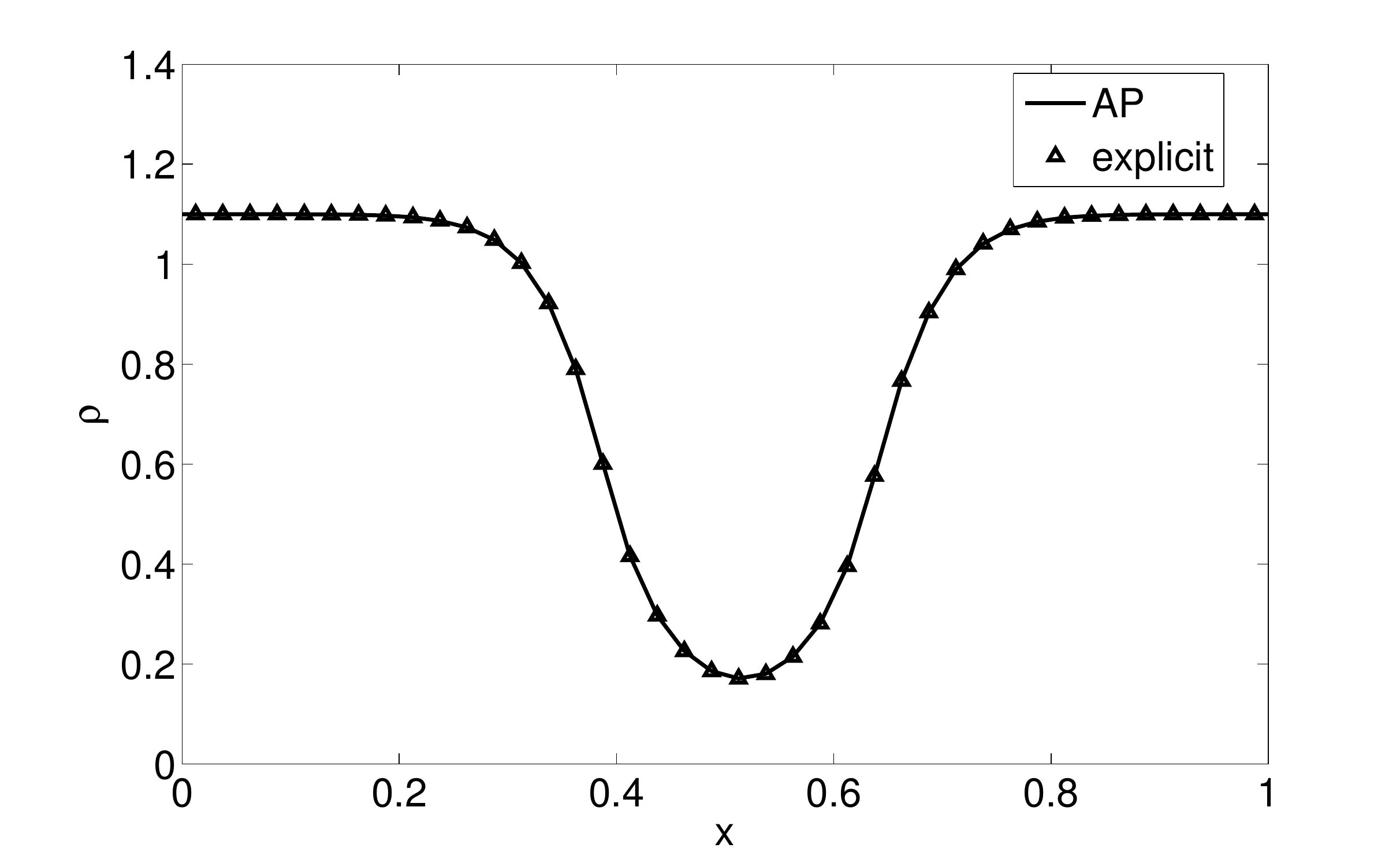}
 \end{minipage}
 \begin{minipage}[t]{0.5\linewidth}
      \centering
      \includegraphics[width=1.1\textwidth]{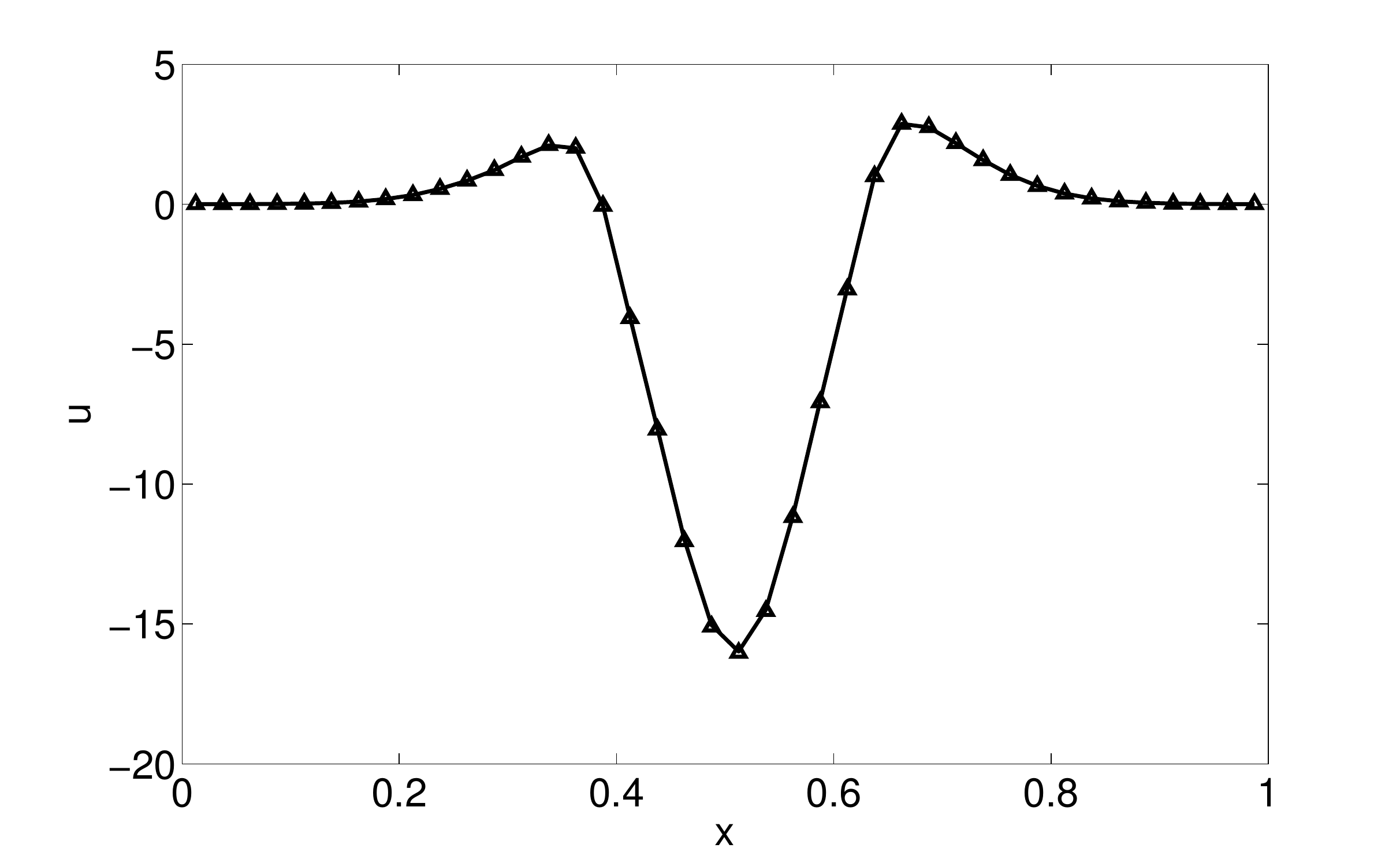}
  \end{minipage}
\\
  \begin{minipage}[t]{0.5\linewidth}
       \centering
       \includegraphics[width=1.1\textwidth]{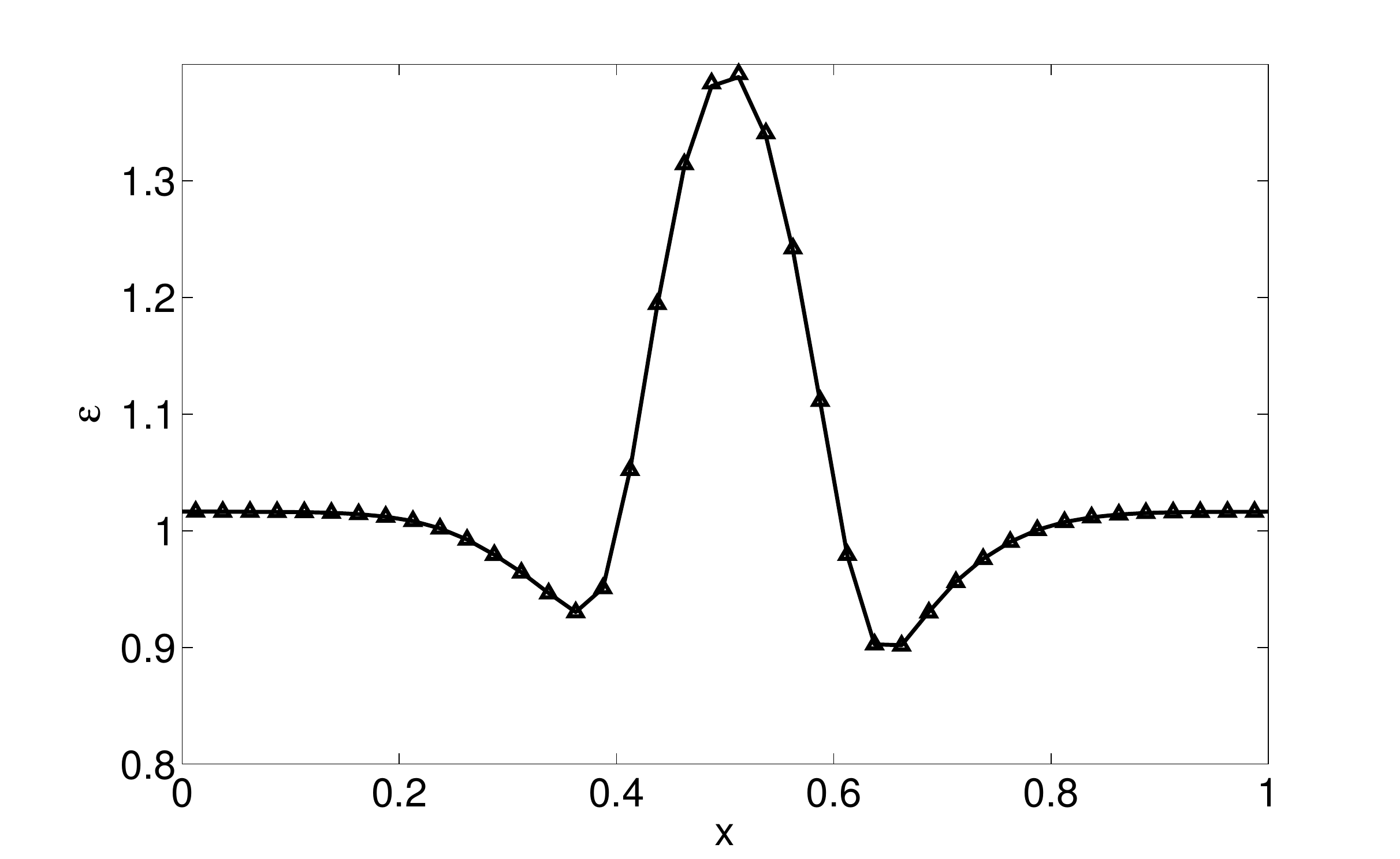}
 \end{minipage}
 \begin{minipage}[t]{0.5\linewidth}
      \centering
      \includegraphics[width=1.1\textwidth]{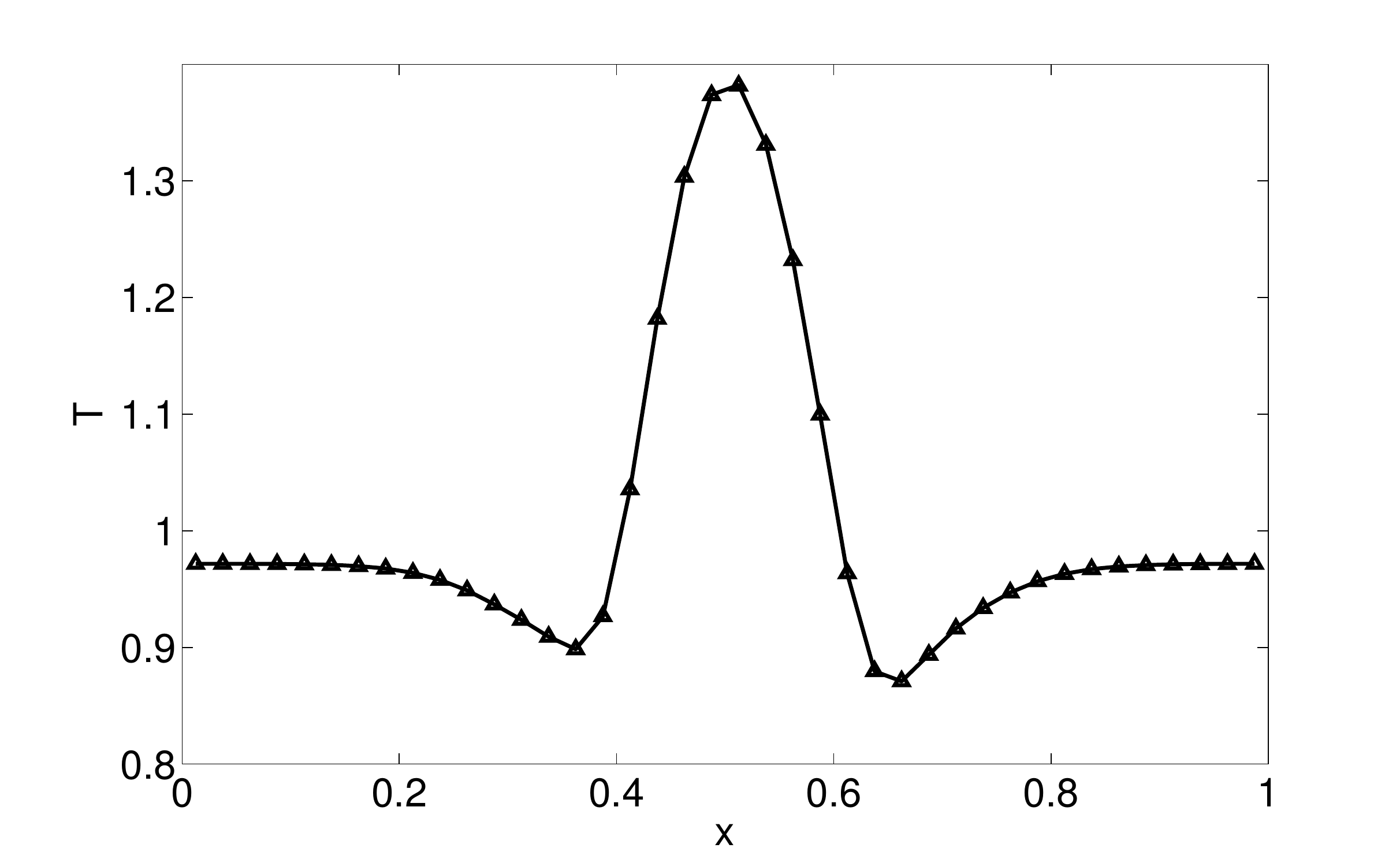}
  \end{minipage}
  \\
    \begin{minipage}[t]{0.5\linewidth}
       \centering
       \includegraphics[width=1.1\textwidth]{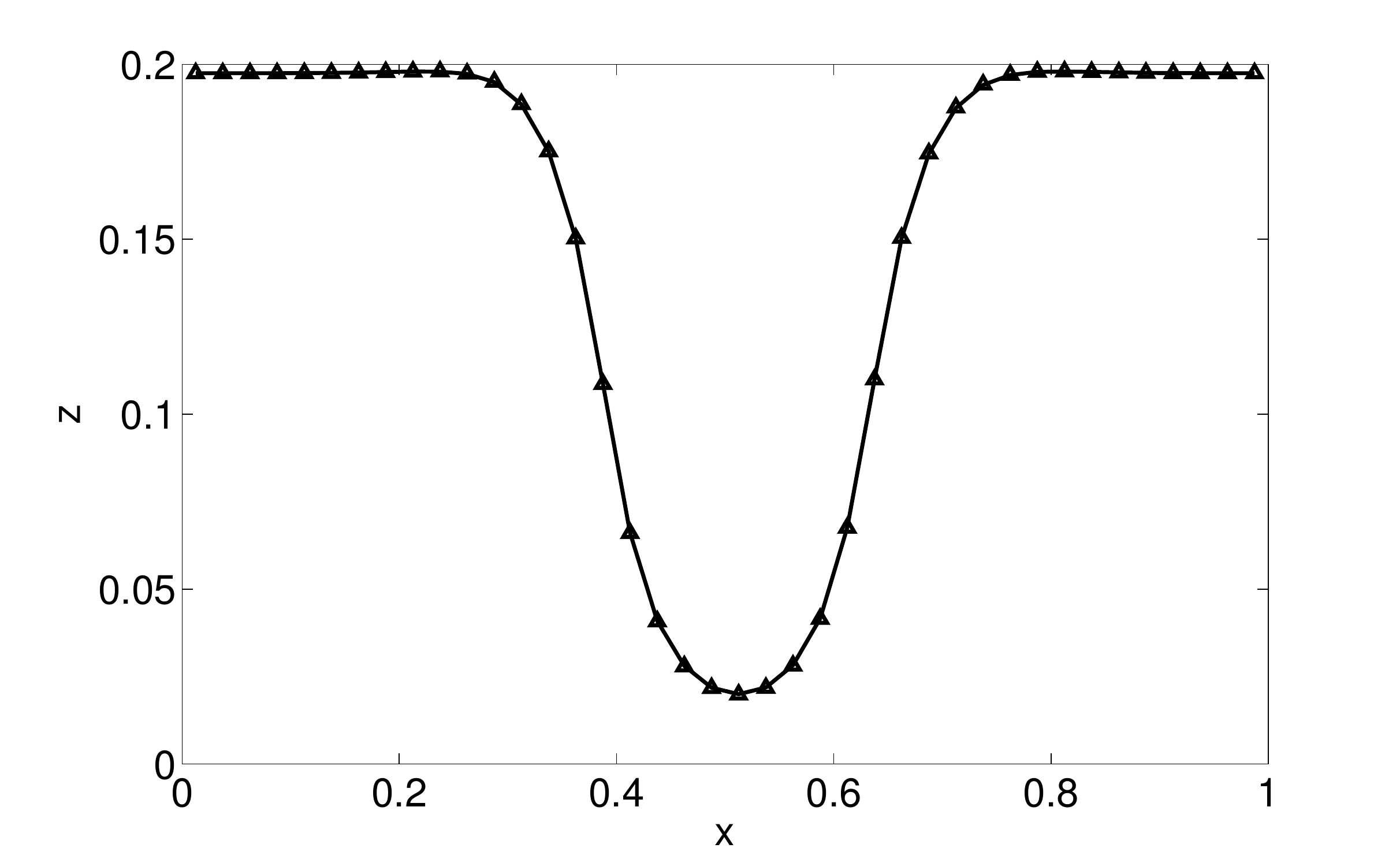}
 \end{minipage}
 \begin{minipage}[t]{0.5\linewidth}
      \centering
      \includegraphics[width=1.1\textwidth]{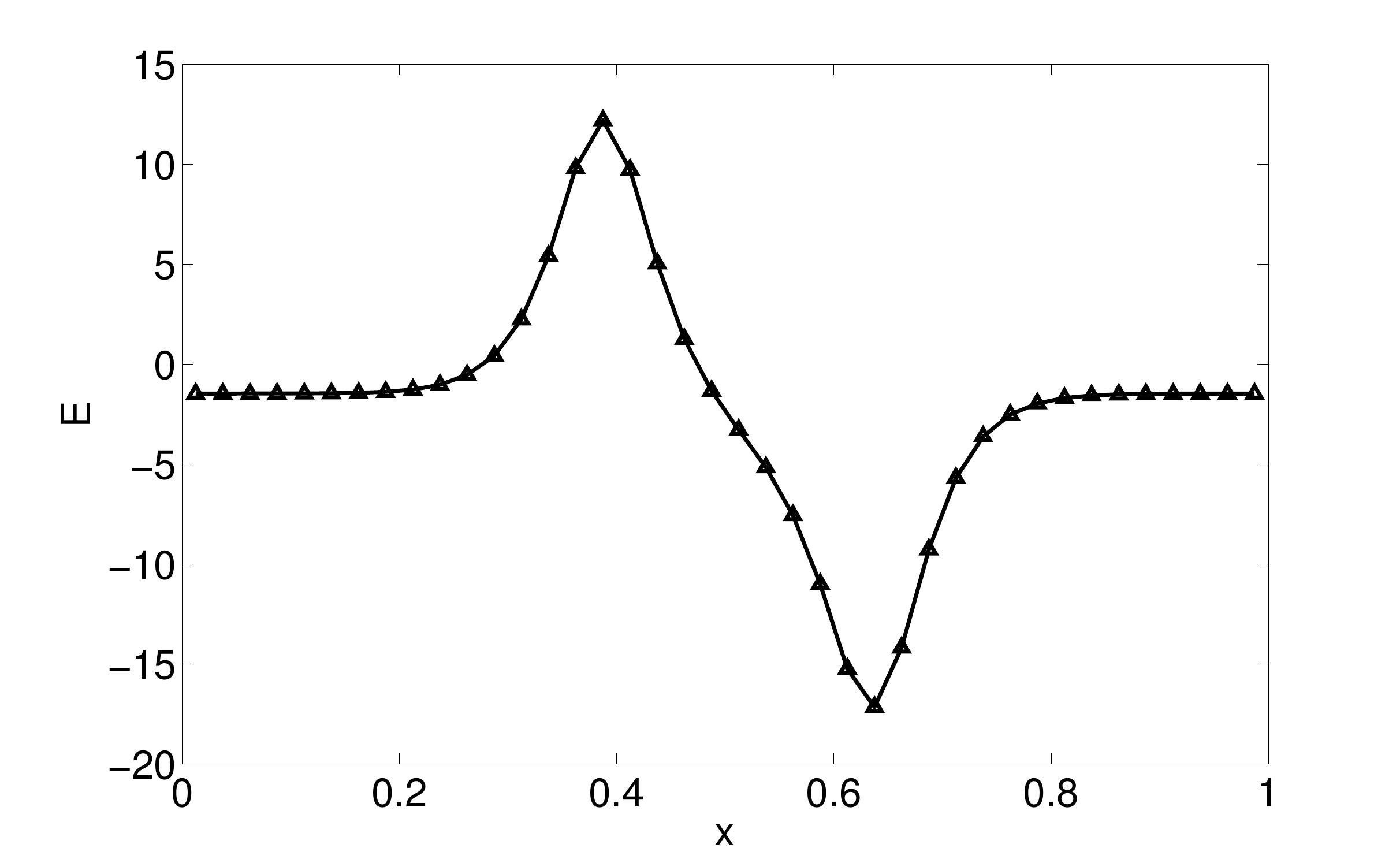}
  \end{minipage}
 \caption{Density $\rho$, velocity $u$, internal energy $\energy$, temperature $T$, fugacity $z$, and electric field $E$. `---' is the AP scheme, `$\triangle$' is the forward Euler scheme. Here $\alpha = 1$, $\eta = 1$, $N_x = 40$, $N_k = 64$, $\Dt = 0.2 \Dx^2$.}
 \label{fig: eta1alpha1}
 \end{figure}

   \begin{figure}[!h]
  \begin{minipage}[t]{0.5\linewidth}
       \centering
       \includegraphics[width=1.1\textwidth]{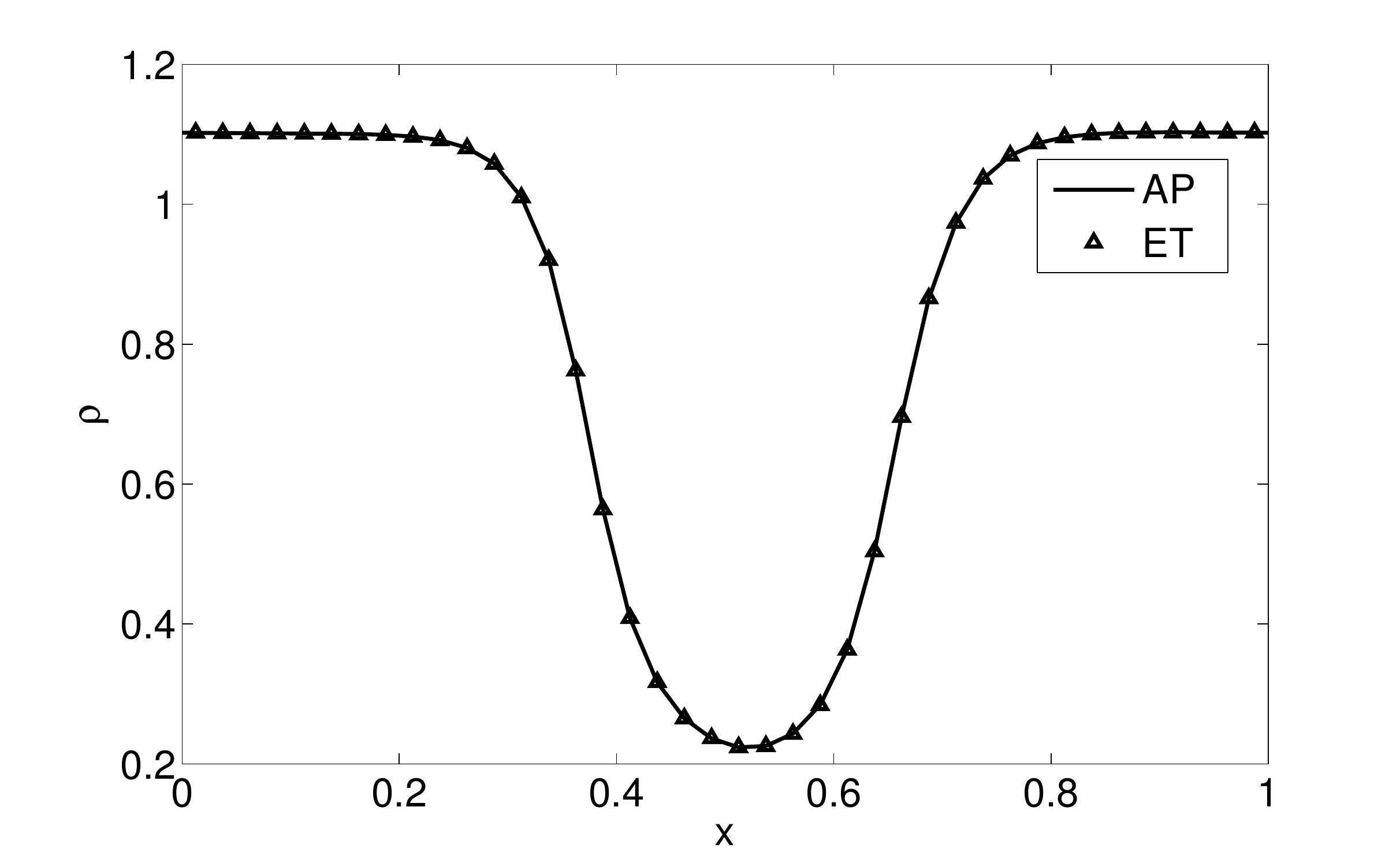}
 \end{minipage}
 \begin{minipage}[t]{0.5\linewidth}
      \centering
      \includegraphics[width=1.1\textwidth]{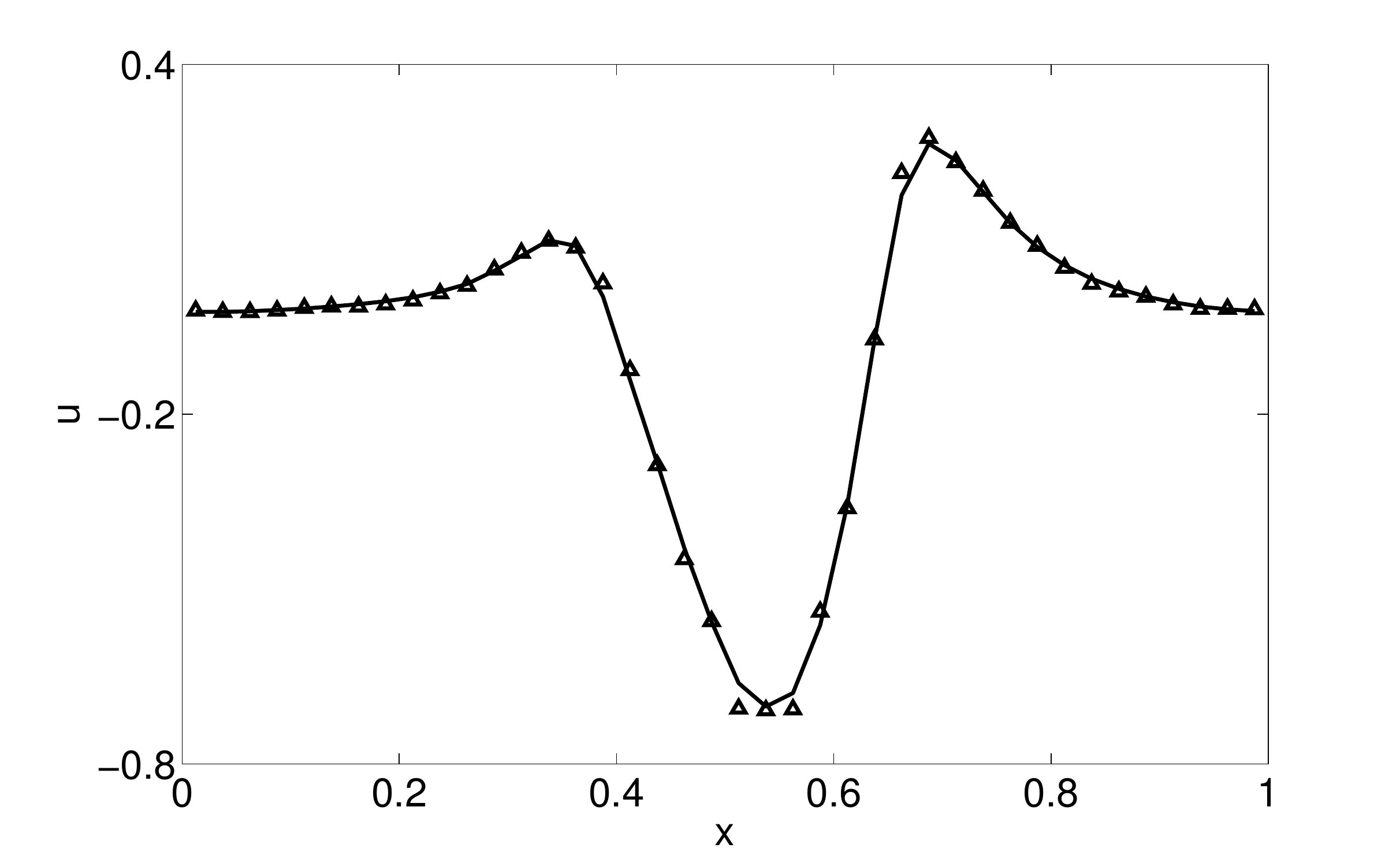}
  \end{minipage}
\\
  \begin{minipage}[t]{0.5\linewidth}
       \centering
       \includegraphics[width=1.1\textwidth]{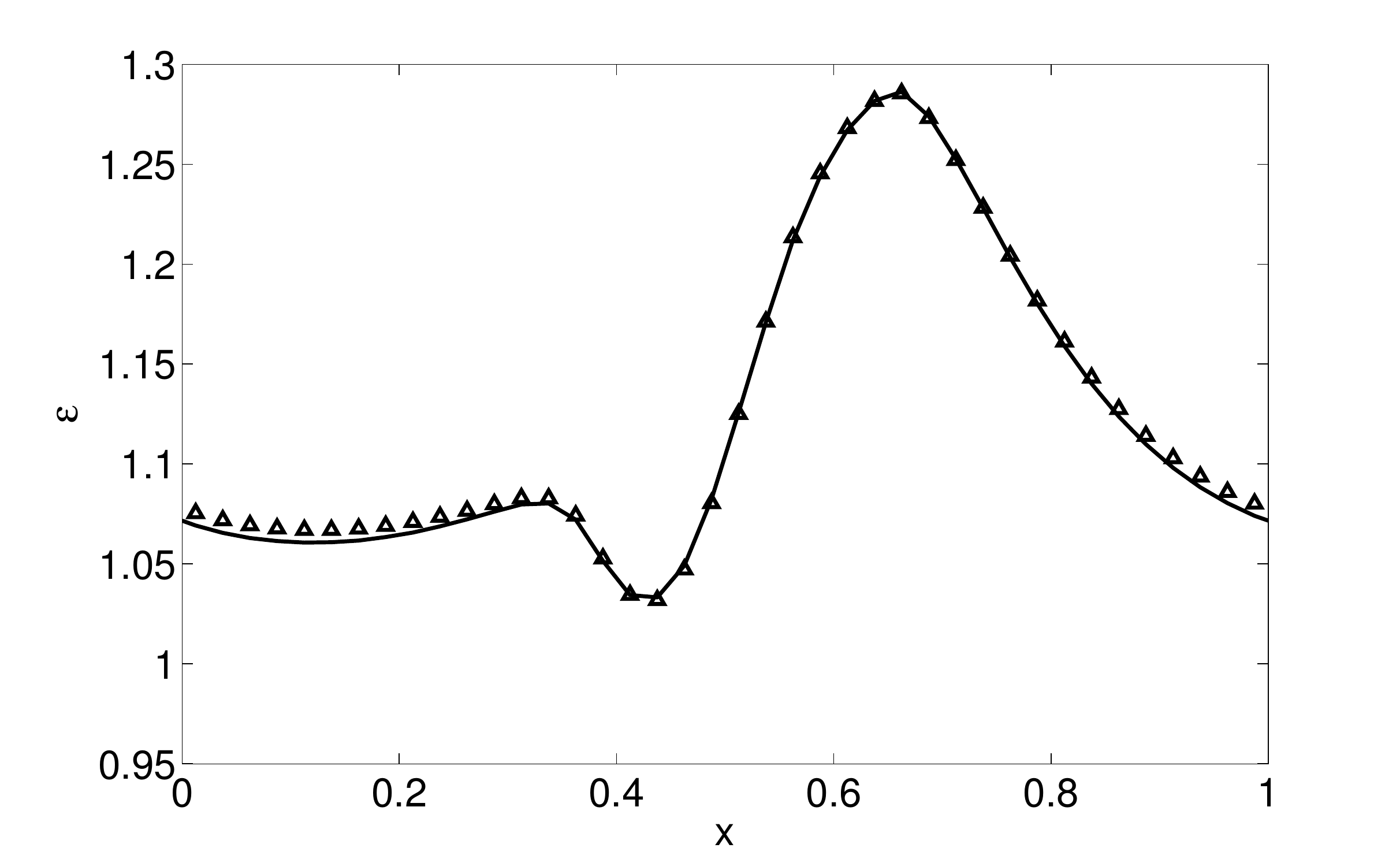}
 \end{minipage}
 \begin{minipage}[t]{0.5\linewidth}
      \centering
      \includegraphics[width=1.1\textwidth]{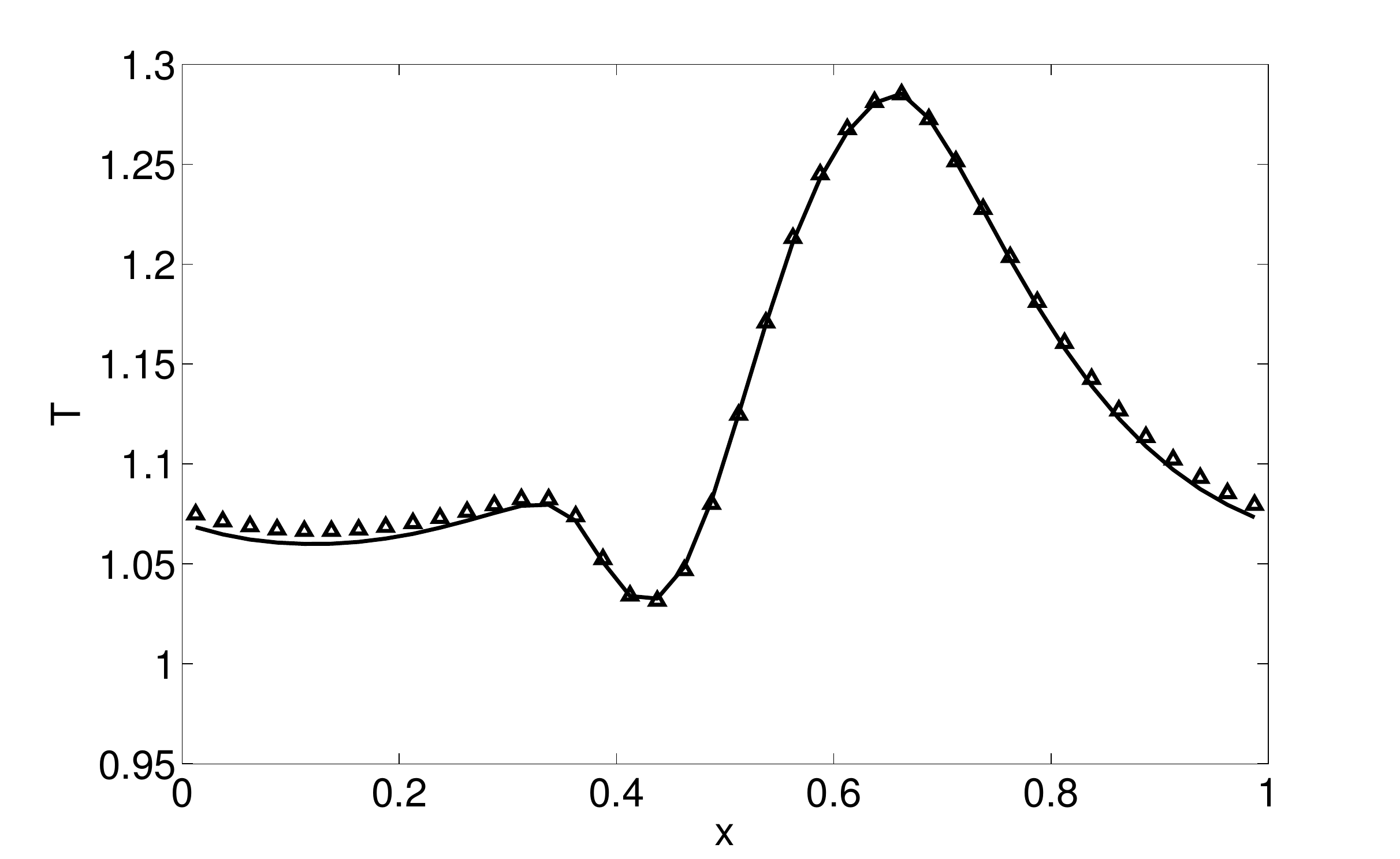}
  \end{minipage}
  \\
    \begin{minipage}[t]{0.5\linewidth}
       \centering
       \includegraphics[width=1.1\textwidth]{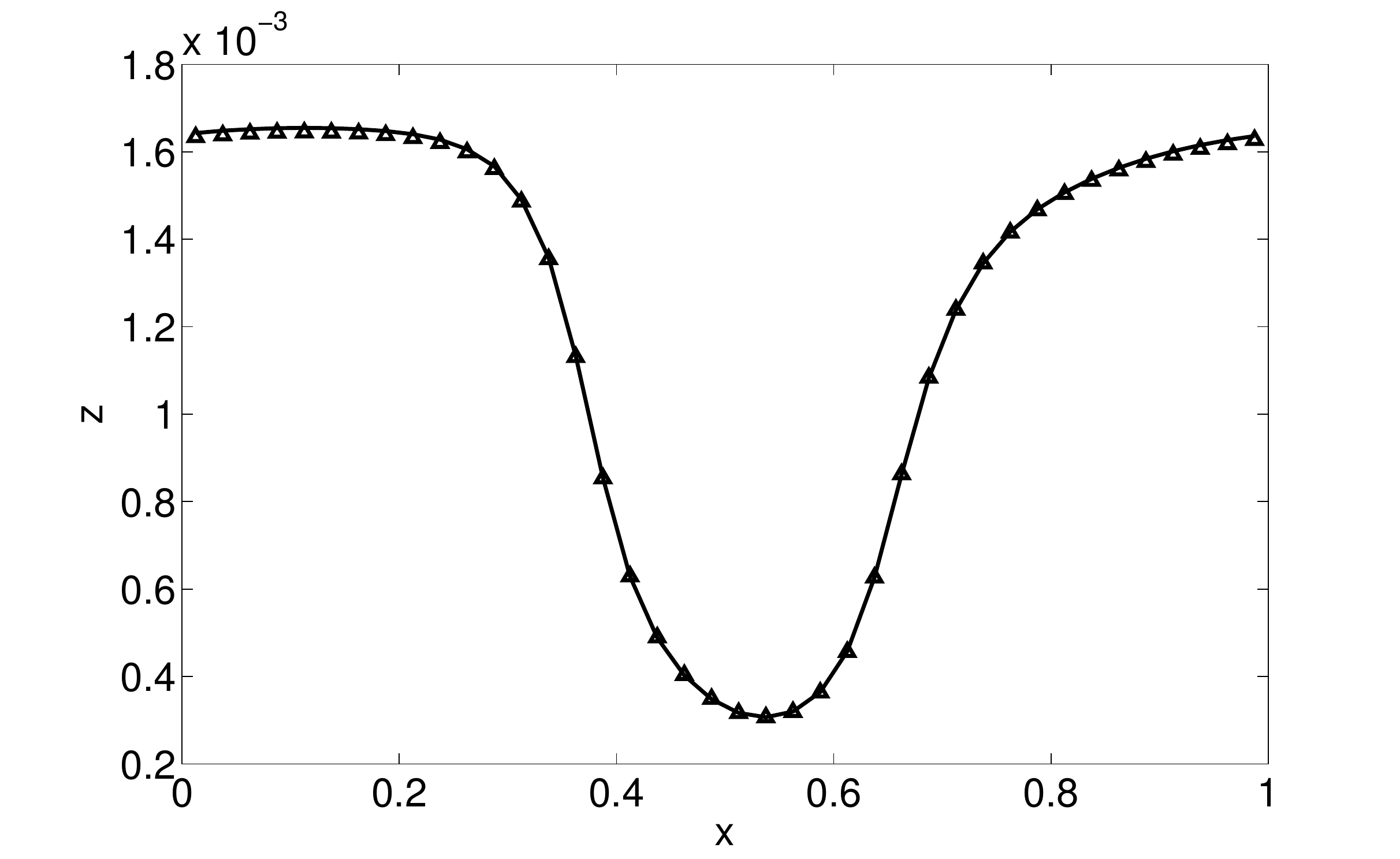}
 \end{minipage}
 \begin{minipage}[t]{0.5\linewidth}
      \centering
      \includegraphics[width=1.1\textwidth]{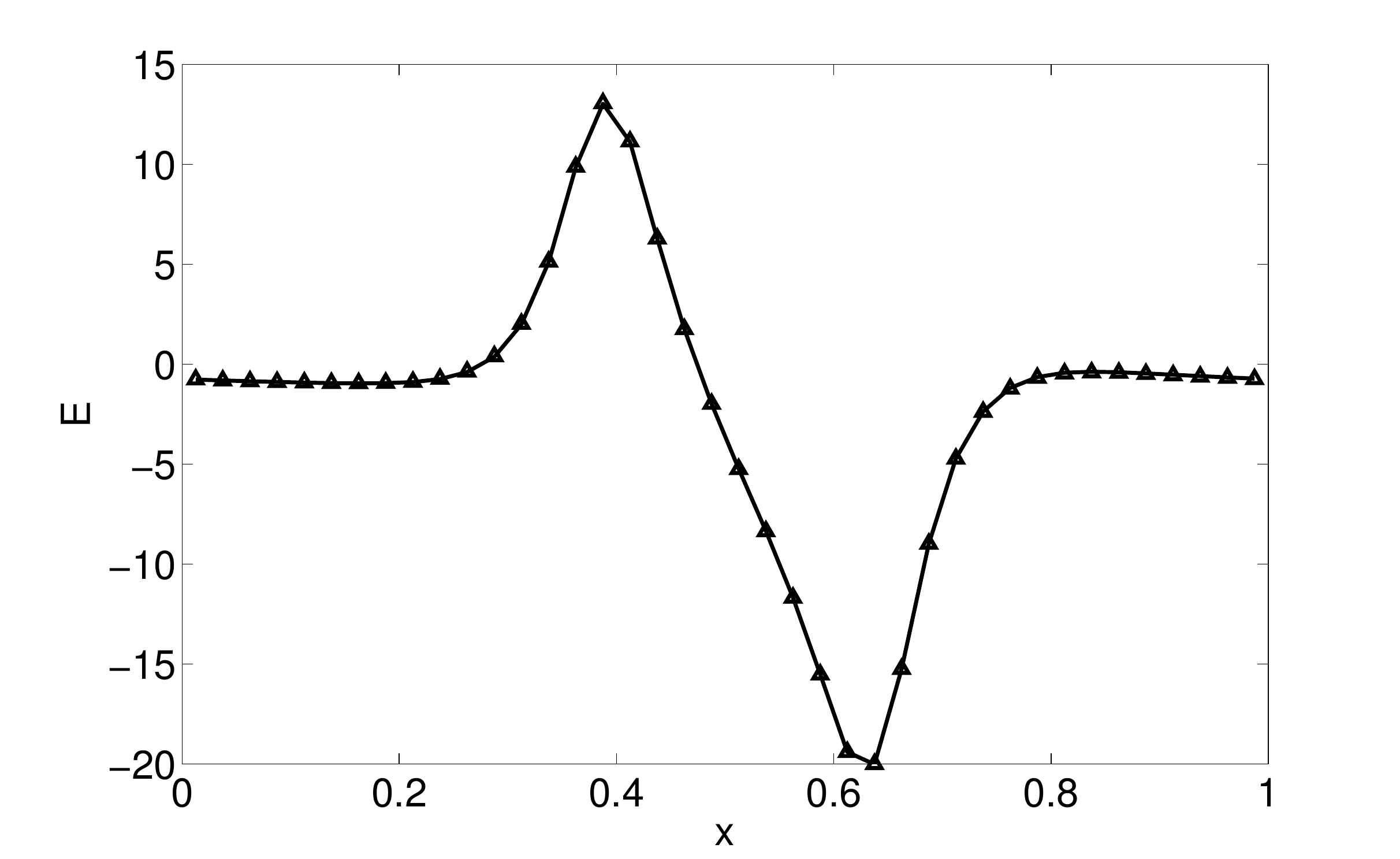}
  \end{minipage}
 \caption{Density $\rho$, velocity $u$, internal energy $\energy$, temperature $T$, fugacity $z$, and electric field $E$. `---' is the AP scheme, `$\triangle$' is the kinetic scheme (for the ET system). Here $\alpha = 1e-3$, $\eta = 0.01$, $N_x = 40$, $N_k = 32$, $\Dt = 0.2 \Dx^2$.}
 \label{fig: eta1e-2alpha1e-3}
 \end{figure}
 
 \begin{figure}[!h]
  \begin{minipage}[t]{0.5\linewidth}
       \centering
       \includegraphics[width=1.1\textwidth]{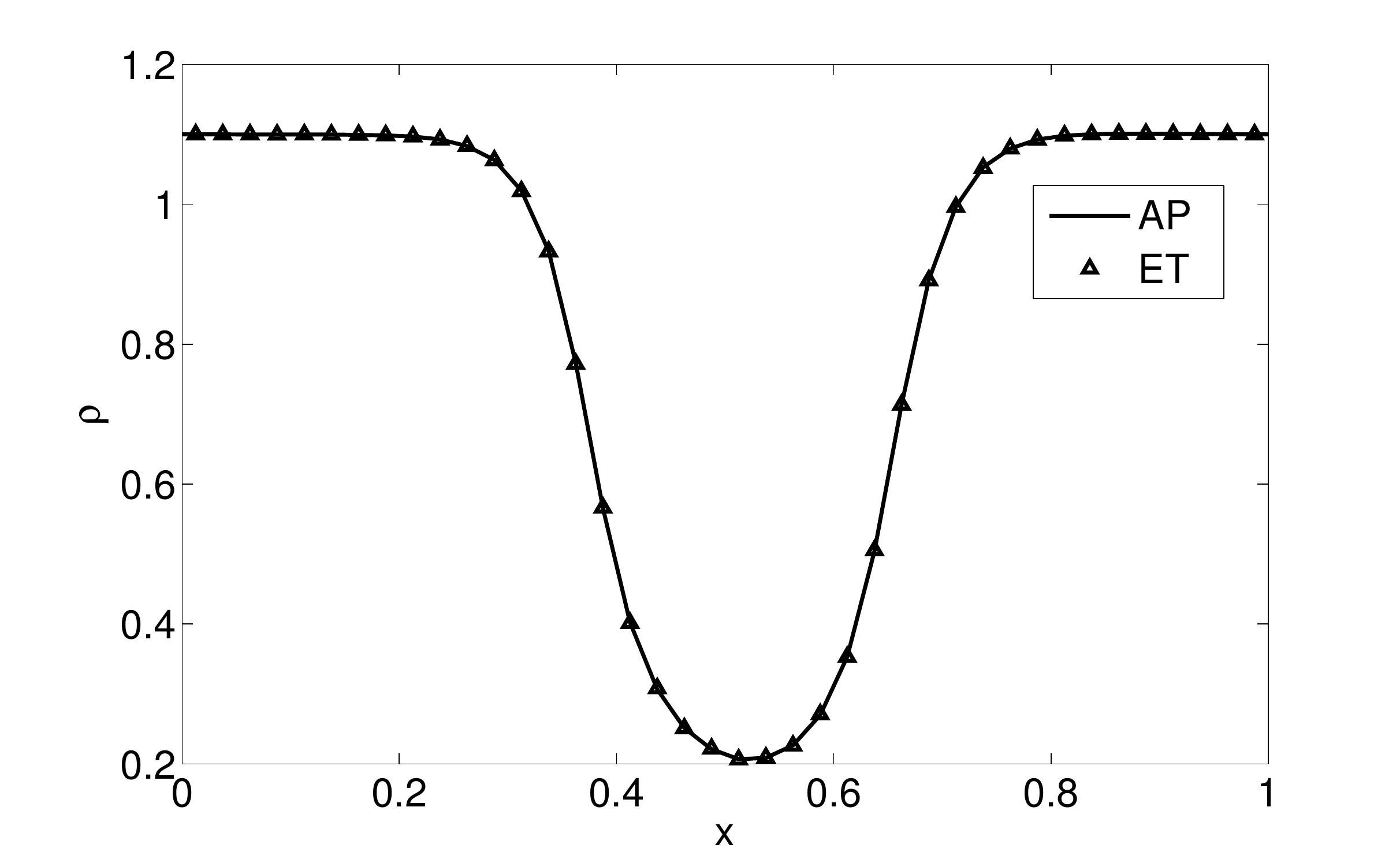}
 \end{minipage}
 \begin{minipage}[t]{0.5\linewidth}
      \centering
      \includegraphics[width=1.1\textwidth]{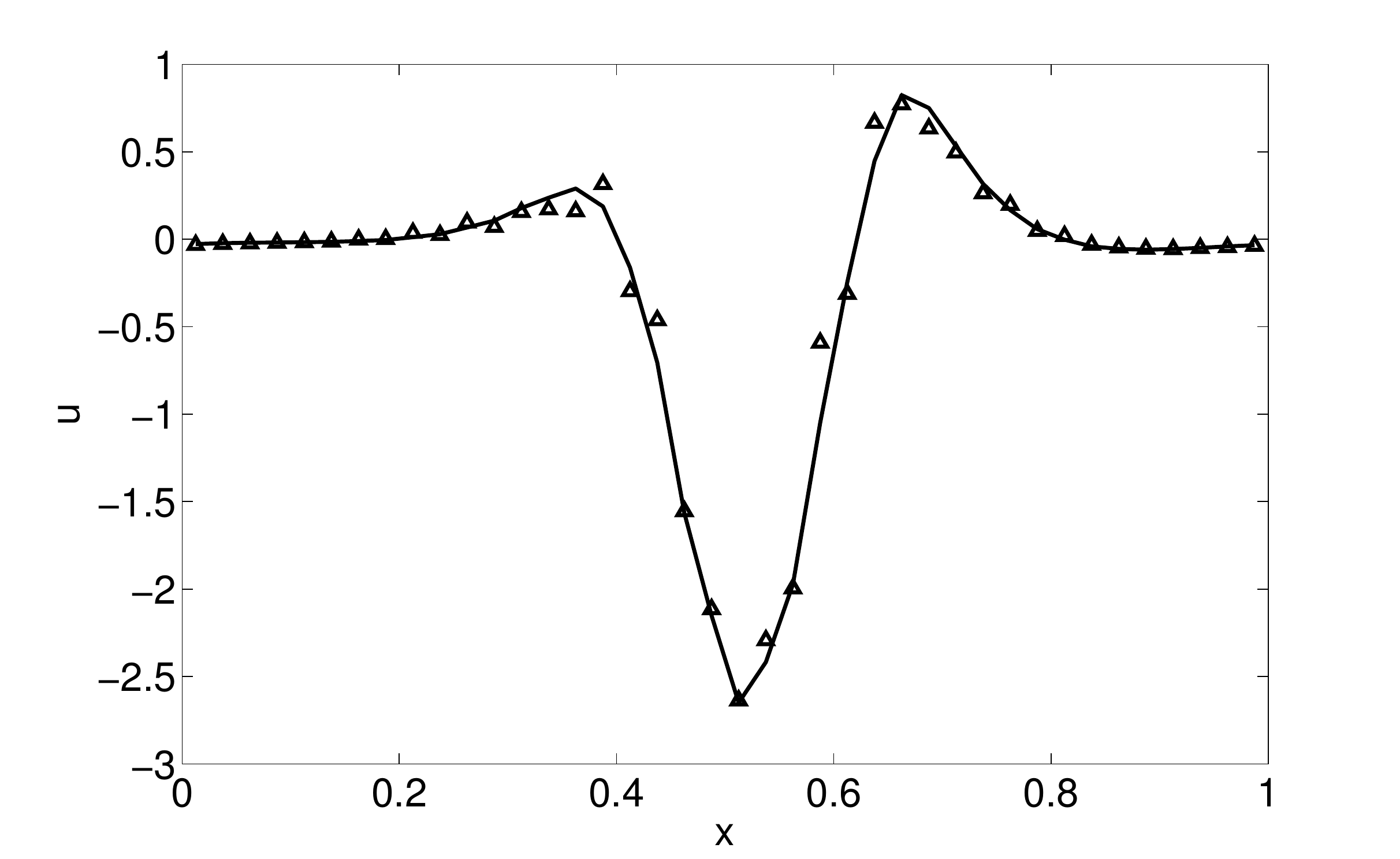}
  \end{minipage}
\\
  \begin{minipage}[t]{0.5\linewidth}
       \centering
       \includegraphics[width=1.1\textwidth]{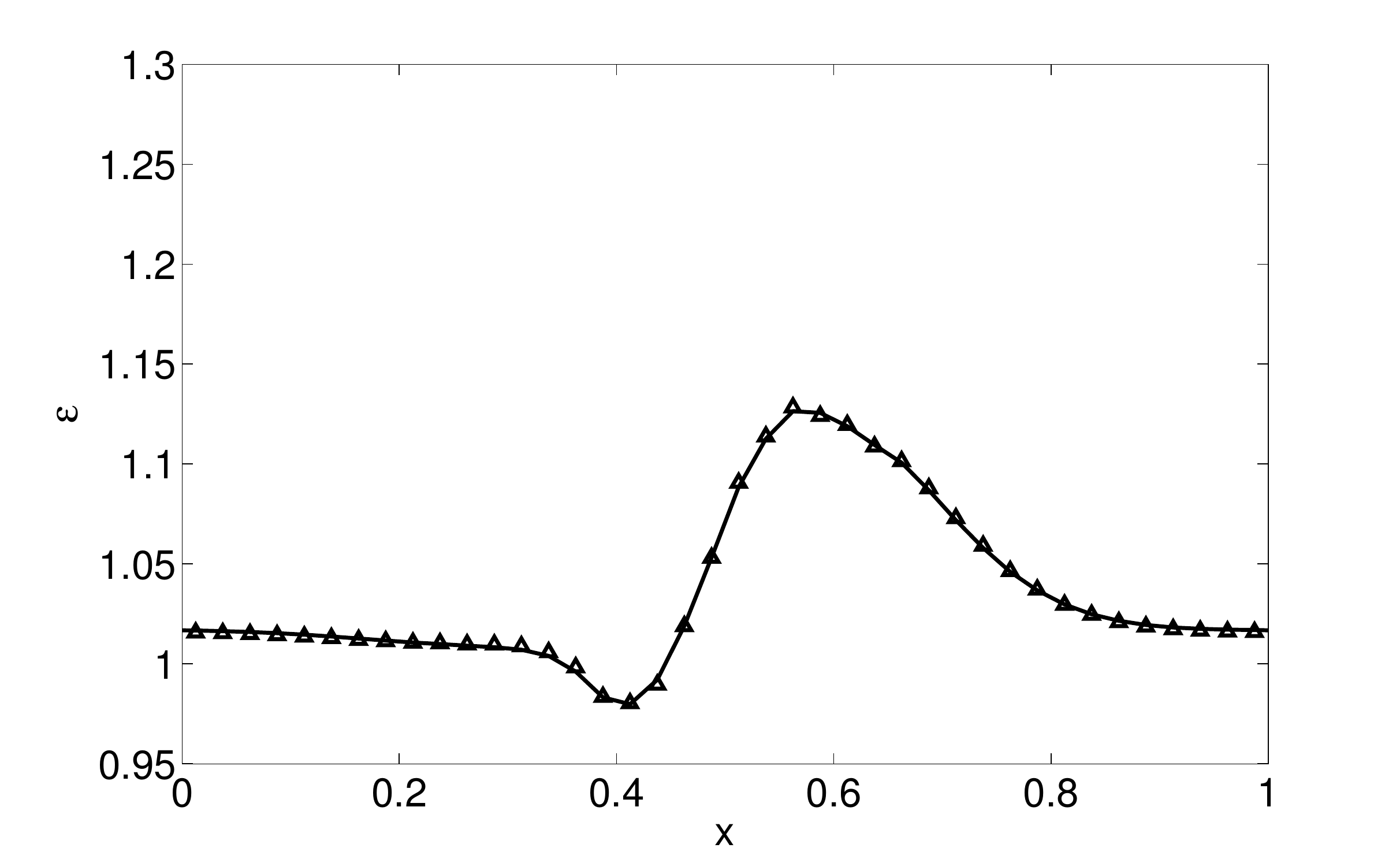}
 \end{minipage}
 \begin{minipage}[t]{0.5\linewidth}
      \centering
      \includegraphics[width=1.1\textwidth]{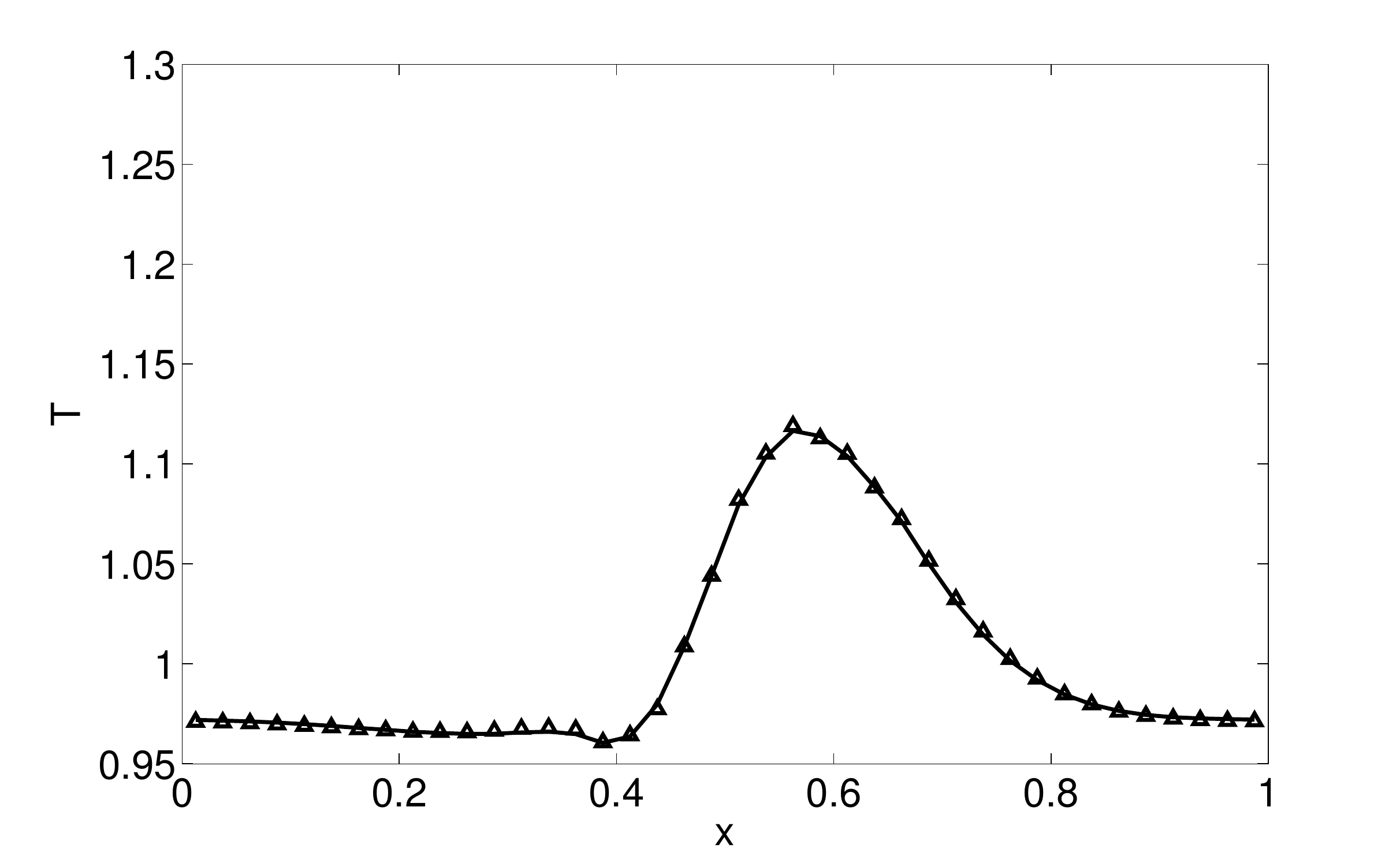}
  \end{minipage}
  \\
    \begin{minipage}[t]{0.5\linewidth}
       \centering
       \includegraphics[width=1.1\textwidth]{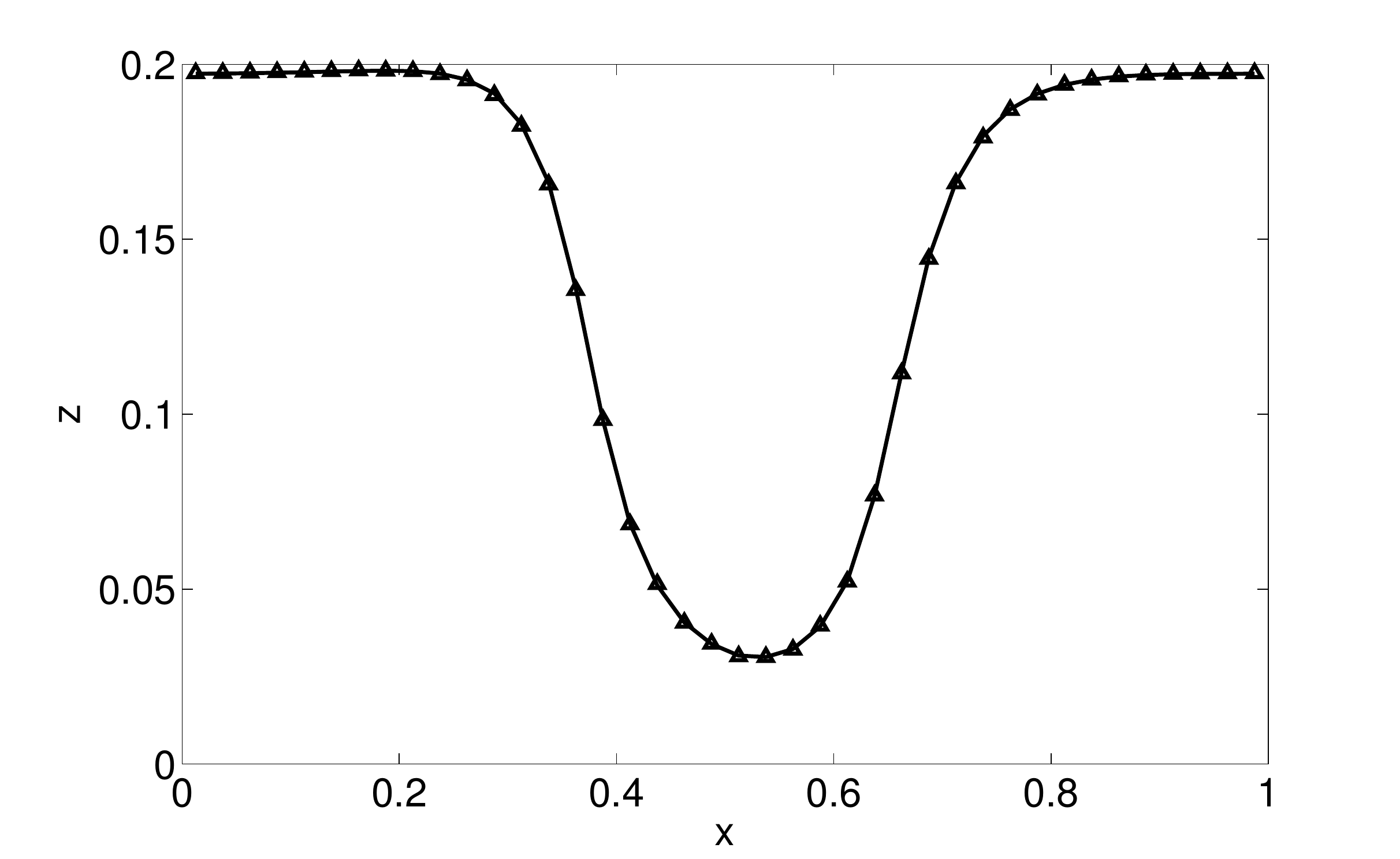}
 \end{minipage}
 \begin{minipage}[t]{0.5\linewidth}
      \centering
      \includegraphics[width=1.1\textwidth]{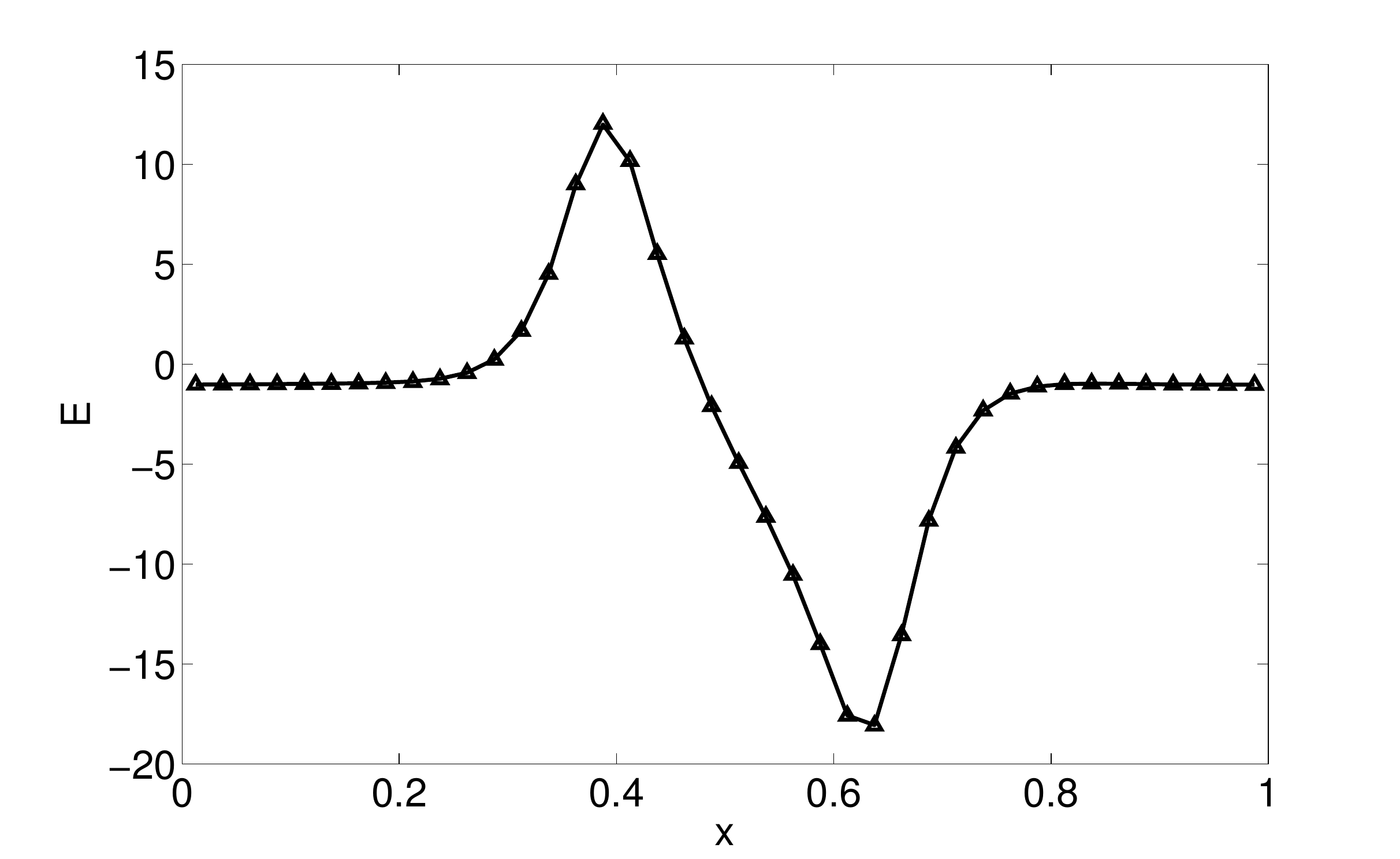}
  \end{minipage}
 \caption{Density $\rho$, velocity $u$, internal energy $\energy$, temperature $T$, fugacity $z$, and electric field $E$. `---' is the AP scheme, `$\triangle$' is the kinetic scheme (for the ET system). Here $\alpha = 1e-3$, $\eta = 1$, $N_x = 40$, $N_k = 64$, $\Dt = 0.2 \Dx^2$.}
  \label{fig: eta1alpha1e-3}
 \end{figure}

\section{Conclusion}
We constructed an asymptotic preserving scheme for a multiscale semiconductor Boltzmann equation (coupled with Poisson) that in the diffusive regime captures the energy-transport limit. Because of the two different scales appearing in the collision operator, the previous AP schemes for the drift-diffusion limit does not work well. A key ingredient in our scheme is to set a suitable threshold for the stiffer collision term such that once this threshold is crossed, the less stiff collision begins to dominate. In this way, the convergence of the numerical solution to the local equilibrium mimics the Hilbert expansion at the continuous level. We analyzed this asymptotic behavior using a simplified BGK model. A new fast spectral method for the elastic collision operator was also introduced. Several numerical results confirmed the uniform stability of our scheme with respect to the mean free path, from kinetic regime to energy-transport regime. In particular, a 1-D $n^+$--$n$--$n^+$ ballistic silicon diode was simulated to verify its efficiency in both degenerate and non degenerate cases. 

\section*{Acknowledgments}
This project was initiated during the authorsÕ participation at the KI-Net Conference ``Quantum Systems" held by CSCAMM, University of Maryland, May 2013. Both authors acknowledge the generous support from the institute. J.H. thanks Prof. Shi Jin for initially pointing out the ET model, and Prof. Irene Gamba and Prof. Lexing Ying for helpful discussion and support. L.W. thanks Prof. Robert Krasny for fruitful discussion on modeling of semiconductors. 

\bibliographystyle{siam}
\bibliography{semiB2reference}

\end{document}